\documentclass[a4paper, 11pt, oneside, reqno, svgnames]{amsart}
\usepackage[a4paper]{geometry}
\usepackage[utf8]{inputenc}
\usepackage[english]{babel}
\usepackage[T1]{fontenc}
\usepackage{textcomp}

\usepackage{txfonts}

\usepackage{apptools}

\usepackage[mathscr]{euscript}

\usepackage{tikz-cd}
\usepackage{quiver}
\usepackage{todonotes}

\usepackage{amsmath}
\usepackage{amsfonts}
\usepackage{amsthm}
\usepackage{latexsym}
\usepackage{amssymb}
\usepackage{xcolor}
\definecolor{blue(munsell)}{rgb}{0.0, 0.5, 0.69}

\makeatletter
\def\l@subsection{\@tocline{2}{0pt}{2pc}{5pc}{}}
\makeatother

\newenvironment{psmallmatrix}{\left(\begin{smallmatrix}}{\end{smallmatrix}\right)}

\numberwithin{equation}{section}

\DeclareMathOperator{\Hom}{Hom}

\DeclareMathOperator{\Ob}{Ob}

\DeclareMathOperator{\Mod}{Mod}
\DeclareMathOperator{\modfp}{mod}

\DeclareMathOperator{\Fundg}{Fun_{dg}}
\DeclareMathOperator{\hproj}{h-proj}
\DeclareMathOperator{\pretr}{pretr}
\DeclareMathOperator{\cone}{C}

\DeclareMathOperator{\Tw}{Tw}

\DeclareMathOperator{\Proj}{Proj}
\DeclareMathOperator{\DGProj}{DGProj}

\DeclareMathOperator{\compdg}{dgm}

\DeclareMathOperator{\dercomp}{\mathsf{D}}
\DeclareMathOperator{\dercompdg}{\mathsf{D}_{\mathrm{dg}}}

\DeclareMathOperator{\RHom}{\mathbb R\!\Hom}
\DeclareMathOperator{\lotimes}{\overset{\mathbb L}{\otimes}}
\newcommand{\cat}{\mathscr}
\newcommand{\varcat}{\mathbf}
\newcommand{\opp}[1]{{#1}^{\mathrm{op}}}
\newcommand{\kat}{\mathsf}

\newcommand{\Hqe}{\kat{Hqe}}

\newcommand{\holim}{\mathrm{ho}\!\varprojlim}
\newcommand{\hocolim}{\mathrm{ho}\!\varinjlim}

\newcommand{\basering}[1]{\mathbf{#1}}

\newtheorem{theorem}{Theorem}[section]
\newtheorem*{theorem*}{Theorem}
\newtheorem{proposition}[theorem]{Proposition}
\newtheorem{corollary}[theorem]{Corollary}

\newtheorem{lemma}[theorem]{Lemma}

\theoremstyle{remark}
\newtheorem{remark}[theorem]{Remark}
\newtheorem{example}[theorem]{Example}

\theoremstyle{definition}
\newtheorem{definition}[theorem]{Definition}

\newtheorem{construction}[theorem]{Construction}

\AtAppendix{\counterwithin{theorem}{section}}

\title{T-structures on unbounded twisted complexes}

\author{Francesco Genovese} 
\address[Francesco Genovese]{Univerzita Karlova, Matematicko-fyzik\'{a}ln\'{i} fakulta, Katedra Algebry, Sokolovsk\'{a} 49/83, 186 75 Praha 8, \v{C}esk\'a republika.}
\email{genovese@karlin.mff.cuni.cz}

\thanks{The author acknowledges the support of the Czech Science Foundation grant [GA \v{C}R 20-13778S]}

\begin{document}

\begin{abstract}
This paper is a sequel to \emph{T-structures and twisted complexes on derived injectives} by the same author with W. Lowen and M. Van den Bergh. We define a dg-category of unbounded twisted complexes on a dg-category, which is particularly interesting in the case of dg-categories of derived injectives or derived projectives associated to a t-structure. On such unbounded twisted complexes we define a natural ``injective'' and dually a ``projective'' t-structure. This is intended as a direct generalization of the homotopy categories of injective or projective objects of an abelian category. 
\end{abstract}

\maketitle

\tableofcontents

\section*{Introduction}
If $\mathfrak A$ is an abelian category (typically, the category of modules over a ring or the category of quasi-coherent sheaves on a scheme), we can define its \emph{derived category} $\dercomp(\mathfrak A)$ by taking complexes of objects of $\mathfrak A$ and formally inverting quasi-isomorphisms.

If $\mathfrak A$ has enough injectives, it is well known that the bounded above derived category $\dercomp^+(\mathfrak A)$ can be described as the homotopy category of bounded below complexes of injective objects $\mathsf{K}^+(\operatorname{Inj}(\mathfrak A))$:
\begin{equation} \label{equation:hoinjectives_plus}
\dercomp^+(\mathfrak A) \cong \mathsf{K}^+(\operatorname{Inj}(\mathfrak A)).
\end{equation}
We have a dual analogous result if $\mathfrak A$ has enough projectives:
\begin{equation} \label{equation:hoprojectives_minus}
\dercomp^-(\mathfrak A) \cong \mathsf{K}^-(\operatorname{Proj}(\mathfrak A)).
\end{equation}

More recently, the (unbounded) \emph{homotopy category of injectives} and the \emph{homotopy category of projectives} $\mathsf{K}(\operatorname{Inj}(\mathfrak A))$ and $\mathsf{K}(\operatorname{Proj}(\mathfrak A))$ have been studied, in particular when $\mathfrak A$ is the category of modules over some suitable ring. Basic references are \cite{krause-stable} \cite{iyengar-krause-acyclicity} \cite{jorgensen-projective} \cite{neeman-homotopyflat} \cite{murfet-mockproj} \cite{neeman-homotopyinjectives} \cite{stovicek-purity}. An example of a result achieved by those investigations is as follows: if $R$ is a Noetherian commutative ring admitting a dualizing complex, we can interpret \emph{Grothendieck duality} as an equivalence of triangulated categories
\begin{equation} \label{equation:grothendieckduality}
\mathsf{K}(\operatorname{Proj}(R)) \cong \mathsf{K}(\operatorname{Inj}(R)),
\end{equation}
cf. \cite[Theorem 4.2]{iyengar-krause-acyclicity}. Both categories are compactly generated (see \cite[Proposition 2.3]{krause-stable} and \cite{jorgensen-projective}), and the above equivalence restricts to an equivalence between the compact objects:
\begin{equation}
    \opp{\dercomp^{\mathrm b}(\modfp(R))} \cong \dercomp^{\mathrm b}(\modfp(R)),
\end{equation}
where $\modfp(R)$ is the category of finitely presented $R$-modules.

We would like to set the above discussion in a broader framework. It is clear from recent work \cite{genovese-lowen-vdb-dginj} \cite{genovese-ramos-gabrielpopescu} that the ``correct'' way of generalizing results and constructions of abelian categories (and their derived categories) is to consider \emph{(enhanced) triangulated categories endowed with t-structures}\footnote{Or, à la Lurie, \emph{prestable $\infty$-categories \cite[\S C]{lurie-SAG}}.}, instead of plain triangulated categories: t-structures provide the necessary ``bridge'' between the classical (abelian) framework and the derived (higher) framework, and allow for direct generalizations of concepts such as injective or projective objects or results such as the Gabriel-Popescu theorem. 

We will work with \emph{differential graded (dg-) categories} as enhancements of triangulated categories (cf. \cite{bondal-kapranov-enhanced}). Following the above philosophy, the t-exact quasi-equivalence
\begin{equation}
    \cat A^+ \cong \Tw^+(\operatorname{DGInj(\cat A)}),
\end{equation}
was proven in \cite{genovese-lowen-vdb-dginj}, where $\cat A$ is a suitable pretriangulated dg-category with a t-structure and \emph{enough derived injectives}, and $\Tw^+(\operatorname{DGInj(\cat A)})$ is the dg-category of bounded below \emph{twisted complexes of derived injectives}, endowed with a suitable t-structure. There is a dual result when $\cat A$ has enough derived projectives:
\begin{equation}
    \cat A^- \cong \Tw^-(\DGProj(\cat A)).
\end{equation}
It is clear that the above results directly generalize \eqref{equation:hoinjectives_plus} and \eqref{equation:hoprojectives_minus}, which can be recovered by taking $\cat A = \dercompdg(\mathfrak A)$ to be the derived dg-category of a suitable abelian category $\mathfrak A$, endowed with the canonical t-structure with heart $\mathfrak A$.

The goal of this paper is to find a suitable generalization of the unbounded homotopy categories of injectives or projectives in this broader framework of t-structures. It turns out that we can define a nicely behaved dg-category of \emph{unbounded twisted complexes} $\Tw(\varcat A)$ over a suitable dg-category $\varcat A$ (which will be, in most real life applications, a dg-category of derived injectives or projectives of some given t-structure, namely, of the form $\operatorname{DGInj}(\cat A)$ or $\DGProj(\cat A)$). The main result (Theorem \ref{theorem:tstruct_unbounded}) will allow us to ``extend'' the t-structure on $\Tw^+(\operatorname{DGInj}(\cat A))$ (dually, on $\Tw^-(\DGProj(\cat A))$ to a uniquely determined \emph{injective t-structure} on $\Tw(\operatorname{DGInj}(\cat A))$ (dually, a uniquely determined \emph{projective t-structure} on $\Tw(\DGProj(\cat A))$). In particular, we will be able to speak of the \emph{cohomology} of a twisted complex of derived injectives or derived projectives, and we will have a natural notion of acyclicity.

This work is a direct sequel (or perhaps a spin-off) of \cite{genovese-lowen-vdb-dginj} and is completely foundational. It is the necessary basis of an upcoming project where the properties of dg-categories of the form $\Tw(\operatorname{DGInj}(\cat A))$ or $\Tw(\DGProj(\cat A))$ are investigated. For example, we expect -- under suitable assumptions -- compact generation. Moreover, if $R$ is a dg-algebra cohomologically concentrated in nonpositive degrees, we have dg-categories of derived injectives and derived projectives $\operatorname{DGInj}(R)$ and $\DGProj(R)$ associated to the canonical t-structure on the derived dg-category $\dercompdg(R)$; under suitable assumptions on $R$, the Grothendieck duality \eqref{equation:grothendieckduality} should hopefully be generalized to a quasi-equivalence
\[
\Tw(\DGProj(R)) \cong \Tw(\operatorname{DGInj}(R)).
\]

Nonpositive dg-algebras are essentially \emph{affine derived schemes}; we believe that the ``homotopy category of derived injectives'' described by such unbounded twisted complexes could also play a role in general derived algebraic geometry. This and possibly more will be addressed in future work.

\subsection*{Structure of the paper}
In \S \ref{section:dg_preliminaries} we deal with the background and preliminary results on dg-categories which we will need throughout the rest of the paper. In \S \ref{subsection:basics} we provide a concise survey on the basic concepts of dg-category theory. In \S \ref{subsec:pretr_dgcat} we discuss shifts of objects and cones of closed degree $0$ morphisms in dg-categories. There, we prove the technical Lemma \ref{lemma:iso_cones}, which deals with induced isomorphisms between cones and will be used in later results in the paper. \S \ref{subsec:addzeroobj} and \S \ref{subsec:addsumprod} deal with adjoining \emph{strict} zero objects, direct sums and products to a given dg-category. In particular, we prove Lemma \ref{lemma:homotopyproducts_closure}, where we check that we can replace a dg-category with ``homotopy products'' (or coproducts) with a dg-category with strict products (or coproducts). After dealing with sequential homotopy (co)limits in \S \ref{subsection:hocolim} and truncations in \S \ref{subsection:truncations}, we introduce the key notions of \emph{t-structure} and \emph{co-t-structure} on pretriangulated dg-categories in \S \ref{subsection:tstruct_def}.

In \S \ref{section:twistedcomplexes} we introduce the main object of our work, namely, \emph{twisted complexes} over a given dg-category. After discussing the basic definitions in \S \ref{subsection:twcomplexes_basics}, we show in \S \ref{subsection:twcompl_brutaltruncations} how a twisted complex can be reconstructed by taking suitable (homotopy) limits or colimits along suitable (left or right) \emph{brutal truncations}, see Corollary \ref{corollary:brutaltrunc_limitcolimit}. \S \ref{subsection:twcompl_iso} is devoted to proving Proposition \ref{proposition:iso_twisted_components}, which tells us that isomorphisms of twisted complexes can be characterized as ``componentwise isomorphisms''. In \S \ref{subsection:twcompl_qeq} we carefully check that taking twisted complexes preserves quasi-fully faithful dg-functors and quasi-equivalences (Lemma \ref{lemma:tw_quasifullyfaithful} and Proposition \ref{prop:Tw_preserve_quasiequivalences}). Finally, in \S  \ref{subsec:tw_addsumprod} we deal with closure of the dg-category of twisted complexes under cones, products or coproducts.

The last section of the paper (\S \ref{section:tstructures}) is devoted to the main result of the article. As mentioned above, our main goal is to prove that the dg-category of unbounded twisted complexes over a suitable dg-category is endowed with both a natural co-t-structure (discussed in \S \ref{subsection:cotstruct_twisted}) and a natural t-structure, ``extending'' either the t-structure on bounded below twisted complexes or the one on bounded above twisted complexes, already known from \cite{genovese-lowen-vdb-dginj} and revisited in \S \ref{subsection:tstruct_bounded_twisted}.
\begin{theorem*}[see Proposition \ref{proposition:twistedcomplexes_cotstruct} and Theorem \ref{theorem:tstruct_unbounded}]
Let $\varcat A$ be a dg-category with suitable properties. The dg-category $\Tw(\varcat A)$ of unbounded twisted complexes is endowed with a co-t-structure $(\Tw(\varcat A)^w_{\geq 0}, \Tw(\varcat A)^w_{\leq 0})$, where $\Tw(\varcat A)^w_{\geq 0}$ is essentially given by the twisted complexes concentrated in nonnegative degrees and $\Tw(\varcat A)^w_{\leq 0}$ is essentially given by the twisted complexes concentrated in nonpositive degrees. 

Suppose that our assumptions on $\varcat A$ guarantee that the dg-category $\Tw^-(\varcat A)$ of bounded above twisted complexes is endowed with the t-structure described in \cite{genovese-lowen-vdb-dginj} (see also Proposition \ref{proposition:tstructures_boundedtw}). Then, there is a unique t-structure $(\Tw(\varcat A)^{\operatorname{proj}}_{\leq 0}, \Tw(\varcat A)^{\operatorname{proj}}_{\geq 0})$ on $\Tw(\varcat A)$, called the \emph{projective t-stucture}, such that
\[ 
\Tw(\varcat A)^{\operatorname{proj}}_{\leq 0} = \Tw^-(\varcat A)_{\leq 0} = \Tw(\varcat A)^w_{\leq 0}
\]
and the inclusion $\Tw^-(\varcat A) \hookrightarrow \Tw(\varcat A)$ is t-exact. The heart of such t-structure is equivalent to $\modfp(H^0(\varcat A))$, the category of finitely presented right $H^0(\varcat A)$-modules.

Dually, suppose that our assumptions on $\varcat A$ guarantee that the dg-category $\Tw^+(\varcat A)$ of bounded below twisted complexes is endowed with the t-structure described in \cite{genovese-lowen-vdb-dginj} (see also Proposition \ref{proposition:tstructures_boundedtw}). Then, there is a unique t-structure $(\Tw(\varcat A)^{\operatorname{inj}}_{\leq 0}, \Tw(\varcat A)^{\operatorname{inj}}_{\geq 0})$ on $\Tw(\varcat A)$, called the \emph{injective t-stucture}, such that
\[
\Tw(\varcat A)^{\operatorname{inj}}_{\geq 0} = \Tw^+(\varcat A)_{\geq 0} = \Tw(\varcat A)^w_{\geq 0}
\]
and the inclusion $\Tw^+(\varcat A) \hookrightarrow \Tw(\varcat A)$ is t-exact. The heart of such t-structure is equivalent to $\opp{\modfp(H^0(\opp{\varcat A}))}$.
\end{theorem*}

A few words on the ``suitable assumptions'' on $\varcat A$ in the Theorem above. The existence of the natural co-t-structure on $\Tw(\varcat A)$ is very general and follows if $\varcat A$ is cohomologically concentrated in nonpositive degrees, with $H^0(\varcat A)$ being also additive. On the other hand, the existence of the projective t-structure on $\Tw(\varcat A)$ follows essentially if $\varcat A$ is a \emph{dg-category of derived projectives} (Definition \ref{definition:dgcat_derproj}); dually, the existence of the injective t-structure on $\Tw(\varcat A)$ follows essentially if $\varcat A$ is a \emph{dg-category of derived injectives}. For completeness, we discuss in Appendix \ref{appendix:derproj_cotstruct} the notions of \emph{derived projective} and \emph{derived injective} objects in triangulated categories with t-structures, and we explain how such concepts are strictly connected to co-t-structures interacting nicely with the given t-structures (Theorem \ref{theorem:derproj_cotstruct}). This result is most likely already known by experts but we could not find any specific reference for it.

\subsection*{Acknowledgements}
The author thanks Wendy Lowen and Michel Van den Bergh for first introducing him to derived injective objects and twisted complexes as a tool to understand resolutions in the framework of t-structures.

The author also thanks Adam-Christiaan Van Roosmalen for explaining him co-t-structures and their connection to injective and projective objects.

\emph{Last but not least}, the author thanks Jan \v{S}\v{t}ov\'{\i}\v{c}ek for the many interesting discussions and useful comments during the preparation of the manuscript.

\section{Preliminaries on dg-categories} \label{section:dg_preliminaries}

We fix a base commutative ring $\basering k$. Unless otherwise specified our constructions are in the $\basering k$-linear context, although we do not always say it. 

We also fix a Grothendieck universe $\mathbb U$; unless otherwise specified, all categories will be locally $\mathbb U$-small. We also fix another universe $\mathbb V \ni \mathbb U$, so that even locally $\mathbb U$-small categories will be $\mathbb V$-small: this will be useful when we need to take categories of functors. That said, we will essentially disregard such set-theoretical issues.

\subsection{Basics} \label{subsection:basics}
We recall the definition of \emph{differential graded (dg) category} and some basic constructions. We assume the reader to have some familiarity with the theory. See also \cite{keller-dgcat} \cite{toen-lectures}.
\begin{definition}
A \emph{dg-category} $\cat A$ is a category enriched over the closed symmetric monoidal category of chain complexes over $\basering k$. Concretely, it consists of a collection of objects $\Ob(\cat A)$, and for any pair of objects $A, B \in \Ob(\cat A)$ a complex of morphisms $\cat A(A,B)$, with unital and associative compositions:
\[
\cat A(B,C) \otimes \cat A(A,B) \to \cat A(A,C).
\]

Dg-functors between dg-categories are defined in the obvious way.
\begin{itemize}
    \item If $\cat A$ is a dg-category, we can define the \emph{opposite dg-category} $\opp{\cat A}$.
    \item If $\cat A$ is a dg-category, we have the \emph{homotopy category} $H^0(\cat A)$ and the \emph{graded homotopy category} $H^*(\cat A)$. They are obtained by taking the same set of objects and then zeroth cohomology or graded cohomology of the complexes of morphisms.
    \item Complexes of $\basering k$-modules form a dg-category $\compdg(\basering k)$.
    \item Let $\cat A$ and $\cat B$ be dg-categories. There is a \emph{tensor product} $\cat A \otimes \cat B$ and a \emph{dg-category of dg-functors} $\Fundg(\cat A, \cat B)$.
    \item Let $\cat A$ be a dg-category. We denote by
    \[
    \compdg(\cat A) = \Fundg(\opp{\cat A}, \compdg(\basering k))
    \]
    the dg-category of \emph{right $\cat A$-dg-modules}. The Yoneda Lemma holds and yields the Yoneda embedding:
    \begin{align*}
    h = h_{\cat A} \colon \cat A & \hookrightarrow \compdg(\cat A), \\
    A & \mapsto \cat A(-,A).
    \end{align*}
    \item Let $\cat A$ be a dg-category. Its \emph{derived category} $\dercomp(\cat A)$ is defined as the Verdier quotient of $H^0(\compdg(\cat A))$ by the subcategory of acyclic dg-modules. 
    
    Dg-enhancements of $\dercomp(\cat A)$ can be described using \emph{h-projective} of $\emph{h-injective}$ dg-modules (see \cite{lunts-schnurer-smoothness-equivariant} for details). Namely:
    \[
    H^0(\operatorname{h-proj}(\cat A)) \cong \dercomp(\cat A), \qquad H^0(\operatorname{h-inj}(\cat A)) \cong \dercomp(\cat A).
    \]

    We remark that the Yoneda embedding factors through $\hproj(\cat A)$:
    \[
    \cat A \hookrightarrow \hproj(\cat A),
    \]
    hence it induces a \emph{derived Yoneda embedding}
    \[
    H^0(\cat A) \hookrightarrow \dercomp(\cat A).
    \]
\end{itemize}
\end{definition}

\subsubsection*{Quasi-functors} We have a category $\kat{dgCat}$ of (small with respect to some universe) dg-categories and dg-functors. By formally inverting the \emph{quasi-equivalences}, we obtain the \emph{homotopy category of dg-categories}
\[
\Hqe = \kat{dgCat}[\operatorname{Qe}^{-1}].
\]
The tensor product of dg-categories can be derived, yielding a symmetric monoidal structure $- \lotimes -$ in $\Hqe$. An important result (see \cite{toen-dgcat-invmath} \cite{canonaco-stellari-internalhoms}) is that the resulting symmetric monoidal category $(\Hqe, \lotimes)$ is closed. Namely, for (small) dg-categories $\cat A$ and $\cat B$ there is a dg-category $\RHom(\cat A, \cat B)$ and a natural isomorphism
\[
\Hqe(\cat C \lotimes \cat A, \cat B) \cong \Hqe(\cat C, \RHom(\cat A, \cat B)).
\]

The dg-category $\RHom(\cat A, \cat B)$ can be described in terms of \emph{quasi-functors} (see \cite{canonaco-stellari-internalhoms}). For our purposes, a quasi-functor $F \colon \cat A \to \cat B$ is a dg-functor $F \colon \cat A \to \compdg(\cat B)$ such that $F(A)$ is quasi-isomorphic to $\cat B(-,\Phi_F(A))$ for some $\Phi_F(A) \in \cat B$.

Any quasi-functor $F \colon \cat A \to \cat B$ induces a graded functor $H^*(F) \colon H^*(\cat A) \to H^*(\cat B)$. We say that $F$ is \emph{invertible}, or (with a little abuse of terminology) a \emph{quasi-equivalence}, if $H^*(F)$ is an equivalence. In that case, we may conclude that $\cat A$ and $\cat B$ are isomorphic in the homotopy category $\Hqe$.

\subsection{Pretriangulated dg-categories} \label{subsec:pretr_dgcat}
Let $\cat A$ be a dg-category. We denote by $\pretr(\cat A)$ the \emph{pretriangulated hull} of $\cat A$, namely, the closure of the image of the Yoneda embedding $\cat A \hookrightarrow \compdg(\cat A)$ in $\compdg(\cat A)$ under taking shifts and mapping cones. We remark that the full dg-subcategory $\hproj(\cat A)$ of h-projective dg-modules in $\compdg(\cat A)$ contains the image of the Yoneda embedding and is closed under shifts and mapping cones. Hence, the Yoneda embedding factors as follows:
\[
\cat A \hookrightarrow \pretr(\cat A) \hookrightarrow \hproj(\cat A) \hookrightarrow \compdg(\cat A).
\]

\begin{definition}[cf. {\cite{bondal-kapranov-enhanced}}]
Let $\cat A$ be a dg-category. We say that $\cat A$ is \emph{strongly pretriangulated} if $\cat A \hookrightarrow \pretr(\cat A)$ is a dg-equivalence. We say that $\cat A$ is \emph{pretriangulated} if $\cat A \hookrightarrow \pretr(\cat A)$ is a quasi-equivalence.
\end{definition}
The dg-categories $\compdg(\cat A), \hproj(\cat A), \operatorname{h-inj}(\cat A), \pretr(\cat A)$ are all strongly pretriangulated. If $\cat A$ is pretriangulated, we can replace it up to quasi-equivalence by $\pretr(\cat A)$, which is strongly pretriangulated.

The homotopy category $H^0(\cat A)$ of a pretriangulated dg-category $\cat A$ has a ``canonical'' structure of triangulated category. The crucial property of pretriangulated dg-categories is that, unlike triangulated categories, they have \emph{functorial shifts and cones}.

We can check that a dg-category $\cat A$ is strongly pretriangulated if and only if it is closed under \emph{pretriangles}, namely, sequences of the form
\begin{equation} \label{equation:pretriangle}
\begin{tikzcd}[ampersand replacement=\&]
	A \& B \& {\cone(f)} \& {A[1],}
	\arrow["f", from=1-1, to=1-2]
	\arrow["p", shift left=1, from=1-3, to=1-4]
	\arrow["j", shift left=1, from=1-2, to=1-3]
	\arrow["s", shift left=1, from=1-3, to=1-2]
	\arrow["i", shift left=1, from=1-4, to=1-3]
\end{tikzcd}
\end{equation}
where $f \colon A \to B$ is a closed (that is, $df=0$) degree $0$ morphism in $\cat A$. $\cone(f)$ is called the \emph{cone of $f$} and $A[1]$ is called the \emph{$1$-shift of $A$}. In general, the $m$-shifts $A[m]$ of $A$ come with closed invertible degree $n-m$ morphisms (``shifted identitity morphisms'')
\[
1_{(A,n,m)} \colon A[n] \to A[m],
\]
such that $1_{(A,m,n)} \circ 1_{(A,n,m)} = 1_{(A,n,n)}=1_{A[n]}$. The morphisms $j,p,i,s$ in the pretriangle \eqref{equation:pretriangle} are of degree $0$ and characterize $\cone(f)$ as the biproduct $A[1] \oplus B$ in the underlying graded category of $\cat A$. Moreover, they satisfy the following equations:
\[
dj=0, \quad dp=0, \quad di = jf 1_{(A,1,0)}, \quad ds =  -f 1_{(A,1,0)} p.
\]
We refer to \cite[\S 4.3]{bondal-larsen-lunts-grothendieck} for more details.

We remark that a dg-category $\cat A$ is strongly pretriangulated if and only if its opposite $\opp{\cat A}$ is strongly pretriangulated. A pretriangle in $\opp{\cat A}$ corresponds to a ``rotated pretriangle'' in $\cat A$, namely, a sequence of the form:
\[
A[-1] \to \cone(f)[-1] \to A \xrightarrow{f} B. 
\]

Clearly, some shifts and cones (hence, some pretriangles) may exist in a given dg-category $\cat A$ even if $\cat A$ is not strongly pretriangulated.
\begin{remark}
Let $\cat A$ be a strongly pretriangulated dg-category. Consider the following (not necessarily commutative) diagram:
\[\begin{tikzcd}[ampersand replacement=\&]
	A \& B \\
	{A'} \& {B',}
	\arrow["f", from=1-1, to=1-2]
	\arrow["u"', from=1-1, to=2-1]
	\arrow["{f'}"', from=2-1, to=2-2]
	\arrow["v", from=1-2, to=2-2]
	\arrow["h"{description}, from=1-1, to=2-2]
\end{tikzcd}\]
where $f,f'$ are closed degree $0$ morphisms, $u,v$ are degree $n$ morphisms and $h$ is a degree $n-1$ morphism. We can find a morphism
\[
w \colon \cone(f) \to \cone(f')
\]
determined by the triple $(u,v,h)$, such that the central and the right squares of the following diagram are (strictly) commutative:
\[\begin{tikzcd}[ampersand replacement=\&]
	A \& B \& {\cone(f)} \& {A[1]} \\
	{A'} \& {B'} \& {\cone(f')} \& {A'[1].}
	\arrow["f", from=1-1, to=1-2]
	\arrow["u"', from=1-1, to=2-1]
	\arrow["{f'}"', from=2-1, to=2-2]
	\arrow["v", from=1-2, to=2-2]
	\arrow["j", from=1-2, to=1-3]
	\arrow["p", from=1-3, to=1-4]
	\arrow["{j'}", from=2-2, to=2-3]
	\arrow["{p'}", from=2-3, to=2-4]
	\arrow["w", dashed, from=1-3, to=2-3]
	\arrow["{u[1]}", from=1-4, to=2-4]
\end{tikzcd}\]
The rows are pretriangles and $j,j',p,p'$ are the natural morphisms associated to them. In matrix notation, with respect to the biproduct decompositions $\cone(f) = A[1] \oplus B$ and $\cone(f')=A'[1] \oplus B'$, the morphism $w$ and its differential $dw$ are given by:
\[
w = \begin{pmatrix} u[1] & 0 \\ h 1_{(A,1,0)} & v \end{pmatrix}, \qquad dw = \begin{pmatrix} (-du)[1] & 0 \\ (dh +f'u -(-1)^n vf) 1_{(A,1,0)} & dv \end{pmatrix}.
\]
\end{remark}

We now prove a technical lemma:
\begin{lemma} \label{lemma:iso_cones}
Let $\cat A$ be a dg-category. Consider the following diagram of objects and morphisms in $\cat A$:
\[\begin{tikzcd}[ampersand replacement=\&]
	A \& B \& {\cone(f)} \& {A[1]} \\
	{A'} \& {B'} \& {\cone(f')} \& {A'[1],}
	\arrow["f", from=1-1, to=1-2]
	\arrow["u"', from=1-1, to=2-1]
	\arrow["{f'}"', from=2-1, to=2-2]
	\arrow["v", from=1-2, to=2-2]
	\arrow["h"{description}, from=1-1, to=2-2]
	\arrow["j", from=1-2, to=1-3]
	\arrow["j'", from=2-2, to=2-3]
	\arrow["p", from=1-3, to=1-4]
	\arrow["p'", from=2-3, to=2-4]
	\arrow["{u[1]}", from=1-4, to=2-4]
	\arrow["w", from=1-3, to=2-3]
\end{tikzcd}\]
where the rows are pretriangles and all morphisms are closed and of degree $0$, except $h$ which is of degree $-1$ and such that 
\[
dh = vf -f'u.
\]
The middle and right squares of the above diagram are strictly commutative. The morphism $w$ is expressed in matrix notation as
\[
w = \begin{pmatrix} u[1] & 0 \\ h 1_{(A,1,0)} & v \end{pmatrix}.
\]

If $u$ and $v$ have inverses $u'$ and $v'$ in $H^0(\cat A)$, then there is a degree $-1$ morphism $h' \colon A' \to B$ such that
\[
dh' = v'f'-fu'
\]
and such that the morphism
\[
w' = \begin{pmatrix} u'[1] & 0 \\ h' 1_{(A',1,0)} & v' \end{pmatrix} \colon \cone(f') \to \cone(f)
\]
is an inverse of $w$ in $H^0(\cat A)$. The morphism $w'$ automatically fits in the following diagram:
\[\begin{tikzcd}[ampersand replacement=\&]
	A' \& B' \& {\cone(f')} \& {A'[1]} \\
	{A} \& {B} \& {\cone(f)} \& {A[1].}
	\arrow["f'", from=1-1, to=1-2]
	\arrow["u'"', from=1-1, to=2-1]
	\arrow["{f}"', from=2-1, to=2-2]
	\arrow["v'", from=1-2, to=2-2]
	\arrow["h'"{description}, from=1-1, to=2-2]
	\arrow["j'", from=1-2, to=1-3]
	\arrow["j", from=2-2, to=2-3]
	\arrow["p'", from=1-3, to=1-4]
	\arrow["p", from=2-3, to=2-4]
	\arrow["{u'[1]}", from=1-4, to=2-4]
	\arrow["w'", from=1-3, to=2-3]
\end{tikzcd}\]
The middle and right squares of the above diagram are strictly commutative.
\end{lemma}
\begin{proof}
We first deal with the inverses in $H^0(\cat A)$, which is the most interesting case. From the equality 
\[
[vf] = [f'u]
\]
in $H^0(\cat A)$, and the invertibility of $[u]$ and $[v]$, we deduce that there exists a degree $-1$ morphism $h_0 \colon A' \to B$ such that
\[
v'f'-fu'=dh_0.
\]
The invertibility of $u$ and $v$ in $H^0(\cat A)$ is expressed explicitly as follows:
\begin{align*}
& \begin{cases}
u'u = 1_A - d\tilde{u}, \\
v'v = 1_B + d\tilde{v},
\end{cases}, \\
& \begin{cases}
uu' = 1_{A'} - d\tilde{u'}, \\
vv' = 1_{B'} + d\tilde{v'}.
\end{cases}
\end{align*}
for suitable maps $\tilde{u}, \tilde{u'}, \tilde{v}, \tilde{v'}$ of degree $-1$. We are going to find closed morphisms
\begin{align*}
    z'_0, z'_1 \colon A' \to B
\end{align*}
such that
\[
\left[\begin{pmatrix} u'[1] & 0 \\ (h_0+z'_0) 1_{(A',1,0)} & v' \end{pmatrix} \right]
\]
will be a left inverse of $[w]$, and
\[
\left[\begin{pmatrix} u'[1] & 0 \\ (h_0+z'_1) 1_{(A',1,0)} & v' \end{pmatrix} \right]
\]
will be a right inverse of $[w]$. Closedness of $z'_0$ and $z'_1$ implies that the above matrices yield closed morphisms, so that we get well defined morphisms in $H^0(\cat A)$. In the end, such left and right inverses will coincide (in $H^0(\cat A)$) and we may take $h'$ to be either $h_0+z'_0$ or $h_0+z'_1$. We start by setting
\begin{align*}
r & := f\tilde{u} + \tilde{v}f - h'u -v'h, \\
r' & := f'\tilde{u'}+\tilde{v'}f' - hu'-vh'. 
\end{align*}
A simple computation gives $dr=0, dr'=0$. We define:
\begin{align*}
z'_0 & := ru', \\
z'_1 & := v'r'.
\end{align*}
$z'_0$ and $z'_1$ are indeed closed; we have $[z'_0 u]= [r]$ and $[v z'_1] = [r']$ in $H^0(\cat A)$, namely
\begin{align*}
(h'+z'_0)u + v'h &= d\tilde{h} + f\tilde{u} + \tilde{v}f \\
hu' + v(h'+z'_1) &= d\tilde{h'} + f'\tilde{u'} + \tilde{v'}f',
\end{align*}
for suitable maps $\tilde{h}, \tilde{h'}$. Finally, we obtain
\begin{align*}
 \begin{pmatrix} u'[1] & 0 \\ (h_0+z'_0)1_{(A',1,0)} & v' \end{pmatrix} & \begin{pmatrix} u[1] & 0 \\ h 1_{(A,1,0)} & v \end{pmatrix} \\
 &= \begin{pmatrix} 1_A[1] & 0 \\ 0 & 1_B \end{pmatrix} + \begin{pmatrix} (-d\tilde{u})[1] & 0 \\ (d\tilde{h} + f \tilde{u} + \tilde{v}f)1_{(A,1,0)} & d\tilde{v} \end{pmatrix} \\
& = 1_{\cone(f)} + d\begin{pmatrix} \tilde{u}[1] & 0 \\ \tilde{h}1_{(A,1,0)} & \tilde{v} \end{pmatrix}.
\end{align*}
and analogously
\begin{align*}
\begin{pmatrix} u[1] & 0 \\ h1_{(A,1,0)} & v \end{pmatrix} & \begin{pmatrix} u'[1] & 0 \\ (h_0 +z'_1) 1_{(A',1,0)} & v' \end{pmatrix} \\
 &= \begin{pmatrix} 1_{A'}[1] & 0 \\ 0 & 1_{B'} \end{pmatrix} + \begin{pmatrix} (-d\tilde{u}')[1] & 0 \\ (d\tilde{h}' + f' \tilde{u}' + \tilde{v}'f')1_{(A',1,0)} & d\tilde{v}' \end{pmatrix} \\
& = 1_{\cone(f')} + d\begin{pmatrix} \tilde{u}'[1] & 0 \\ \tilde{h}' 1_{(A',1,0)} & \tilde{v}' \end{pmatrix}.
\end{align*}
Hence, the proof is completed in the case of inverses in $H^0(\cat A)$.

The case where $u$ and $v$ have strict inverses (in $Z^0(\cat A)$) is easier. We may just set
\[
h' = v^{-1} h u^{-1}
\]
and a direct computation yields the result.
\end{proof}
\begin{remark} \label{remark:iso_cones_boundaries}
The proof of the above Lemma \ref{lemma:iso_cones} in the case of inverses in $H^0(\cat A)$ actually yields a more precise result, as follows. In the same setting and notations as above, we obtain closed degree $0$ morphisms
\[
w^r, w^l \colon \cone(f') \to \cone(f),
\]
which fit in the diagram
\[\begin{tikzcd}[ampersand replacement=\&]
	A' \& B' \& {\cone(f')} \& {A'[1]} \\
	{A} \& {B} \& {\cone(f)} \& {A[1],}
	\arrow["f'", from=1-1, to=1-2]
	\arrow["u'"', from=1-1, to=2-1]
	\arrow["{f}"', from=2-1, to=2-2]
	\arrow["v'", from=1-2, to=2-2]
	\arrow["j'", from=1-2, to=1-3]
	\arrow["j", from=2-2, to=2-3]
	\arrow["p'", from=1-3, to=1-4]
	\arrow["p", from=2-3, to=2-4]
	\arrow["{u'[1]}", from=1-4, to=2-4]
	\arrow["{w^l}"', shift right=1, from=1-3, to=2-3]
	\arrow["{w^r}", shift left=1, from=1-3, to=2-3]
\end{tikzcd}\]
where the middle and right squares are strictly commutative. Moreover, there exist degree $-1$ morphisms
\begin{align*}
    h^l \colon \cone(f) &\to \cone(f), \\
    h^r \colon \cone(f') & \to \cone(f'),
\end{align*}
such that
\begin{align*}
    w^l \circ w & = 1_{\cone(f)} + dh^l, \\
    w \circ w^r &= 1_{\cone(f')} + dh^r,
\end{align*}
and $h^l$ and $h^r$ fit (respectively) in the following diagrams:
\begin{align*}
    \begin{tikzcd}[ampersand replacement=\&]
	A \& B \& {\cone(f)} \& {A'[1]} \\
	{A} \& {B} \& {\cone(f)} \& {A[1],}
	\arrow["f'", from=1-1, to=1-2]
	\arrow["\tilde{u}"', from=1-1, to=2-1]
	\arrow["{f}"', from=2-1, to=2-2]
	\arrow["\tilde{v}", from=1-2, to=2-2]
	\arrow["j", from=1-2, to=1-3]
	\arrow["j", from=2-2, to=2-3]
	\arrow["p", from=1-3, to=1-4]
	\arrow["p", from=2-3, to=2-4]
	\arrow["{\tilde{u}[1]}", from=1-4, to=2-4]
	\arrow["{h^l}", from=1-3, to=2-3]
\end{tikzcd} \\
\begin{tikzcd}[ampersand replacement=\&]
	A' \& B' \& {\cone(f')} \& {A'[1]} \\
	{A'} \& {B'} \& {\cone(f')} \& {A'[1],}
	\arrow["f'", from=1-1, to=1-2]
	\arrow["\tilde{u}'"', from=1-1, to=2-1]
	\arrow["{f'}"', from=2-1, to=2-2]
	\arrow["\tilde{v}'", from=1-2, to=2-2]
	\arrow["j'", from=1-2, to=1-3]
	\arrow["j'", from=2-2, to=2-3]
	\arrow["p'", from=1-3, to=1-4]
	\arrow["p'", from=2-3, to=2-4]
	\arrow["{\tilde{u}[1]}", from=1-4, to=2-4]
	\arrow["{h^r}", from=1-3, to=2-3]
\end{tikzcd}
\end{align*}
where the middle and right squares are strictly commutative.
\end{remark}

\subsection{Formally adding zero objects} \label{subsec:addzeroobj}
Let $\cat A$ be a dg-category. We define a dg-category $\cat A_{\{0\}}$ as follows:
\begin{itemize}
    \item $\Ob(\cat A_{\{0\}}) = \Ob(\cat A) \coprod \{0\}$.
    \item Morphisms are described as follows:
    \begin{equation*}
        \cat A_{\{0\}}(A,B) = \begin{cases} \cat A(A,B) & \text{if $A,B \in \Ob(\cat A)$} \\
        0 & \text{if $A=0$ or $B=0$} \end{cases}
    \end{equation*}
    with the obvious compositions.
\end{itemize}

Clearly, $\cat A_{\{0\}}$ can be identified with the full dg-subcategory of $\compdg(\cat A)$ containing the image of the Yoneda embedding $\cat A \hookrightarrow \compdg(\cat A)$ and the zero dg-module
\[
0(A)= 0 \in \compdg(\basering k), \qquad A \in \cat A.
\]

If $u \colon \cat A \to \cat B$ is a dg-functor, there is an induced dg-functor
\[
u_{\{0\}} \colon \cat A_{\{0\}} \to \cat B_{\{0\}},
\]
defined by
\[
u_{\{0\}}(A) = \begin{cases} u(A) & \text{if $A \in \Ob(\cat A)$}, \\
 0 & \text{if $A=0$.} \end{cases}
\]
We get a functor
\begin{equation}
\begin{split}
    (-)_{\{0\}} \colon \kat{dgCat} & \to \kat{dgCat}, \\
    \cat A & \mapsto \cat A_{\{0\}}.
\end{split}
\end{equation}
The following result is straightforward.
\begin{proposition}
Let $\cat A$ be a dg-category. There is a fully faithful dg-functor
\[
\cat A \hookrightarrow \cat A_{\{0\}},
\]
which is natural in $\cat A \in \kat{dgCat}$.

Moreover:
\begin{itemize}
    \item If $u \colon \cat A \to \cat B$ is quasi-fully faithful (respectively a quasi-equivalence), the same is true for $u_{\{0\}} \colon \cat A_{\{0\}} \to \cat B_{\{0\}}$.
    \item If $\cat A$ is such that $H^0(\cat A)$ has zero objects, the dg-functor $\cat A \hookrightarrow \cat A_{\{0\}}$ is a quasi-equivalence.
\end{itemize}
\end{proposition}

\subsection{Adjoining strict direct sums or products} \label{subsec:addsumprod}
Let $\cat A$ be a dg-category. If $\cat A$ has (finite or infinite) \emph{cohomological} direct sums or products, we would like to replace it with a quasi-equivalent dg-category which has \emph{strict} direct sums or products.

In the case of finite direct sums (which are the same as finite products) this is not too difficult:
\begin{lemma} \label{lemma:dgcat_strictadditive_closure}
Let $\cat A$ be a dg-category. We define the dg-category $\cat A^\oplus$ as the closure of $\cat A$ in $\compdg(\cat A)$ under finite direct sums (including zero objects), namely, the full dg-subcategory of $\compdg(\cat A)$ whose objects are finite direct sums of representables $\cat A(-,A)$. Clearly, $\cat A^\oplus$ has strict direct sums and zero objects.

If $H^0(\cat A)$ is additive, then the inclusion dg-functor
\[
\cat A \hookrightarrow \cat A^\oplus
\]
is a quasi-equivalence.
\end{lemma}
\begin{proof}
We just need to show essential surjectivity of $H^0(\cat A) \to H^0(\cat A^\oplus)$. Let
\[
M = \cat A(-,A_1) \oplus \cat A(-,A_2)
\]
be a binary product in $\cat A^\oplus$. We are going to show that this is isomorphic to some representable $\cat A(-,A)$ in $H^0(\cat A^\oplus)$; the case of any finite direct sum will be obtained by a straightforward induction. So, let $A$ be a direct sum of $A_1$ and $A_2$ in $H^0(\cat A)$. This is actually a biproduct, so we have degree $0$ morphisms
\[
j_i \colon A_i \to A, \quad p_i \colon A \to A_i, \qquad i=1,2,
\]
such that
\[
[p_i j_i] =[1_{A_i}], \quad [i_1p_1+i_2p_2] =[1_A], \quad [p_2 j_1]=[0], \quad [p_1 j_2]=[0].
\]
We can use the $p_i$ and $j_i$ to define morphisms
\[
\begin{psmallmatrix}
{j_1}_* \\ {j_2}_* 
\end{psmallmatrix} \colon \cat A(-,A) \to \cat A(-,A_1) \oplus \cat A(-,A_2)
\]
and
\[
({p_1}_*, {p_2}_*) \colon \cat A(-,A_1) \oplus \cat A(-,A_2) \to \cat A(-,A).
\]
It is then immediate to see that these morphisms are inverse to each other when viewed in $H^0(\cat A^\oplus)$.

To conclude, we show that if $Z$ is a zero object in $H^0(\cat A)$, then $\cat A(-,Z) \cong 0$ in $H^0(\cat A^\oplus)$. Indeed, the identity morphism $1_Z \colon Z \to Z$ is such that $[1_Z]=[0]$ in $H^0(\cat A)$. This implies that the same is true for
\[
[1_{\cat A(-,Z)}] \colon \cat A(-,Z) \to \cat A(-,Z).
\]
in $H^0(\cat A^\oplus)$. This means that $\cat A(-,Z)$ is a zero object in $H^0(\cat A^\oplus)$.
\end{proof}

Adjoining \emph{infinite} strict direct sums or products is a little trickier. We first deal with products; the case of coproducts will be dual and easily dealt with later afterwards. The goal is to prove the following:
\begin{lemma} \label{lemma:homotopyproducts_closure}
Let $\cat A$ be a dg-category such that the graded cohomology $H^*(\cat A)$ has direct products indexed by sets of cardinality $\leq \kappa$, where $\kappa$ is an infinite regular cardinal. Then, we can find a dg-category $\cat A^\Pi$ which has strict products indexed by sets of cardinality $\leq \kappa$, which is quasi-equivalent to $\cat A$.

Dually, let $\cat A$ be a dg-category such that the graded cohomology $H^*(\cat A)$ has direct sums indexed by sets of cardinality $\leq \kappa$, where $\kappa$ is an infinite regular cardinal. Then, we can find a dg-category $\cat A^\amalg$ which has strict products indexed by sets of of cardinality $\leq \kappa$, which is quasi-equivalent to $\cat A$.
\end{lemma}
\begin{proof}
We prove only the first claim, the other being obtained by duality, namely, by replacing $\cat A$ with $\opp{\cat A}$. Let $\{A_i : i \in I\}$ be a family of objects in $\cat A$ indexed by a set $I$ of cardinality $\leq \kappa$. Let $A=\prod_i A_i$ be a product of the $A_i$ in $H^*(\cat A)$. This implies that the morphism
\[
\cat A(-,A) \to \prod_i \cat A(-,A_i)
\]
is a quasi-isomorphism in $\compdg(\cat A)$. Naively, we would take the closure of $\cat A$ in $\compdg(\cat A)$ under such products, but this will not work: quasi-isomorphisms are not isomorphisms in $H^0(\compdg(\cat A))$. Hence, we need to take resolutions. \emph{h-injective resolutions} (cf. \cite[\S 4.1.1]{lunts-schnurer-smoothness-equivariant}) are the ones we need, because the product of h-injective dg-modules is again h-injective. Technically, we argue as follows: consider the full dg-subcategory $\widetilde{\cat A}$ of $\compdg(\cat A)$ of dg-modules of the form $R \cat A(-,A)$, where in general $M \to R(M)$ is an h-injective resolution of any dg-module $M \in \compdg(\cat A)$. The inclusion
\[
\widetilde{\cat A} \hookrightarrow \compdg(\cat A)
\]
is actually a quasi-functor $\widetilde{\cat A} \to \cat A$. Indeed, every $R\cat A(-,A)$ is quasi-isomorphic to $\cat A(-,A)$, being an h-injective resolution. Moreover, this quasi-functor has an inverse in graded cohomology:
\begin{align*}
H^*(\cat A) & \to H^*(\widetilde{\cat A}), \\
A & \mapsto R \cat A(-,A).
\end{align*}
We may conclude that it is an invertible quasi-functor, hence $\widetilde{\cat A}$ is indeed quasi-equivalent to $\cat A$.

Now, we may define $\cat A^\Pi$ as the smallest full dg-subcategory of $\operatorname{h-inj}(\cat A)$ which contains $\widetilde{\cat A}$ and it is closed under strict products (indexed by sets of cardinality $\leq \kappa$). It is straightforward to check that $\cat A^\Pi$ is the full dg-subcategory of $\operatorname{h-inj}(\cat A)$ spanned by
\[
\left\{ \prod_{i \in I} \widetilde{A}_i : \widetilde{A}_i \in \widetilde{\cat A},\ |I| \leq \kappa \right\}.
\]

If $H^*(\cat A)$ has such indexed products, we can prove that the inclusion
\[
\widetilde{\cat A} \hookrightarrow \cat A^\Pi
\]
is a quasi-equivalence. Indeed, let $\{A_i : i \in I\}$ be a family of objects in $\cat A$ (with $|I| \leq \kappa$), and let $A$ be a product of the $A_i$ in $H^*(\cat A)$. Taking h-injective resolutions, we get a morphism
\[
R\cat A(-,A) \to \prod_{i \in I} R\cat A(-,A_i)
\]
in $\operatorname{h-inj}(\cat A)$. We know that it is a quasi-isomorphism; but quasi-isomorphisms between h-injective dg-modules are actually isomorphisms in $H^0(\operatorname{h-inj}(\cat A))$. We conclude that
\[
H^0(\widetilde{\cat A}) \to H^0(\cat A^\Pi)
\]
is essentially surjective, hence $\widetilde{\cat A} \hookrightarrow \cat A^\Pi$ is indeed a quasi-equivalence. We conclude that our original $\cat A$ is quasi-equivalent to $\cat A^\Pi$, the latter having strict products. 
\end{proof}
\begin{remark}
By construction, both $\cat A^\Pi$ and $\cat A^\amalg$ have strict zero objects (obtained as empty products or coproducts) and strict finite direct sums.
\end{remark}

\subsection{Homotopy (co)limits} \label{subsection:hocolim}
We will discuss the definition and some properties of \emph{sequential} homotopy (co)limits in dg-categories, which we will need in this work. Such homotopy (co)limits are understood as ``Milnor (co)limits'' using mapping telescopes, as follows.
\begin{definition}
Let $\cat A$ be a dg-category. Let $(A_{n+1} \xrightarrow{a_{n+1,n}} A_n)_{n \geq 0}$ be a sequence of closed degree $0$ morphisms in $\cat A$. A \emph{homotopy limit} of this sequence is an object $\holim_n A_n \in \cat A$ together with a quasi-isomorphism of right $\cat A$-dg-modules:
\[
\cat A(-,\holim_n A_n) \xrightarrow{\sim} \holim_n \cat A(-,A_n),
\]
where $\holim_n \cat A(-,A_n)$ sits in the following rotated pretriangle of right $\cat A$-dg-modules (by ``rotated pretriangle'' we just mean ``pretriangle in the opposite category''):
\[
\holim_n \cat A(-,A_n) \to \prod_{n \geq 0} \cat A(-,A_n) \xrightarrow{1-\nu} \prod_{n \geq 0} \cat A(-,A_n).
\]
The morphism $\nu$ is induced by
\[
\prod_{n \geq 0} \cat A(-,A_n) \xrightarrow{\operatorname{pr_{n+1}}} \cat A(-,A_{n+1}) \xrightarrow{(a_{n+1,n})_*} \cat A(-,A_n). 
\]

\emph{Homotopy colimits} are understood as homotopy limits in the opposite dg-category $\opp{\cat A}$.
\end{definition}

Let $\cat A$ be a pretriangulated dg-category, and assume moreover that $H^0(\cat A)$ has countable products. Since $\cat A$ is pretriangulated, this implies that $H^*(\cat A)$ has countable products, and moreover that for any countable family $\{A_n : n \geq 0\}$ we have a quasi-isomorphism of right $\cat A$-dg-modules:
\[
\cat A(-,\prod_{n \geq 0} A_n) \xrightarrow{\approx} \prod_{n \geq 0} \cat A(-,A_n),
\]
namely, $\cat A$ has \emph{homotopy products}.

In this setup, let $(A_{n+1} \xrightarrow{a_{n+1,n}} A_n)_{n \geq 0}$ be a sequence of closed degree $0$ morphisms in $\cat A$. We can take the following distinguished triangle in $H^0(\cat A)$:
\begin{equation}
\holim_n A_n \to \prod_{n \geq 0} A_n \xrightarrow{1-\nu} \prod_{n \geq 0} A_n,
\end{equation}
where $\nu$ is the morphism in $H^0(\cat A)$ induced by
\[
\prod_{n \geq 0} A_n \xrightarrow{pr_{n+1}} A_{n+1} \xrightarrow{a_{n+1,n}} A_n.
\]
We obtain the following commutative diagram in $\dercomp(\cat A)$:
\[\begin{tikzcd}[ampersand replacement=\&]
	{\cat A(-,\holim_n A_n)} \& {\cat A(-,\prod_{n \geq 0} A_n)} \& {\cat A(-,\prod_{n \geq 0} A_n)} \\
	{\holim_n \cat A(-,A_n)} \& {\prod_{n \geq 0}\cat A(-,A_n)} \& {\prod_{n \geq 0} \cat A(-,A_n).}
	\arrow[from=1-1, to=1-2]
	\arrow["{(1-\nu)_*}", from=1-2, to=1-3]
	\arrow["\approx"', from=1-2, to=2-2]
	\arrow["{1-\nu}", from=2-2, to=2-3]
	\arrow["\approx"', from=1-3, to=2-3]
	\arrow["\approx"', from=1-1, to=2-1]
	\arrow[from=2-1, to=2-2]
\end{tikzcd}\]
The rows are distinguished triangles, and the vertical morphisms are all quasi-isomorphisms. We abused notation and wrote $\nu$ for both the morphism $\prod_{n \geq 0} A_n \to \prod_{n \geq 0} A_n$ and the morphism $\prod_{n \geq 0} \cat A(-,A_n) \to \prod_{n \geq 0} \cat A(-,A_n)$. We conclude that the object $\holim_n A_n \in \cat A$, together with the quasi-isomorphism
\[
\cat A(-,\holim_n A_n) \xrightarrow{\approx} \holim_n \cat A(-,A_n),
\]
is a homotopy limit of our given sequence.

Dually, assume that $\cat A$ is pretriangulated and $H^0(\cat A)$ has countable coproducts (which we denote as direct sums). Let $(A_n \xrightarrow{a_{n,n+1}} A_{n+1})_{n \geq 0}$ be a sequence of closed degree morphisms in $\cat A$. We can take the following distinguished triangle in $H^0(\cat A)$:
\begin{equation}
 \bigoplus_{n \geq 0} A_n \xrightarrow{1-\mu} \bigoplus_{n \geq 0} A_n \to \hocolim_n A_n,
\end{equation}
where $\mu$ is the morphism in $H^0(\cat A)$ induced by
\[
A_n \xrightarrow{a_{n,n+1}} A_{n+1} \xrightarrow{\operatorname{incl}_{n+1}} \bigoplus_{n \geq 0} A_n.
\]
Reasoning as above, we find a quasi-isomorphism of left $\cat A$-dg-modules:
\[
\cat A(\hocolim_n A_n,-) \xrightarrow{\approx} \cat \holim_n \cat A(A_n,-),
\]
which exhibits $\hocolim_n A_n$ as the homotopy colimit of the given sequence.

\subsubsection*{Comparison with ordinary limits and colimits}
Let $\cat A$ be a dg-category, and let $(A_{n+1} \xrightarrow{a_{n+1,n}} A_n)_{n \geq 0}$ be a sequence of closed degree $0$ morphisms in $\cat A$. We may define the \emph{limit} of this sequence as an object $\varprojlim_n A_n \in \cat A$ together with an isomorphism of right $\cat A$-dg-modules
\[
\cat A(-,\varprojlim_n A_n) \xrightarrow{\sim} \varprojlim_n \cat A(-,A_n).
\]
We may describe $\varprojlim_n \cat A(-,A_n)$ as the following kernel taken in $\compdg(\cat A)$:
\[
0 \to \varprojlim_n \cat A(-,A_n) \to \prod_{n \geq 0} \cat A(-,A_n) \xrightarrow{1-\nu} \prod_{n \geq 0} \cat A(-,A_n).
\]
This sequence is not in general exact, namely, $1-\nu$ is not in general surjective. It will be so under suitable additional assumptions, such as in the following proposition:
\begin{proposition} \label{proposition:lim1_vanishing}
In the above setup, assume that the morphism of complexes
\[
(a_{n+1,n})_* \colon \cat A(Z,A_{n+1}) \to \cat A(Z,A_n)
\]
is surjective for all $Z \in \cat A$. Then, the morphism $1-\nu$ is surjective and we have an exact sequence:
\[
0 \to \varprojlim_n \cat A(-,A_n) \to \prod_{n \geq 0} \cat A(-,A_n) \xrightarrow{1-\nu} \prod_{n \geq 0} \cat A(-,A_n) \to 0.
\]
\end{proposition}
\begin{proof}
Surjectivity of $1-\nu$ as a morphism of dg-modules is equivalent to surjectivity of its components
\[
(1-\nu)^p \colon \prod_{n \geq 0} \cat A(Z,A_n)^p \to \prod_{n \geq 0} \cat A(Z,A_n)^p
\]
as homomorphisms of abelian groups. Then, the result is well known (see, for instance, \cite[Lemma 3.53]{weibel-homological}).
\end{proof}
\begin{corollary} \label{corollary:lim1_vanishing}
Assume the setup of the above Proposition \ref{proposition:lim1_vanishing}, in particular that
\[
(a_{n+1,n})_* \colon \cat A(Z,A_{n+1}) \to \cat A(Z,A_n)
\]
is surjective for all $Z \in \cat A$. Then, there is an isomorphism
\[
\varprojlim_n \cat A(-,A_n) \xrightarrow{\approx} \holim_n \cat A(-,A_n)
\]
in the derived category $\dercomp(\cat A)$.

Moreover, assume that the limit of $(A_{n+1} \xrightarrow{a_{n+1,n}} A_n)_{n \geq 0}$ exists in $\cat A$. Then, the object $\varprojlim_n A_n \in \cat A$, together with the quasi-isomorphism
\[
\cat A(-,\varprojlim_n A_n) \xrightarrow{\sim} \varprojlim_n \cat A(-,A_n) \xrightarrow{\approx} \holim_n \cat A(-,A_n)
\]
is a homotopy limit of the given sequence. In particular, we have an isomorphism
\[
\varprojlim_n A_n \cong \holim_n A_n
\]
in $H^0(\cat A)$.
\end{corollary}
\begin{proof}
This follows directly from Proposition \ref{proposition:lim1_vanishing}, since short exact sequences of complexes (and also of $\cat A$-dg-modules) yield distinguished triangles in the derived category.
\end{proof}
Applying the above arguments to the opposite dg-category $\opp{\cat A}$ yields analogue results about colimits and homotopy colimits. The details are left to the reader.

Corollary \ref{corollary:lim1_vanishing} is relevant, because it allows us (at least in good cases) to work with homotopy limits and colimits by using the ordinary limits and colimits, which are easier to deal with.

\subsection{Truncations of dg-categories} \label{subsection:truncations}
Let $V$ be a chain complex. The \emph{(left) truncation} $\tau_{\leq 0}V$ is defined as the chain complex such that:
\begin{align*}
    (\tau_{\leq 0} V)^i &= 0, \qquad i > 0, \\
    (\tau_{\leq 0} V)^0 &= Z^0(V), \\
    (\tau_{\leq 0} V)^i &= V^i, \qquad i < 0,
\end{align*}
with the induced differential. 

We notice that $H^i(\tau_{\leq 0}V)=H^i(V)$ for all $i \leq 0$ and $H^i(\tau_{\leq 0} V)=0$ for all $i > 0$. Moreover, there is natural (injective) chain map
\[
\tau_{\leq 0} V \to V.
\]
\begin{lemma} \label{lemma:truncation_exact}
The left truncation of complexes is compatible with direct sums and direct products. Namely, we have natural isomorphisms
\begin{align*}
    \tau_{\leq 0} \prod_i V_i & \xrightarrow{\sim} \prod_i \tau_{\leq 0} V_i, \\
    \bigoplus_i \tau_{\leq 0} V_i & \xrightarrow{\sim} \tau_{\leq 0} \bigoplus_i V_i,
\end{align*}
for any family $\{V_i : i \in I\}$ of chain complexes.
\end{lemma}
\begin{proof}
This follows from the fact that direct sums and direct products of complexes are exact.
\end{proof}
\begin{definition}
Let $\cat A$ be a dg-category. We define the \emph{truncation of $\cat A$} as the dg-category $\tau_{\leq 0} \cat A$ with the same objects of $\cat A$ and hom complexes defined by:
\[
(\tau_{\leq 0} \cat A)(A,B) = \tau_{\leq 0}(\cat A(A,B)),
\]
using the left truncation of complexes.

We denote by
\[
i_{\leq 0} \colon \tau_{\leq 0} \cat A \to \cat A
\]
the natural dg-functor, which is the identity on objects and given by the inclusions $\tau_{\leq 0} \cat A(A,B) \to \cat A(A,B)$ on hom complexes.
\end{definition}

We now check the compatibility of truncations with the closures under zero objects, finite or infinite direct sums or products which we discussed in \S \ref{subsec:addzeroobj} and \S \ref{subsec:addsumprod}.

\begin{lemma} \label{lemma:dgcat_trunc_formalzeroobj}
Let $\cat A$ be a dg-category. There is an isomorphism
\[
\tau_{\leq 0}({\cat A_{\{0\}}}) \cong (\tau_{\leq 0} \cat A)_{\{0\}},
\]
which we will interpret as an identification, dropping parentheses and writing just $\tau_{\leq 0} \cat A_{\{0\}}$. Similarly, for a given dg-functor $u \colon \cat A \to \cat B$, we may identify
\[
\tau_{\leq 0}({u_{\{0\}}}) = (\tau_{\leq 0} u)_{\{0\}}
\]
and write just $\tau_{\leq 0} u_{\{0\}}$.

In particular: if $\cat A$ has strict zero objects, the same is true for $\tau_{\leq 0} \cat A$.
\end{lemma}
\begin{proof}
Straightforward.
\end{proof}
\begin{lemma} \label{lemma:trunc_directsumclosure}
Let $\cat A$ be a dg-category. There is an isomorphism
\[
\tau_{\leq 0}({\cat A^\oplus}) \cong (\tau_{\leq 0} \cat A^\oplus),
\]
which we will interpret as an identification, dropping parentheses and writing just $\tau_{\leq 0} \cat A^\oplus$.

In particular: if $\cat A$ has strict direct sums (and zero objects) the same is true for $\tau_{\leq 0} \cat A$.
\end{lemma}
\begin{proof}
Straightforward.
\end{proof}
\begin{lemma}  \label{lemma:trunc_prodclosure}
Let $\cat A$ be a dg-category and let $\kappa$ be a regular cardinal. If $\cat A$ has strict direct products indexed by sets of cardinality $\leq \kappa$, the same is true for $\tau_{\leq 0} \cat A$. Dually, if $\cat A$ has strict coproducts indexed by sets of cardinality $\leq \kappa$, the same is true for $\tau_{\leq 0} \cat A$.
\end{lemma}
\begin{proof}
We prove the first assertion, the other one following from the same argument applied to $\opp{\cat A}$.

Let $\{A_i : i \in I\}$ be a family of objects in $\cat A$, with $|I|\leq \kappa$. Let $A$ be a product of the $A_i$, together with the isomorphism of right $\cat A$-dg-modules
\[
\cat A(-,A) \xrightarrow{\sim} \prod_{i \in I} \cat A(-,A_i).
\]
Taking truncations and using Lemma \ref{lemma:truncation_exact}, we find an isomorphism of right $\tau_{\leq 0} \cat A$-dg-modules:
\[
\tau_{\leq 0} \cat A(-,A) \xrightarrow{\sim} \tau_{\leq 0} \prod_{i \in I} \cat A(-,A_i) \xrightarrow{\sim} \prod_{i \in I} \tau_{\leq 0} \cat A(-,A_i).
\]
This exhibits $A$ as a product of the $A_i$ in $\tau_{\leq 0} \cat A$.
\end{proof}

\subsection{t-structures and co-t-structures} \label{subsection:tstruct_def}
A \emph{t-structure} on a triangulated category gives a formal way of truncating objects and also yields a cohomology theory. T-structures on (pretriangulated) dg-categories are defined just as t-structures on their (triangulated) homotopy categories.
\begin{definition}
Let $\cat A$ be a pretriangulated dg-category. A \emph{t-structure} on $\cat A$ is a t-structure on $H^0(\cat A)$ in the sense of \cite{beilinson-bernstein-deligne-perverse}.

We shall denote a given t-structure as a pair $(\cat A_{\leq 0}, \cat A_{\geq 0})$. In general, $\cat A_{\leq n}$ and $\cat A_{\geq n}$ denote the full dg-subcategories respectively spanned by the objects of the left and right aisles $H^0(\cat A)_{\leq n}$ and $H^0(\cat A)_{\geq n}$, for $n \in \mathbb Z$.
\end{definition}

The intersection $H^0(\cat A)_{\leq 0} \cap H^0(\cat A)_{\geq 0}$ is a full abelian subcategory of $H^0(\cat A)$ denoted by $H^0(\cat A)^\heartsuit$ and called the \emph{heart} of the t-structure. We denote by
\begin{equation}
H^0_t \colon H^0(\cat A) \to H^0(\cat A)^\heartsuit
\end{equation}
the cohomological functor associated to the t-structure. We also set
\[
H^n_t(-)=H^0_t(-[n]).
\]
We will say that a given t-structure is \emph{non-degenerate} if $A \cong 0$ in $H^0(\cat A)$ is equivalent to $H^n_t(A) \cong 0$ in $H^0(\cat A)^\heartsuit$ for all $n \in \mathbb Z$.

The inclusion of the left aisle $i_{\leq n} \colon H^0(\cat A)_{\leq n} \hookrightarrow H^0(\cat A)$ has a right adjoint
\begin{equation}
    \tau_{\leq n} \colon H^0(\cat A) \to H^0(\cat A)_{\leq n}.
\end{equation}
Analogously, the inclusion of the right aisle $i_{\geq n} \colon H^0(\cat A)_{\geq n} \to H^0(\cat A)$ has a left adjoint
\begin{equation}
    \tau_{\geq n} \colon H^0(\cat A) \to H^0(\cat A)_{\geq n}.
\end{equation}
Moreover, for any $A \in H^0(\cat A)$ and $n \in \mathbb Z$, there is a distinguished triangle
\begin{equation}
    i_{\leq n} \tau_{\leq n} A \to A \to i_{\geq n+1} \tau_{\geq n+1} A;
\end{equation}
the arrows are given respectively by the counit of the adjunction $i_{\leq n} \dashv \tau_{\leq n}$ and by the unit of the adjunction $\tau_{\geq n+1} \dashv i_{\geq n+1}$. We will sometimes ease notation and write $\tau_{\leq n}A$ instead of $i_{\leq n} \tau_{\leq n} A$, or $\tau_{\geq n+1} A$ instead of $i_{\geq n+1} \tau_{\geq n+1} A$.

We will need the following easy lemma:
\begin{lemma} \label{lemma:truncations_iso_complexes}
Let $\cat A$ be a pretriangulated dg-category with a t-structure $(\cat A_{\leq 0}, \cat A_{\geq 0})$. For any $A \in \cat A_{\leq n}$ and $B \in \cat A$ (with $n \in \mathbb Z$), the morphism $\tau_{\leq n} B \to B$ induces an isomorphism
\[
\tau_{\leq 0} \cat A(-,\tau_{\leq n} B) \xrightarrow{\approx} \tau_{\leq 0} \cat A(-,B),
\]
in the derived category $\dercomp(\tau_{\leq 0}\cat A_{\leq n})$.

Analogously, for any $B \in \cat A$ and $C \in \cat A_{\geq n}$, the morphism $B \to \tau_{\geq n} B$ induces an isomorphism
\[
\tau_{\leq 0} \cat A(\tau_{\geq n} B, -) \xrightarrow{\approx} \tau_{\leq 0} \cat A(B, -),
\]
in the derived category $\dercomp(\tau_{\leq 0}\opp{\cat A}_{\geq n})$.
\end{lemma}
\begin{proof}
We prove only the first assertion; the argument for the second one is analogous. Let $i \geq 0$. The cohomology
\[
H^{-i}(\cat A(A,\tau_{\leq n} B) \to H^{-i}(\cat A(A,B))
\]
can be identified with
\begin{equation} \label{equation:tstruct_lefttrunc_adj} \tag{$\ast$}
H^0(\cat A)(A[i], \tau_{\leq n} B) \to H^0(\cat A)(A[i],B).
\end{equation}
Now, $A[i]$ lies in $\cat A_{\leq n}$ because left aisles are closed under nonnegative shifts. Hence, the result follows by observing that \eqref{equation:tstruct_lefttrunc_adj} is an isomorphism, since $\tau_{\leq n} B \to B$ is the counit of the adjunction involving the left truncation $\tau_{\leq n}$.
\end{proof}

In order to completely determine a t-structure on $\cat A$, it is often enough to specify just one of the subcategories $\cat A_{\leq 0}$ or $\cat A_{\geq 0}$:
\begin{proposition}[{\cite[Proposition 1.1]{keller-vossieck-aisles}}] \label{proposition:tstruct_from_aisles}
Let $\cat A$ be a pretriangulated dg-category. Let $\cat A_{\leq 0}$ be a full dg-subcategory of $\cat A$ such that:
\begin{itemize}
    \item $H^0(\cat A_{\leq 0})$ is strictly full, additive and stable under extensions in $H^0(\cat A)$.
    \item $H^0(\cat A_{\leq 0})$ is stable under positive shifts in $H^0(\cat A)$.
    \item The inclusion functor $H^0(\cat A_{\leq 0}) \hookrightarrow H^0(\cat A)$ has a right adjoint.
\end{itemize}
Then, if $\cat A_{\geq 0}$ is the full dg-subcategory of $\cat A$ spanned by the objects of $H^0(\cat A_{\leq 0})^\perp[1]$, we conclude that $(\cat A_{\leq 0}, \cat A_{\geq 0})$ is a t-structure on $\cat A$.

Dually, let $\cat A_{\geq 0}$ be a full dg-subcategory of $\cat A$ such that:
\begin{itemize}
    \item $H^0(\cat A_{\geq 0})$ is strictly full, additive and stable under extensions in $H^0(\cat A)$.
    \item $H^0(\cat A_{\geq 0})$ is stable under negative shifts in $H^0(\cat A)$.
    \item The inclusion functor $H^0(\cat A_{\geq 0}) \hookrightarrow H^0(\cat A)$ has a left adjoint.
\end{itemize}
Then, if $\cat A_{\leq 0}$ is the full dg-subcategory of $\cat A$ spanned by the objects of ${^\perp}\!H^0(\cat A_{\geq 0})[-1]$, we conclude that $(\cat A_{\leq 0}, \cat A_{\geq 0})$ is a t-structure on $\cat A$.
\end{proposition}

Along t-structures, we may endow a pretriangulated dg-category with a \emph{co-t-structure}:
\begin{definition}
Let $\cat A$ be a pretriangulated dg-category. A \emph{co-t-structure} on $\cat A$ is a co-t-structure on the homotopy category $H^0(\cat A)$ in the sense of \cite[Definition 1.1.1]{bondarko-weight}, \cite[Definition 2.4]{paukszello-cotstruct}. See also \cite{jorgensen-cotstruct} for a more recent survey.

We shall denote a given co-t-structure as a pair $(\cat A_{\geq 0}^w, \cat A_{\leq 0}^w)$. In general, $\cat A_{\geq n}^w$ and $\cat A_{\leq n}^w$ denote the full dg-subcategories respectively spanned by the objects of the right and left coaisles $H^0(\cat A)_{\leq n}^w$ and $H^0(\cat A)_{\geq n}^w$, for $n \in \mathbb Z$.

The intersection $H^0(\cat A)^w_{\geq 0} \cap H^0(\cat A)^w_{\leq 0}$ is called the \emph{co-heart} of the co-t-structure.
\end{definition}

Let $(\cat A_{\geq 0}^w, \cat A_{\leq 0}^w)$ be a co-t-structure on $\cat A$ and let $A \in \cat A$.  For $n \in \mathbb Z$, we have a distinguished triangle in $H^0(\cat A)$
\begin{equation}
    \sigma_{\geq n} A \to A \to \sigma_{\leq n-1} A,
\end{equation}
where $\sigma_{\geq n} A \in \cat A_{\geq n}^w$ and $\sigma_{\leq n-1} A \in \cat A_{\leq n-1}^w$. It is worth remarking that, in contrast to the truncations with respect to a t-structure, $\sigma_{\geq n} A$ and $\sigma_{\leq n-1} A$ \emph{do not in general yield functors}.

Sometimes, we have both a t-structure and a co-t-structure on a given dg-category which interact nicely with each other:
\begin{definition} \label{definition:adjacent_cotstruct}
Let $\cat A$ be a pretriangulated dg-category. Moreover, let $(\cat A_{\leq 0}, \cat A_{\geq 0})$ be a t-structure on $\cat A$ and let $(\cat A_{\geq 0}^w, \cat A_{\leq 0}^w)$ be a co-t-structure on $\cat A$. We say that the co-t-structure is \emph{left adjacent} (respectively \emph{right adjacent}) to the t-structure if $\cat A_{\geq 0} = \cat A_{\geq 0}^w$ (respectively if $\cat A_{\leq 0}=\cat A_{\leq 0}^w$).
\end{definition}
This notion of compatibility between t-structures and co-t-strucures is related to ``approximations'' with derived injective or projective objects: this is dealt with in Appendix \ref{appendix:derproj_cotstruct}.

\section{Twisted complexes} \label{section:twistedcomplexes}
\subsection{Basics} \label{subsection:twcomplexes_basics}
We now define the main object of this article, namely, \emph{twisted complexes} on a given dg-category.
\begin{definition}
Let $\varcat A$ be a dg-category, strictly concentrated in nonpositive degrees and with strict zero objects. We define the \emph{dg-category $\Tw(\varcat A)$ of (one-sided, unbounded) twisted complexes on $\varcat A$} as follows:
\begin{itemize}
    \item An object $X^\bullet = (X^i,x_i^j)$ of $\Tw(\varcat A)$ is a sequence $(X^i)_{i \in \mathbb Z}$ of objects of $\varcat A$ together with morphisms
    \[
    x_i^j \colon X^i \to X^j,
    \]
    each of degree $i-j+1$, such that the following equation holds:
    \begin{equation}
        (-1)^j dx_i^j + \sum_k x^j_k x^k_i = 0.
    \end{equation}
    Notice that $x^i_j=0$ whenever $i-j+1>0$. 
    \item A degree $p$ morphism $f \colon (X^i,x_i^j) \to (Y^i,y_i^j)$ is a family of morphisms
    \[
    f_i^j \colon X^i \to Y^j,
    \]
    each of degree $i-j+p$. Notice that $f_i^j=0$ whenever $i-j+p>0$. The differential of $f$ is given by
    \begin{equation}
        (df)_i^j = (-1)^j df_i^j + \sum_k (y^j_k f^k_i - (-1)^p f^j_k x^k_i).
    \end{equation}
    \item Given morphisms $f \colon (X^i, x_i^j) \to (Y^i,y_i^j)$ and $g \colon (Y^i,y_i^j) \to (Z^i, z_i^j)$, its composition is given by:
    \begin{equation}
        (g \circ f)_i^j = \sum_k g^j_k f^k_i.
    \end{equation}
    Identities are the obvious ones.
\end{itemize}

Checking that $\Tw(\varcat A)$ is indeed a dg-category is a little tedious but straightforward.

Finally, if $\varcat B$ is any dg-category (without any additional hypothesis), we set:
\begin{equation} \label{eq:twcomplexes_arbitrarydgcat}
\Tw(\varcat B) = \Tw(\tau_{\leq 0} \varcat B_{\{0\}}),
\end{equation}
recalling Remark \ref{lemma:dgcat_trunc_formalzeroobj}.
\end{definition}

\begin{remark} \label{remark:twcomplex_coefficients_inclusion}
Let $\varcat A$ be a dg-category strictly concentrated in nonpositive degrees and with strict zero objects. Then, there is a natural dg-functor
\begin{equation}
\varcat A \to \Tw(\varcat A),
\end{equation}
which sends an object of $A \in \varcat A$ to the following twisted complexe concentrated in degree $0$:
\[
\cdots \to 0 \to A \to 0 \to \cdots.
\]
It is easy to show that $\varcat A \to \Tw(\varcat A)$ is fully faithful.

If $\varcat A$ is strictly concentrated in nonpositive degrees but does not necessarily have zero objects, we still have a fully faithful dg-functor:
\[
\varcat A \hookrightarrow \varcat A_{\{0\}} \hookrightarrow \Tw(\varcat A).
\]

If $\varcat A$ is \emph{cohomologically} concentrated in nonpositive degrees and has cohomological zero objects, we have a diagram of dg-functors:
\[
\varcat A \xleftarrow{\approx} \tau_{\leq 0} \varcat A \xrightarrow{\approx} \tau_{\leq 0} \varcat A_{\{0\}} \hookrightarrow \Tw(\varcat A).
\]
The arrows marked with $\xrightarrow{\approx}$ are quasi-equivalences, hence we may find a quasi-functor $\varcat A \to \Tw(\varcat A)$, again mapping any object $A \in \varcat A$ to the correspondent twisted complex concentrated in degree $0$, which is fully faithful after taking $H^*(-)$. Therefore, the full dg-fubcategory of $\Tw(\varcat A)$ spanned by twisted complexes concentrated in degree $0$ is quasi-equivalent to $\varcat A$.

From this discussion, it is clear that the definition of $\Tw(\varcat A)$ is meaningful only for dg-categories concentrated in nonpositive degrees, strictly or cohomologically. The addition of formal strict zero objects to $\varcat A$ is useful in order to deal with bounded twisted complexes.
\end{remark}

\begin{remark}
The ``one-sidedness'' of both twisted complexes $(X^i,x^j_i)$ and morphisms $f \colon (X^i,x_i^j) \to (Y^i,y_i^j)$ (namely, $x_i^j=0$ for $i-j+1>0$ and $f_i^j =0$ for $i-j+p>0$) implies that the sums $\sum_k x^j_k x^k_i$ and $\sum_k(y^j_k f^k_i - (-1)^p f^j_k x^k_i)$ are finite, and everything is well-defined.
\end{remark}
\begin{remark}
We can picture a twisted complex as follows:
\[\begin{tikzcd}
	\cdots & {X^i} & {X^{i+1}} & {X^{i+2}} & {X^{i+3}} & {X^{i+4}} & \cdots
	\arrow["{x_{i+2}^{i+3}}", from=1-4, to=1-5]
	\arrow["{x_{i+1}^{i+2}}", from=1-3, to=1-4]
	\arrow["{x_{i+3}^{i+4}}", from=1-5, to=1-6]
	\arrow[from=1-6, to=1-7]
	\arrow[from=1-1, to=1-2]
	\arrow["{x_i^{i+1}}", from=1-2, to=1-3]
	\arrow["{x_i^{i+2}}"{description}, curve={height=-24pt}, from=1-2, to=1-4]
	\arrow["{x_i^{i+3}}"{description}, shift left=1, curve={height=-30pt}, from=1-2, to=1-5]
	\arrow["{x_i^{i+4}}"{description}, shift right=1, curve={height=30pt}, from=1-2, to=1-6]
	\arrow["{x_{i+1}^{i+3}}"{description}, curve={height=12pt}, from=1-3, to=1-5]
	\arrow["{x_{i+1}^{i+4}}"{description}, curve={height=24pt}, from=1-3, to=1-6]
	\arrow["{x_{i+1}^{i+4}}"{description}, shift left=1, curve={height=-30pt}, from=1-3, to=1-6]
\end{tikzcd}\]
To simplify notation, we shall often avoid picturing the ``higher twisted differentials'' $x_i^{i+2}, x_i^{i+3}, \ldots$ ($i \in \mathbb Z)$.

To better understand morphisms of twisted complexes, it is worth visualizing a degree $-1$ morphism:
\[\begin{tikzcd}
	&& \cdots & {X^i} & \cdots \\
	\cdots & {Y^{i-2}} & {Y^{i-1}} & {Y^i} & {Y^{i+1}} & \cdots & {}
	\arrow[from=1-3, to=1-4]
	\arrow[from=1-4, to=1-5]
	\arrow[from=2-1, to=2-2]
	\arrow[from=2-2, to=2-3]
	\arrow[from=2-3, to=2-4]
	\arrow[from=2-4, to=2-5]
	\arrow[from=2-5, to=2-6]
	\arrow["{-1}"{description}, from=1-4, to=2-4]
	\arrow["0"{description}, from=1-4, to=2-3]
	\arrow[""{name=0, anchor=center, inner sep=0}, "{-2}"{description}, from=1-4, to=2-5]
	\arrow[""{name=1, anchor=center, inner sep=0}, draw=none, from=1-4, to=2-7]
	\arrow["\cdots"{description}, Rightarrow, draw=none, from=0, to=1]
\end{tikzcd}\]
and also a degree $1$ morphism:
\[\begin{tikzcd}
	\cdots & {X^i} & \cdots \\
	\cdots & {Y^i} & {Y^{i+1}} & {Y^{i+2}} & {Y^{i+3}} & \cdots & {}
	\arrow[from=1-1, to=1-2]
	\arrow[from=1-2, to=1-3]
	\arrow[from=2-2, to=2-3]
	\arrow[from=2-3, to=2-4]
	\arrow[from=2-4, to=2-5]
	\arrow[from=2-1, to=2-2]
	\arrow["0"{description}, from=1-2, to=2-3]
	\arrow["{-1}"{description}, from=1-2, to=2-4]
	\arrow[from=2-5, to=2-6]
	\arrow[""{name=0, anchor=center, inner sep=0}, "{-2}"{description}, from=1-2, to=2-5]
	\arrow[""{name=1, anchor=center, inner sep=0}, draw=none, from=1-2, to=2-7]
	\arrow["\cdots", Rightarrow, draw=none, from=0, to=1]
\end{tikzcd}\]
We pictured only nonzero components, and listed their degrees on the labels.
\end{remark}

If $X^\bullet = (X^i,x_i^j)$ is an object in $\Tw(\varcat A)$, we may always define its \emph{shift} $X^\bullet[n] =(X[n]^i, x[n]_i^j)$ as follows, for all $n \in \mathbb Z$:
\begin{equation} \label{eq:twistedcomplex_shift}
\begin{split}
    X^\bullet[n]^i &= X^{i+n}, \\
    x[n]_i^j & = (-1)^n x_{i+n}^{j+n}.
\end{split}
\end{equation}
We can easily check:
\begin{equation}
\Hom_{\Tw(\varcat A)}(X^\bullet,Y^\bullet[n]) \cong \Hom_{\Tw(\varcat A)}(X^\bullet[-n],Y^\bullet) \cong \Hom_{\Tw(\varcat A)}(X^\bullet,Y^\bullet)[n].
\end{equation}

The construction $\Tw(-)$ is functorial. Namely, if $u \colon \varcat A \to \varcat B$ is a dg-functor between dg-category strictly concentrated in nonpositive degrees, we can define a dg-functor
\begin{equation}
    \Tw(u) \colon \Tw(\varcat A) \to \Tw(\varcat B)
\end{equation}
as follows:
\begin{itemize}
    \item For any object $X^\bullet = (X^i,x_i^j)$ in $\Tw(\varcat A)$, we set 
    \begin{equation}
    \Tw(u)(X^i,x_i^j) = (u(X^i), u(x_i^j)).
    \end{equation}
    \item If $f \colon (X^i,x_i^j) \to (Y^i,y_i^j)$ is a degree $p$ morphism in $\Tw(\varcat A)$, we define:
    \begin{equation}
    \begin{split}
    \Tw(u)(f) \colon (u(X^i), u(x_i^j)) & \to (u(Y^i), u(y_i^j)), \\
    \Tw(u)(f)_i^j &= u(f_i^j).
    \end{split}
    \end{equation}
\end{itemize}
It is easy to see that $\Tw(1_{\varcat A}) = 1_{\Tw(\varcat A)}$ and that $\Tw(vu)=\Tw(v)\Tw(u)$, for composable dg-functors $u$ and $v$.

If $u \colon \varcat A \to \varcat B$ is a dg-functor between any dg-categories, we set:
\begin{equation}
    \Tw(u) = \Tw(\tau_{\leq 0} u_{\{0\}}).
\end{equation}

In the end, we obtain a functor
\begin{equation}
    \Tw(-) \colon \kat{dgCat} \to \kat{dgCat}.
\end{equation}

We may also define dg-categories of twisted complexes which are bounded from above or below:
\begin{definition} \label{definition:boundedtwistedcomplexes}
Let $\varcat A$ be a dg-category. We define $\Tw^+(\varcat A)$ to be the full dg-subcategory of $\Tw(\varcat A)$ of twisted complexes $(X^i,x_i^j)$ such that $X^n=0$ for $n \ll 0$.

Similarly, we define $\Tw^-(\varcat A)$ to be the full dg-subcategory of $\Tw(\varcat A)$ of twisted complexes $(X^i,x_i^j)$ such that $X^n=0$ for $n \gg 0$.

The mappings $\varcat A \mapsto \Tw^-(\varcat A)$ and $\varcat A \mapsto \Tw^+(\varcat A)$ are functorial in the obvious way.
\end{definition}

Twisted complexes are quite nicely behaved with respect to taking opposites. Namely, we can directly prove the following isomorphisms, which we will view as identifications:
\begin{equation} \label{equation:twistedcomplexes_opposites}
\begin{split}
    \opp{\Tw(\opp{\varcat A})} & \cong \Tw(\varcat A), \\
    \opp{\Tw^-(\opp{\varcat A})} & \cong \Tw^+(\varcat A), \\
    \opp{\Tw^+(\opp{\varcat A})} & \cong \Tw^-(\varcat A).
\end{split}
\end{equation}

\subsection{Brutal truncations of twisted complexes} \label{subsection:twcompl_brutaltruncations}
Throughout this part, we fix a dg-category $\varcat A$ strictly concentrated in nonpositive degrees and with strict zero objects.

\begin{definition}
Let $X^\bullet = (X^i, x_i^j)$ be an object in $\Tw(\varcat A)$. For all $N \in \mathbb Z$, we define a twisted complex $(\sigma_{\geq N}X)^\bullet = ((\sigma_{\geq N}X)^i, (\sigma_{\geq N}x)_i^j)$ as follows:
\begin{equation}
    \begin{split}
        (\sigma_{\geq N}X)^i &= \begin{cases} X^i & i,j \geq N,\\ 0 & \text{otherwise}, \end{cases} \\
        (\sigma_{\geq N}x)_i^j &= \begin{cases} x_i^j & i,j \geq N,\\ 0 & \text{otherwise}. \end{cases}
    \end{split}
\end{equation}
We also definre a twisted complex $(\sigma_{\leq N}X)^\bullet = ((\sigma_{\leq N}X)^i, (\sigma_{\leq N}x)_i^j)$ as follows:
\begin{equation}
    \begin{split}
        (\sigma_{\leq N}X)^i &= \begin{cases} X^i & i,j \leq N,\\ 0 & \text{otherwise}, \end{cases} \\
        (\sigma_{\leq N}x)_i^j &= \begin{cases} x_i^j & i,j \leq N,\\ 0 & \text{otherwise}. \end{cases}
    \end{split}
\end{equation}
\end{definition}
\begin{remark}
We shall sometimes use the following notations, for a given twisted complex $X^\bullet \in \Tw(\varcat A)$ and for integers $n \leq m$:
\begin{align*}
    X^\bullet_{\geq n} & =\sigma_{\geq n}X^\bullet, \\
    X^\bullet_{\leq m} &=\sigma_{\leq m} X^\bullet, \\
    X^\bullet_{[n,m]} &= \sigma_{\geq n} \sigma_{\leq m} X^\bullet = \sigma_{\leq m} \sigma_{\geq n} X^\bullet.
\end{align*}
\end{remark}
\begin{remark} \label{remark:brutaltruncations_Z0_functoriality}
If $f \colon X^\bullet \to Y^\bullet $ is a closed degree $0$ morphism in $\Tw(\varcat A)$, we have induced closed degree $0$ morphisms
\begin{align*}
\sigma_{\leq m} f &= f_{\leq m} \colon X^\bullet_{\leq m} \to Y^\bullet_{\leq m}, \\
\sigma_{\geq n} f &= f_{\geq n} \colon X^\bullet_{\geq n} \to Y^\bullet_{\geq n}, \\
\sigma_{\geq n} \sigma_{\leq m} f = \sigma_{\leq n} \sigma_{\geq m} f &= f_{[m,n]} \colon X^\bullet_{[n,m]} \to Y^\bullet_{[n,m]},
\end{align*}

Moreover, such brutal truncations are functorial in $Z^0(\Tw(\varcat A))$. Namely, we have identities:
\begin{align*}
    (gf)_{\leq m} &= g_{\leq m} f_{\leq m}, \qquad 1_{\leq m} = 1_{X^\bullet_{\leq m}}, \\
    (gf)_{\geq n} &= g_{\geq n} f_{\geq n}, \qquad 1_{\geq n} = 1_{X^\bullet_{\geq n}}, \\
    (gf)_{[n,m]} &= g_{[n,m]} f_{[n,m]}, \qquad 1_{[n,m]} = 1_{X^\bullet_{[n,m]}},
\end{align*}
whenever $f \colon X^\bullet \to Y^\bullet$ and $g \colon Y^\bullet \to Z^\bullet$ closed degree $0$ morphisms in $\Tw(\varcat A)$. The proof of the above identities is straightforward even if a little tedious, and is left to the reader.
\end{remark}
For all $N \in \mathbb Z$, there is a closed degree $0$ morphism
\begin{equation}
\begin{split}
    j_{N,N-1} & \colon (\sigma_{\geq N} X)^\bullet  \to (\sigma_{\geq N-1} X)^\bullet, \\
    (j_{N,N-1})_i^j &= \begin{cases} 1_{X^i} & i=j, \\ 0 & \text{otherwise}. \end{cases}
\end{split}
\end{equation}
Moreover, there is a closed degree $0$ morphism
\begin{equation} \label{equation:twistedcomplex_inclusions}
    \begin{split}
    j_N & \colon (\sigma_{\geq N} X)^\bullet  \to X^\bullet, \\
    (j_N)_i^j &= \begin{cases} 1_{X^i} & i=j, \\ 0 & \text{otherwise}. \end{cases}
\end{split}
\end{equation}

Analogously, for all $N \in \mathbb Z$ there is a closed degree $0$ morphism
\begin{equation}
\begin{split}
    p_{N,N-1} & \colon (\sigma_{\leq N} X)^\bullet  \to (\sigma_{\leq N-1} X)^\bullet, \\
    (p_{N,N-1})_i^j &= \begin{cases} 1_{X^i} & i=j, \\ 0 & \text{otherwise}. \end{cases}
\end{split}
\end{equation}
Moreover, there is a closed degree $0$ morphism
\begin{equation} \label{equation:twistedcomplex_projections}
    \begin{split}
    p_N & \colon X^\bullet  \to (\sigma_{\leq N} X)^\bullet, \\
    (p_N)_i^j &= \begin{cases} 1_{X^i} & i=j, \\ 0 & \text{otherwise}. \end{cases}
\end{split}
\end{equation}

We can picture the above morphisms with the following diagrams:
\begin{equation} \label{eq:inclusions_righttruncations}
\begin{tikzcd}[ampersand replacement=\&]
	{(\sigma_{\geq N} X)^\bullet} \&\& \cdots \& 0 \& {X^N} \& {X^{N+1}} \& \cdots \\
	{(\sigma_{\geq N-1} X)^\bullet} \& \cdots \& 0 \& {X^{N-1}} \& {X^N} \& {X^{N+1}} \& \cdots \\
	{X^\bullet} \& \cdots \& {X^{N-2}} \& {X^{N-1}} \& {X^N} \& {X^{N+1}} \& \cdots
	\arrow[from=1-5, to=1-6]
	\arrow[from=1-6, to=1-7]
	\arrow[from=2-4, to=2-5]
	\arrow[from=2-5, to=2-6]
	\arrow[from=2-6, to=2-7]
	\arrow[from=1-4, to=1-5]
	\arrow[from=1-3, to=1-4]
	\arrow[from=2-3, to=2-4]
	\arrow[from=2-2, to=2-3]
	\arrow[equal, from=1-5, to=2-5]
	\arrow[equal, from=1-6, to=2-6]
	\arrow[from=3-2, to=3-3]
	\arrow[from=3-3, to=3-4]
	\arrow[from=3-4, to=3-5]
	\arrow[from=3-5, to=3-6]
	\arrow[from=3-6, to=3-7]
	\arrow[equal, from=2-4, to=3-4]
	\arrow[equal, from=2-5, to=3-5]
	\arrow[equal, from=2-6, to=3-6]
	\arrow["{j_{N,N-1}}", from=1-1, to=2-1]
	\arrow["{j_{N-1}}", from=2-1, to=3-1]
	\arrow["{j_N}"', curve={height=30pt}, from=1-1, to=3-1]
\end{tikzcd}
\end{equation}
and
\begin{equation} \label{eq:projections_lefttruncations}
\begin{tikzcd}[ampersand replacement=\&]
	{X^\bullet} \& \cdots \& {X^{N-2}} \& {X^{N-1}} \& {X^N} \& {X^{N+1}} \& \cdots \\
	{(\sigma_{\leq N}X)^\bullet} \& \cdots \& {X^{N-2}} \& {X^{N-1}} \& {X^N} \& 0 \& \cdots \\
	{(\sigma_{\leq N-1}X)^\bullet} \& \cdots \& {X^{N-2}} \& {X^{N-1}} \& 0 \& \cdots
	\arrow[from=1-3, to=1-4]
	\arrow[from=1-4, to=1-5]
	\arrow[from=1-2, to=1-3]
	\arrow[from=1-5, to=1-6]
	\arrow[from=1-6, to=1-7]
	\arrow[from=2-3, to=2-4]
	\arrow[from=2-4, to=2-5]
	\arrow[from=2-5, to=2-6]
	\arrow[from=2-2, to=2-3]
	\arrow[from=2-6, to=2-7]
	\arrow[from=3-2, to=3-3]
	\arrow[from=3-3, to=3-4]
	\arrow[from=3-4, to=3-5]
	\arrow[from=3-5, to=3-6]
	\arrow[equal, from=1-3, to=2-3]
	\arrow[equal, from=2-3, to=3-3]
	\arrow[equal, from=1-4, to=2-4]
	\arrow[equal, from=2-4, to=3-4]
	\arrow[equal, from=1-5, to=2-5]
	\arrow["{p_N}", from=1-1, to=2-1]
	\arrow["{p_{N,N-1}}", from=2-1, to=3-1]
	\arrow["{p_{N-1}}"', curve={height=24pt}, from=1-1, to=3-1]
\end{tikzcd}
\end{equation}
It is immediate to see that
\begin{align*}
    j_N & = j_{N-1} \circ j_{N,N-1}, \\
    p_{N-1} &= p_{N,N-1} \circ p_N,
\end{align*}
for all $N \in \mathbb Z$.

We now go on to check that any twisted complex $X^\bullet$ is both direct limit of $((\sigma_{\geq -k}X)^\bullet \xrightarrow{j_{-k,-k-1}} (\sigma_{\geq -k-1}X)^\bullet))_k$ and inverse limit of $((\sigma_{\leq k} X)^\bullet \xrightarrow{p_{k,k-1}} (\sigma_{\leq k-1}X)^\bullet))_k$. We are actually going to prove a stronger statement:
\begin{proposition} \label{prop:brutaltrunc_limitcolimit}
Let $X^\bullet \in \Tw(\varcat A)$ be a twisted complex. Let $Z^\bullet \in \Tw(\varcat A)$ be any twisted complex. Then, there are short exact sequences of chain complexes
\begin{align*}
    0 \to \Tw(\varcat A)(X^\bullet, Z^\bullet) \xrightarrow{{(j_{-k}^*)}_k} \prod_{k \geq 0} \Tw(\varcat A)((\sigma_{\geq -k} X)^\bullet, Z^\bullet) \xrightarrow{1-\mu} \prod_{k \geq 0} \Tw(\varcat A)((\sigma_{\geq -k} X)^\bullet, Z^\bullet) \to 0, \\
    0 \to \Tw(\varcat A)(Z^\bullet, X^\bullet) \xrightarrow{({p_k}_*)_k} \prod_{k\geq 0} \Tw(\varcat A)(Z^\bullet, (\sigma_{\leq k} X)^\bullet) \xrightarrow{1-\mu'} \prod_{k \geq 0} \Tw(\varcat A)(Z^\bullet, (\sigma_{\leq k} X)^\bullet) \to 0,
\end{align*}
natural in $Z^\bullet$. The morphisms $\mu$ and $\mu'$ are defined by:
\begin{align*}
    \mu((f_{-k})_k) &= (f_{-k-1} \circ j_{-k,-k-1})_k, \\
    \mu'((g_k)_k) &= (p_{k+1,k} \circ f_{k+1})_k.
\end{align*}
\end{proposition}
\begin{proof}
Every claim, save for the surjectivity of $1-\mu$ and $1-\mu'$, follows from the following two facts:
\begin{itemize}
    \item For any sequence of (degree $p$) morphisms $f_{-k} \colon (\sigma_{\geq -k}X)^\bullet \to Z^\bullet$ such that $f_{-k} = f_{-k-1} \circ j_{-k,-k-1}$, there is a unique (degree $p$) morphism $f \colon X^\bullet \to Z^\bullet$, such that $f \circ j_{-k}=f_{-k}$ for all $k$.
    \item For any sequence of (degree $p$) morphisms $g_k \colon Z^\bullet \to (\sigma_{\leq k}X)^\bullet$ such that $p_{k+1,k} \circ g_k = g_{k+1}$, there is a unique (degree $p$) morphism $g \colon Z^\bullet \to X^\bullet$, such that $p_k \circ g = g_k$ for all $k$.
\end{itemize}
Both facts can be proved directly.

Now, we prove surjectivity of $1-\mu$ and $1-\mu'$. Recalling Proposition \ref{proposition:lim1_vanishing}, this follows by checking that the morphisms
\begin{align*}
    \Tw(\varcat A)((\sigma_{\geq -k-1} X)^\bullet, Z^\bullet) & \xrightarrow{j_{-k,-k-1}^*} \Tw(\varcat A)((\sigma_{\geq -k} X)^\bullet, Z^\bullet), \\
    \Tw(\varcat A)(Z^\bullet,(\sigma_{\leq k+1} X)^\bullet) & \xrightarrow{{p_{k+1,k}}_*} \Tw(\varcat A)(Z^\bullet,(\sigma_{\leq k} X)^\bullet).
\end{align*}
are surjective for all $k \geq 0$. This follows once we see that $j_{-k,-k-1}$ has a degree $0$ (not necessarily closed) left inverse $s_{-k-1,-k}$, and $p_{k+1,k}$ has a degree $0$ (not closed) right inverse $i_{k,k+1}$:
\begin{align*}
    s_{-k-1,-k} \circ j_{-k,-k-1} &= 1, \\
    p_{k+1,k} \circ i_{k,k+1} &=1.
\end{align*}
The definitions of $s_{-k-1,-k}$ and $p_{k+1,k}$ are clear after contemplation of diagrams \eqref{eq:inclusions_righttruncations} and \eqref{eq:projections_lefttruncations}. Precomposition with $s_{-k-1,-k}$ and postcomposition with $i_{k,k+1}$ yield right inverses of respectively $j_{-k,-k-1}^*$ and ${p_{k+1,k}}_*$.
\end{proof}
\begin{corollary} \label{corollary:brutaltrunc_limitcolimit}
In the setup of the above Proposition \ref{prop:brutaltrunc_limitcolimit}, we have isomorphisms:
\begin{align}
    \Tw(\varcat A)(X^\bullet, Z^\bullet) & \xrightarrow{\approx} \varprojlim_{k \geq 0} \Tw(\varcat A)((\sigma_{\geq -k} X)^\bullet, Z^\bullet), \\
    \Tw(\varcat A)(Z^\bullet, X^\bullet) & \xrightarrow{\approx} \varprojlim_{k \geq 0} \Tw(\varcat A)(Z^\bullet,(\sigma_{\leq k} X)^\bullet).
\end{align}
Hence, we may write
\begin{align*}
    X^\bullet & \cong \varinjlim_k (\sigma_{\geq -k} X)^\bullet, \\
    X^\bullet & \cong \varprojlim_k (\sigma_{\leq k} X)^\bullet.
\end{align*}

Moreover, we have quasi-isomorphisms:
\begin{align}
    \Tw(\varcat A)(X^\bullet, Z^\bullet) & \xrightarrow{\approx} \holim_{k \geq 0} \Tw(\varcat A)((\sigma_{\geq -k} X)^\bullet, Z^\bullet), \\
    \Tw(\varcat A)(Z^\bullet, X^\bullet) & \xrightarrow{\approx} \holim_{k \geq 0} \Tw(\varcat A)(Z^\bullet,(\sigma_{\leq k} X)^\bullet),
\end{align}
natural in $Z^\bullet \in \Tw(\varcat A)$. These quasi-isomorphisms exhibit $X^\bullet$ as the following homotopy limit or colimit:
\begin{align*}
    X^\bullet & \cong \hocolim_k (\sigma_{\geq -k} X)^\bullet, \\
    X^\bullet & \cong \holim_k (\sigma_{\leq k} X)^\bullet.
\end{align*}
\end{corollary}
\begin{proof}
The first part follows from the left exactness of the sequences of Proposition \ref{prop:brutaltrunc_limitcolimit}.

Let us deal with the second part. By definition,
\begin{align*}
\holim_{k \geq 0} \Tw(\varcat A)((\sigma_{\geq -k} X)^\bullet, Z^\bullet) & \cong \cone(1-\mu)[-1], \\
\holim_{k \geq 0} \Tw(\varcat A)(Z^\bullet,(\sigma_{\leq k} X)^\bullet) & \cong \cone(1-\mu')[-1].
\end{align*}
We conclude by recalling that short exact sequences of complexes yield distinguished triangles in the derived category.
\end{proof}
\begin{corollary} \label{corollary:truncations_morphisms_colimits}
Let $f \colon X^\bullet \to Y^\bullet$ be a closed degree $0$ morphism in $\Tw(\varcat A)$. Then, we can recover $f$ by taking limits or colimits along its truncations:
\begin{align*}
    f &= \varinjlim_n f_{\geq -n}, \\
    f &= \varprojlim_n f_{\leq n},
\end{align*}
where $f_{\geq n}$ and $f_{\leq -n}$ are the brutal truncations of $f$, see also Remark \ref{remark:brutaltruncations_Z0_functoriality}.
\end{corollary}
\begin{proof}
By definition of $f_{\geq -n}$ and $f_{\leq n}$, we have (strictly) commutative diagrams:
\[
\begin{tikzcd}[ampersand replacement=\&]
	{X^\bullet_{\leq -n}} \& {X^\bullet_{\leq -n-1}} \& {X^\bullet} \\
	{Y^\bullet_{\leq -n}} \& {Y^\bullet_{\leq -n-1}} \& {Y^\bullet,}
	\arrow["{j_{-n-1}}"', from=1-2, to=1-3]
	\arrow["{f_{\leq -n-1}}"', from=1-2, to=2-2]
	\arrow["{j_{-n-1}}", from=2-2, to=2-3]
	\arrow["f", from=1-3, to=2-3]
	\arrow["{j_{-n,-n-1}}"', from=1-1, to=1-2]
	\arrow["{j_{-n,-n-1}}", from=2-1, to=2-2]
	\arrow["{f_{\leq -n}}"', from=1-1, to=2-1]
	\arrow["{j_n}", curve={height=-12pt}, from=1-1, to=1-3]
	\arrow["{j_n}"', curve={height=12pt}, from=2-1, to=2-3]
\end{tikzcd}
\qquad
\begin{tikzcd}[ampersand replacement=\&]
	{X^\bullet} \& {X^\bullet_{\geq n+1}} \& {X^\bullet_{\geq n}} \\
	{Y^\bullet} \& {Y^\bullet_{\geq n+1}} \& {Y^\bullet_{\geq n}.}
	\arrow["{p_{n+1,n}}"', from=1-2, to=1-3]
	\arrow["{f_{\geq -n}}"', from=1-2, to=2-2]
	\arrow["{p_{n+1,n}}", from=2-2, to=2-3]
	\arrow["{f_{\geq -n-1}}", from=1-3, to=2-3]
	\arrow["{p_{n+1}}"', from=1-1, to=1-2]
	\arrow["{p_{n+1}}", from=2-1, to=2-2]
	\arrow["f"', from=1-1, to=2-1]
	\arrow["{p_n}", curve={height=-12pt}, from=1-1, to=1-3]
	\arrow["{p_n}"', curve={height=12pt}, from=2-1, to=2-3] 
\end{tikzcd} \qedhere
\]
\end{proof}

We can use the brutal truncations to recover any twisted complex as the cone of a suitable morphism. First, we discuss how to compute cones of morphisms of twisted complexes in general.
\begin{lemma} \label{lemma:mappingcone_twistedcomplexes}
Let $\varcat A$ be a dg-category. Let $f \colon X^\bullet = (X^i, x_i^j) \to (Y^i, x_i^j) = Y^\bullet$ be a closed degree $0$ morphism in $\Tw(\varcat A)$. Assume that the (strict) direct sums $X^{i+1} \oplus Y^i$ exist in $\varcat A$. Then, we may define the twisted complex $\cone(f)^\bullet = (\cone(f)^i, c(f)_i^j)$ as follows:
\begin{equation} \label{eq:mappingcone_twistedcomplex}
    \begin{split}
        \cone(f)^i & = X^{i+1} \oplus Y^i, \\
        c(f)_i^j & = \begin{psmallmatrix}
        -x_{i+j+1}^{i+j+1} & 0 \\ f_{i+1}^j & y_i^j
        \end{psmallmatrix}.
    \end{split}
\end{equation}
The twisted complex $\cone(f)^\bullet$ is the cone of $f$, and it fits in the following pretriangle
\[
\begin{tikzcd}[ampersand replacement=\&]
	X^\bullet \& Y^\bullet \& {\cone(f)} \& {X^\bullet[1],}
	\arrow["f", from=1-1, to=1-2]
	\arrow["p", shift left=1, from=1-3, to=1-4]
	\arrow["j", shift left=1, from=1-2, to=1-3]
	\arrow["s", shift left=1, from=1-3, to=1-2]
	\arrow["i", shift left=1, from=1-4, to=1-3]
\end{tikzcd}
\]
where the morphisms $i,p,j,s$ are defined to be componentwise the canonical morphisms associated to the biproduct $X^\bullet[1] \oplus Y^\bullet$, for example the only non-zero components of $p$ are the $p^i_i$, given by the natural projections
\[
p^i_i \colon X^{i+1} \oplus Y^i \to X^{i+1}.
\]
\end{lemma}
\begin{proof}
This is a direct computation which is left to the reader.
\end{proof}

Let $X^\bullet=(X^i,x_i^j)$ be an object in $\Tw(\varcat A)$. For all $n \in \mathbb Z$, we may use the $x_i^j$ to define a closed degree $0$ morphism:
\begin{align*}
\tilde{x}  \colon (\sigma_{\leq n-1}X^\bullet)  [-1]  &\to \sigma_{\geq n} X^\bullet, \\
\tilde{x}_i^j &= x_{i-1}^j.
\end{align*}

\begin{proposition} \label{prop:pretriangle_cotstruct}
For any twisted complex $X^\bullet$ there is a pretriangle in $\Tw(\varcat A)$:
\begin{equation} \label{eq:pretriangle_cotstruct}
    (\sigma_{\leq n-1}X^\bullet)  [-1]  \xrightarrow{\tilde{x}} \sigma_{\geq n} X^\bullet \xrightarrow{j_n} X^\bullet \xrightarrow{p_{n-1}} \sigma_{\leq n-1}X^\bullet,
\end{equation}
where $j_n$ and $p_n$ are described in \eqref{eq:inclusions_righttruncations} and \eqref{eq:projections_lefttruncations}.

Moreover, if $f \colon X^\bullet \to Y^\bullet$ is a closed degree $0$ morphism in $\Tw(\varcat A)$, there is a degree $-1$ morphism
\[
h \colon (\sigma_{\leq n-1}X^\bullet)  [-1] \to \sigma_{\geq n} Y^\bullet
\]
which fits in the following diagram:
\[\begin{tikzcd}[ampersand replacement=\&]
	{(\sigma_{\leq n-1}X^\bullet)[-1]} \& {\sigma_{\geq n}X^\bullet} \& {X^\bullet} \& {\sigma_{\leq n-1}X^\bullet} \\
	{(\sigma_{\leq n-1}Y^\bullet)[-1]} \& {\sigma_{\geq n}Y^\bullet} \& {Y^\bullet} \& {\sigma_{\leq n-1}Y^\bullet,}
	\arrow["{\bar{x}}", from=1-1, to=1-2]
	\arrow["{j_n}", from=1-2, to=1-3]
	\arrow["{p_{n-1}}", from=1-3, to=1-4]
	\arrow["{f_{\leq n-1}}", from=1-4, to=2-4]
	\arrow["f", from=1-3, to=2-3]
	\arrow["{p_{n-1}}"', from=2-3, to=2-4]
	\arrow["{f_{\geq n}}"', from=1-2, to=2-2]
	\arrow["{j_n}"', from=2-2, to=2-3]
	\arrow["{f_{\leq n-1}[-1]}"', from=1-1, to=2-1]
	\arrow["{\bar{y}}"', from=2-1, to=2-2]
	\arrow["h"{description}, dashed, from=1-1, to=2-2]
\end{tikzcd}\]
where $f_{\geq n}$ and $f_{\leq n-1}$ are described in Remark \ref{remark:brutaltruncations_Z0_functoriality}. The middle and rightmost squares of the diagram are strictly commutative; the left square is commutative in $H^0(\Tw(\varcat A))$ up to $dh$.
\end{proposition}
\begin{proof}
The first claim follows from the above Lemma \ref{lemma:mappingcone_twistedcomplexes} by unwinding everything.

The second claim is tedious but straightforward. The morphism $h$ is defined using suitable components of $f$, and the homotopy commutativity of the left square (up to $dh$) can be proved directly.
\end{proof}

It is sometimes useful to construct twisted complexes by taking iterated cones and (co)limits.
\begin{construction} \label{construction:twistedcompl_construction}
Let $X_0^\bullet$ be a twisted complex in $\Tw(\varcat A)$ which is concentrated in nonnegative degrees (namely, $X_0^i = 0$ for $i<0$). Moreover, let $(X^i : i < 0)$ be a sequence of objects of $\varcat A$, which we view as twisted complexes concentrated in degree $0$ abusing notation. We assume we have a degree $0$ morphism of twisted complexes
\[
X^{-1} \to X^\bullet_0.
\]
We can directly check that the twisted complex $X_{-1}^\bullet$ obtained from $X^\bullet_0$ by ``adjoining'' $X^{-1}$ in degree $-1$:
\[
X_{-1}^\bullet = \quad \cdots \to 0 \to X^{-1} \to X_0^0 \to X_0^1 \to X_0^2 \to \cdots
\]
fits in the following pretriangle:
\[
X^{-1} \to X^\bullet_0 \to X^\bullet_{-1} \to X^{-1}[1],
\]
with the obvious inclusion and projection morphisms.

We can iterate this, if we have another closed degree $0$ morphism
\[
X^{-2}[1] \to X^\bullet_{-1},
\]
where $X^{-2}$ is an object in $\varcat A$. We can adjoin $X^{-2}$ in degree $-2$ and define 
\[
X_{-2}^\bullet = \quad \cdots \to 0 \to X^{-2} \to X^{-1} \to X_0^0 \to X_0^1 \to X_0^2 \to \cdots
\]
In general, we will be able to construct twisted complexes $X^\bullet_{-n}$ concentrated in degrees $\geq -n$ and fitting in pretriangles:
\[
X^{-n}[n-1] \to X^\bullet_{-n} \to X^\bullet_{-n+1} \to X^{-n}[n].
\]
By construction, we have strict identities $\sigma_{\geq -k} X^\bullet_{-n} = X^\bullet_{-k}$ whenever $n\geq k$.

It is also possible to define a twisted complex $X^\bullet$ as ``union'' of the $X^\bullet_{-i}$. In degree $i$ we will have the object $X^i$ from the construction ($i<0$) or $X^i_0$ ($i\geq 0$). Abusing notation, we will denote it by $X^i$ in either case:
\begin{align*}
    X^\bullet = \quad \cdots X^i \to X^{i+1} \to X^{i+2} \to \cdots
\end{align*}
The twisted differentials $x_i^j \colon X^i \to X^j$ are directly induced from the ones of $X^\bullet_{-n}$ for $n$ suitably large. One can directly check that everything is well defined and the formula
\[
(-1)^j dx_i^j + \sum_k x^i_k x^k_j = 0
\]
holds, by using that the $X^\bullet_{-n}$ are all twisted complexes. In particular, we have by construction strict identities:
\[
\sigma_{\geq -n} X^\bullet = X^\bullet_{-n},
\]
for $n \geq 0$. Moreover, the twisted complex $X^\bullet$ is the colimit (and also homotopy colimit)  of the system
\[
X_0 \to X_{-1} \to \cdots
\]
This follows directly from Corollary \ref{corollary:brutaltrunc_limitcolimit}.

Clearly, a ``dual'' construction can be made involving truncations $\sigma_{\leq n}$, the suitable ``rotated'' pretriangles (i.e. pretriangles in the opposite category) and (homotopy) limits. We leave the details to the reader.
\end{construction}

\subsection{Isomorphisms of twisted complexes} \label{subsection:twcompl_iso}
In this part we give a characterization of isomorphisms in the homotopy category $H^0(\Tw(\varcat A))$. The dg-category $\varcat A$ will be strictly concentrated in nonpositive degrees and with strict zero objects.
\begin{proposition} \label{proposition:iso_twisted_components}
Let $f \colon X^\bullet \to Y^\bullet$ be a closed degree $0$ morphism in $\Tw(\varcat A)$. Then, $f$ is an isomorphism in $H^0(\Tw(\varcat A))$ (respectively in $Z^0(\Tw(\varcat A))$) if and only if the components $f_i^i \colon X^i \to Y^i$ are isomorphisms in $H^0(\varcat A)$ (respectively in $Z^0(\varcat A)$), for all $i \in \mathbb Z$.
\end{proposition}
\begin{proof}
Proving that if $f$ is an isomorphism in $H^0(\Tw(\varcat A))$ or $Z^0(\Tw(\varcat A))$ then its components $f_i^i$ are isomorphisms in $H^0(\varcat A)$ or $Z^0(\varcat A)$ is straightforward, and we shall concentrate on the other implication. We shall deal with the case of inverses in $H^0(\Tw(\varcat A))$; the case of strict inverses in $Z^0(\Tw(\varcat A))$ is proved along the same lines but with easier arguments, and is left to the reader.

\emph{Step 1.} We first construct, inductively, a system of both left and right homotopy inverses to the truncated morphisms $f_{[-p,0]}$, for all $p \geq 0$, keeping track of the homotopies. If $p=0$, we have by definition that $f_{[0,0]} = f_0^0 \colon X_{[0,0]} \to Y_{[0,0]}$, and $X^\bullet_{[0,0]}$ and $Y^\bullet_{[0,0]}$ are twisted complexes concentrated in degree $0$. Since $f_0^0$ is an isomorphism in $H^0(\varcat A)$, we can find closed degree $0$ morphisms of twisted complexes
\[
g^l_{[0,0]}, g^r_{[0,0]} \colon Y^\bullet_{[0,0]} \to X^\bullet_{[0,0]}
\]
such that
\begin{align*}
    g^l_{[0,0]} \circ f_{[0,0]}  &= 1_{X^0} + d h^l_{[0,0]}, \\
    f_{[0,0]} \circ g^r_{[0,0]} &= 1_{Y^0}+ d h^r_{[0,0]},
\end{align*}
for suitably chosen degree $-1$ morphisms $h^l_{[0,0]}$ and $h^r_{[0,0]}$.
Inductively, we assume that we have defined closed degree $0$ morphisms
\[
g^l_{[-k,0]}, g^r_{[-k,0]} \colon Y^\bullet_{[-k,0]} \to X^\bullet_{[-k,0]}
\]
and degree $-1$ morphisms
\begin{align*}
    h^l_{[-k,0]} \colon X^\bullet_{[-k,0]} & \to X^\bullet_{[-k,0]}, \\
    h^r_{[-k,0]} \colon Y^\bullet_{[-k,0]} & \to Y^\bullet_{[-k,0]}
\end{align*}
for $k=0,\ldots,p-1$, with the following properties:
\begin{itemize}
    \item The following diagrams are strictly commutative:
    \[\begin{tikzcd}[ampersand replacement=\&]
	{Y^\bullet_{[-k,0]}} \& {Y^\bullet_{[-k-1,0]}} \\
	{X^\bullet_{[-k,0]}} \& {X^\bullet_{[-k-1,0]},}
	\arrow[from=1-1, to=1-2]
	\arrow["{g^l_{[-k,0]}}"', shift right=1, from=1-1, to=2-1]
	\arrow["{g^r_{[-k,0]}}", shift left=1, from=1-1, to=2-1]
	\arrow[from=2-1, to=2-2]
	\arrow["{g^l_{[-k-1,0]}}"', shift right=1, from=1-2, to=2-2]
	\arrow["{g^r_{[-k-1,0]}}", shift left=1, from=1-2, to=2-2]
\end{tikzcd}\]
\[
\begin{tikzcd}[ampersand replacement=\&]
	{X^\bullet_{[-k,0]}} \& {X^\bullet_{[-k-1,0]}} \\
	{X^\bullet_{[-k,0]}} \& {X^\bullet_{[-k-1,0]},}
	\arrow[from=1-1, to=1-2]
	\arrow["{h^l_{[-k,0]}}"', from=1-1, to=2-1]
	\arrow[from=2-1, to=2-2]
	\arrow["{h^l_{[-k-1,0]}}", from=1-2, to=2-2]
\end{tikzcd} \qquad
\begin{tikzcd}[ampersand replacement=\&]
	{Y^\bullet_{[-k,0]}} \& {Y^\bullet_{[-k-1,0]}} \\
	{Y^\bullet_{[-k,0]}} \& {Y^\bullet_{[-k-1,0]}.}
	\arrow[from=1-1, to=1-2]
	\arrow["{h^r_{[-k,0]}}"', from=1-1, to=2-1]
	\arrow[from=2-1, to=2-2]
	\arrow["{h^r_{[-k-1,0]}}", from=1-2, to=2-2]
\end{tikzcd}
\]
for $k=0,\ldots,p-1$, where $Y^\bullet_{[-k,0]} \to Y^\bullet_{[-k-1,0]}$ and $X^\bullet_{[-k,0]} \to X^\bullet_{[-k-1,0]}$ are the natural inclusions, cf. \eqref{eq:inclusions_righttruncations}.
\item For $k=0,\ldots,p-1$, we have:
\begin{align*}
    g^l_{[-k,0]} \circ f_{[-k,0]} &= 1 + dh^l_{[-k,0]}, \\
    f_{[-k,0]} \circ g^r_{[-k,0]} &= 1 + dh^r_{[-k,0]}.
\end{align*}
\end{itemize}
We consider the following diagram, where the rows are pretriangles (cf. Proposition \ref{prop:pretriangle_cotstruct}):
\begin{equation} \label{eq:tower_truncations_first}
\begin{tikzcd}[ampersand replacement=\&]
	{X^\bullet_{[-p,-p]}[-1]} \& {X^\bullet_{[-p+1,0]}} \& {X^\bullet_{[-p,0]}} \& {X^\bullet_{[-p,-p]}} \\
	{Y^\bullet_{[-p,-p]}[-1]} \& {Y^\bullet_{[-p+1,0]}} \& {Y^\bullet_{[-p,0]}} \& {Y^\bullet_{[-p,-p]}.}
	\arrow[from=1-1, to=1-2]
	\arrow[from=1-2, to=1-3]
	\arrow[from=1-3, to=1-4]
	\arrow["{f_{[-p,-p]}[-1]}"', from=1-1, to=2-1]
	\arrow["{f_{[-p+1,0]}}"', from=1-2, to=2-2]
	\arrow["{f_{[-p,-p]}}"', from=1-4, to=2-4]
	\arrow["{f_{[-p,0]}}"', from=1-3, to=2-3]
	\arrow[from=2-1, to=2-2]
	\arrow[from=2-2, to=2-3]
	\arrow[from=2-3, to=2-4]
\end{tikzcd}
\end{equation}
We observe that the middle and the right squares of the above diagram are strictly commutative. $f_{[p,p]}$ is invertible in $H^0(\Tw(\varcat A))$ and we may also apply the inductive hypothesis, so we may find closed degree $0$ morphisms
\begin{align*}
g^l_{[p,p]}, g^r_{[p,p]} \colon Y^\bullet_{[-p,-p]} & \to X^\bullet_{[-p,-p]}, \\
g^l_{[-p+1,0]}, g^r_{[-p+1,0]} \colon Y^\bullet_{[-p+1,0]} & \to X^\bullet_{[-p+1,0]}
\end{align*}
and degree $-1$ morphisms
\begin{align*}
    h^l_{[-p,-p]} \colon X^\bullet_{[-p,-p]} & \to X^\bullet_{[-p,-p]}, \\
    h^r_{[-p,-p]} \colon Y^\bullet_{[-p,-p]} & \to Y^\bullet_{[-p,-p]}, \\
    h^l_{[-p+1,0]} \colon X^\bullet_{[-p+1,0]} & \to X^\bullet_{[-p+1,0]}, \\
    h^r_{[-p+1,0]} \colon Y^\bullet_{[-p+1,0]} & \to Y^\bullet_{[-p+1,0]}
\end{align*}
such that
\begin{align*}
    g^l_{[p,p]} \circ f_{[p,p]} & = 1 + dh^l_{[p,p]}, \\
    f_{[p,p]} \circ g^r_{[p,p]} & = 1 + dh^r_{[p,p]}, \\
    g^l_{[-p+1,0]} \circ f_{[-p+1,0]} & = 1 + dh^l_{[-p+1,0]}, \\
     f_{[-p+1,0]} \circ g^r_{[-p+1,0]} & = 1 + dh^r_{[-p+1,0]}.
\end{align*}
Recalling Lemma \ref{lemma:iso_cones} and Remark \ref{remark:iso_cones_boundaries}, we find closed degree $0$ morphisms 
\[
g^l_{[-p,0]}, g^r_{[-p,0]} \colon Y^\bullet_{[-p,0]} \to X^\bullet_{[-p,0]}
\]
and degree $-1$ morphisms
\begin{align*}
    h^l_{[-p,0]} \colon X^\bullet_{[-p,0]} & \to X^\bullet_{[-p,0]}, \\
    h^r_{[-p,0]} \colon Y^\bullet_{[-p,0]} & \to Y^\bullet_{[-p,0]}.
\end{align*}
The morphisms $g^l_{[-p,0]}$ and $g^r_{[-p,0]}$ fit in the following diagram:
\[\begin{tikzcd}[ampersand replacement=\&]
	{Y^\bullet_{[-p,-p]}[-1]} \& {Y^\bullet_{[-p+1,0]}} \& {Y^\bullet_{[-p,0]}} \& {Y^\bullet_{[-p,-p]}} \\
	{X^\bullet_{[-p,-p]}[-1]} \& {X^\bullet_{[-p+1,0]}} \& {X^\bullet_{[-p,0]}} \& {X^\bullet_{[-p,-p]}.}
	\arrow[from=1-1, to=1-2]
	\arrow[from=1-2, to=1-3]
	\arrow[from=1-3, to=1-4]
	\arrow["{g^l_{[-p,-p]}[-1]}"', shift right=1, from=1-1, to=2-1]
	\arrow["{g^r_{[-p,-p]}[-1]}", shift left=1, from=1-1, to=2-1]
	\arrow["{g^l_{[-p+1,0]}}"', shift right=1, from=1-2, to=2-2]
	\arrow["{g^r_{[-p+1,0]}}", shift left=1, from=1-2, to=2-2]
	\arrow["{g^l_{[-p,-p]}}"', shift right=1, from=1-4, to=2-4]
	\arrow["{g^r_{[-p,-p]}}", shift left=1, from=1-4, to=2-4]
	\arrow["{g^l_{[-p,0]}}"', shift right=1, from=1-3, to=2-3]
	\arrow["{g^r_{[-p,0]}}", shift left=1, from=1-3, to=2-3]
	\arrow[from=2-1, to=2-2]
	\arrow[from=2-2, to=2-3]
	\arrow[from=2-3, to=2-4]
\end{tikzcd}\]
The middle and right squares of the above diagram are strictly commutative; we can draw similar diagrams and conclusions for the degree $-1$ morphisms $h^l_{[-p,0]}$ and $h^r_{[-p,0]}$. Moreover, we have 
\begin{align*}
     g^l_{[-p,0]} \circ f_{[-p,0]} & = 1 + dh^l_{[-p,0]}, \\
     f_{[-p,0]} \circ g^r_{[-p,0]} & = 1 + dh^r_{[-p,0]}.
\end{align*}
The induction is complete.

\emph{Step 2.} We know from Corollary \ref{corollary:truncations_morphisms_colimits} that
\[
f_{\leq 0} = \varinjlim_p f_{[-p,0]} \colon X^\bullet_{\leq 0} \to Y^\bullet_{\leq 0}.
\]
We may take the colimit of the systems of morphisms $g^l_{[-p,0]}$ and $g^r_{[-p,0]}$ defined in Step 1, defining:
\begin{align*}
g^l_{\leq 0} & = \varinjlim_p g^l_{[-p,0]} \colon Y^\bullet_{\leq 0} \to X^\bullet_{\leq 0}, \\
g^r_{\leq 0} & = \varinjlim_p g^r_{[-p,0]} \colon Y^\bullet_{\leq 0} \to X^\bullet_{\leq 0}.
\end{align*}
The homotopies $h^l_{[-p,0]}$ and $h^r_{[-p,0]}$ are themselves a direct system, by construction. We may set:
\begin{align*}
    h^l_{\leq 0} &= \varinjlim_p h^l_{[-p,0]} \colon X^\bullet_{\leq 0} \to X^\bullet_{\leq 0}, \\
    h^r_{\leq 0} &= \varinjlim_p h^r_{[-p,0]} \colon Y^\bullet_{\leq 0} \to Y^\bullet_{\leq 0},
\end{align*}
By dg-functoriality of the direct limit, we have:
\begin{align*}
    g^l_{\leq 0} \circ f_{\leq 0} &= \varinjlim_p (g^l_{[-p,0]} \circ f_{[-p,0]}) = 1 + d \varinjlim_p h^l_{[-p,0]} = 1+d h^l_{\leq 0}, \\
     f_{\leq 0} \circ g^r_{\leq 0} &= \varinjlim_p (f_{[-p,0]}\circ  g^r_{[-p,0]}) = 1 + d\varinjlim_p h^r_{[-p,0]} = 1+d h^r_{\leq 0}.
\end{align*}

\emph{Step 3.} Next, we construct inductively left and right homotopy inverses of $f_{\leq k}$ for $k\geq 0$, again keeping track of the homotopies. Recalling Proposition \ref{prop:pretriangle_cotstruct}, we have the following diagram (here $k\geq 1$):
\begin{equation} \label{eq:tower_truncations_second}
\begin{tikzcd}[ampersand replacement=\&]
	{X^\bullet_{[k,k]}} \& {X^\bullet_{\leq k}} \& {X^\bullet_{\leq k-1}} \& {X^\bullet_{[k,k]}[1]} \\
	{Y^\bullet_{[k,k]}} \& {Y^\bullet_{\leq k}} \& {Y^\bullet_{\leq k-1}} \& {Y^\bullet_{[k,k]}[1].}
	\arrow[from=1-1, to=1-2]
	\arrow[from=1-2, to=1-3]
	\arrow[from=1-3, to=1-4]
	\arrow["{f_{[k,k]}}"', from=1-1, to=2-1]
	\arrow["{f_{\leq k}}"', from=1-2, to=2-2]
	\arrow["{f_{[k,k][1]}}"', from=1-4, to=2-4]
	\arrow["{f_{\leq k-1}}"', from=1-3, to=2-3]
	\arrow[from=2-1, to=2-2]
	\arrow[from=2-2, to=2-3]
	\arrow[from=2-3, to=2-4]
\end{tikzcd}
\end{equation}
The rows are (``rotated'') pretriangles. The left and middle squares are strictly commutative -- compare with \eqref{eq:tower_truncations_first}, where the middle and the right squares were strictly commutative. The morphisms $X^\bullet_{\leq k} \to X^\bullet_{\leq k-1}$ and $Y^\bullet_{\leq k} \to Y^\bullet_{\leq k-1}$ are the projections described in \eqref{eq:projections_lefttruncations}. Our precise goal is to define, inductively for all $k \geq 0$, closed degree $0$ morphisms:
    \begin{equation*}
        g^l_{\leq k}, g^r_{\leq k} \colon Y^\bullet_{\leq k} \to X^\bullet_{\leq k},
    \end{equation*}
    and degree $-1$ morphisms
    \begin{align*}
        h^l_{\leq k} \colon X^\bullet_{\leq k} & \to X^\bullet_{\leq k}, \\
        h^r_{\leq k} \colon Y^\bullet_{\leq k} & \to Y^\bullet_{\leq k}
    \end{align*}
    such that:
\begin{itemize}
    \item  The following diagrams are strictly commutative ($k \geq 1$):
    \[\begin{tikzcd}[ampersand replacement=\&]
	{Y^\bullet_{\leq k}} \& {Y^\bullet_{\leq k-1}} \\
	{X^\bullet_{\leq k}} \& {X^\bullet_{\leq k-1},}
	\arrow[from=1-1, to=1-2]
	\arrow["{g^l_{\leq k}}"', shift right=1, from=1-1, to=2-1]
	\arrow["{g^r_{\leq k}}", shift left=1, from=1-1, to=2-1]
	\arrow[from=2-1, to=2-2]
	\arrow["{g^l_{\leq k-1}}"', shift right=1, from=1-2, to=2-2]
	\arrow["{g^r_{\leq k-1}}", shift left=1, from=1-2, to=2-2]
\end{tikzcd}\]
\[
\begin{tikzcd}[ampersand replacement=\&]
	{X^\bullet_{\leq k}} \& {X^\bullet_{\leq k-1}} \\
	{X^\bullet_{\leq k}} \& {X^\bullet_{\leq k-1},}
	\arrow[from=1-1, to=1-2]
	\arrow["{h^l_{\leq k}}"', from=1-1, to=2-1]
	\arrow[from=2-1, to=2-2]
	\arrow["{h^l_{\leq k-1}}", from=1-2, to=2-2]
\end{tikzcd} \qquad
\begin{tikzcd}[ampersand replacement=\&]
	{Y^\bullet_{\leq k}} \& {Y^\bullet_{\leq k-1}} \\
	{Y^\bullet_{\leq k}} \& {Y^\bullet_{\leq k-1}.}
	\arrow[from=1-1, to=1-2]
	\arrow["{h^r_{\leq k}}"', from=1-1, to=2-1]
	\arrow[from=2-1, to=2-2]
	\arrow["{h^r_{\leq k-1}}", from=1-2, to=2-2]
\end{tikzcd}
\]
\item We have equalities, for $k \geq 0$:
\begin{align*}
    g^l_{\leq k} \circ f_{\leq k} &= 1 + dh^l_{\leq k}, \\
    f_{\leq k} \circ g^r_{\leq k} &= 1 + dh^r_{\leq k}.
\end{align*}
\end{itemize}
The base step of the induction is precisely the above Step 1. The inductive step is proved with essentially the same argument as in Step 1, using straightforward variants of Lemma \ref{lemma:iso_cones} and Remark \ref{remark:iso_cones_boundaries} (with ``rotated'' pretriangles). The details are left to the reader.

\emph{Step 4.} We argue essentially as in Step 2. We now have inverse systems of left and right homotopy inverses $g^l_{\leq k}, g^r_{\leq k}$ of $f_{\leq k}$, together with systems of homotopies $h^l_{\leq k}$ and $h^r_{\leq k}$. For all $k \geq 0$ we may define closed degree $0$ morphisms:
\begin{align*}
g^l & = \varprojlim_k g^l_{\leq k} \colon Y^\bullet \to X^\bullet, \\
g^r & = \varprojlim_k g^r_{\leq k} \colon Y^\bullet \to X^\bullet,
\end{align*}
and degree $-1$ morphisms:
\begin{align*}
    h^l &= \varprojlim_k h^l_{\leq k} \colon X^\bullet \to X^\bullet, \\
    h^r &= \varprojlim_p h^r_{\leq k} \colon Y^\bullet \to Y^\bullet.
\end{align*}
We recall from Corollary \ref{corollary:truncations_morphisms_colimits} that $f$ is the inverse limit of its truncations $f_{\leq k}$:
\[
f = \varprojlim_k f_{\leq k}.
\]

Using dg-functoriality of inverse limits, we find:
\begin{align*}
    g^l \circ f &= \varprojlim_k (g^l_{\leq k} \circ f_{\leq k}) = 1 + d \varprojlim_k h^l_{\leq k} = 1+d h^l, \\
     f \circ g^r &= \varprojlim_k (f_{\leq k}\circ  g^r_{\leq k}) = 1 + d\varprojlim_k h^r_{\leq k} = 1+d h^r.
\end{align*}
We conclude that $f$ has both a left and right homotopy inverse. Hence, it yields an isomorphism in $H^0(\Tw(\varcat A))$, as claimed.
\end{proof}

\subsection{Twisted complexes and quasi-equivalences} \label{subsection:twcompl_qeq}
We want to prove that $\Tw(\varcat A)$ depends only on the quasi-equivalence class of $\varcat A$. We start with an auxiliary technical lemma and then the preservation of quasi-fully faithful dg-functors.

\begin{lemma} \label{lemma:quasiiso_characterization}
Let $f \colon V^\bullet \to W^\bullet$ be a chain map of complexes. Then, $f$ is a quasi-isomorphism if and only if the following condition holds:
\begin{itemize}
    \item Let $p \in \mathbb Z$ and let $y \in W^p$ and $x' \in V^{p+1}$ such that $dy=f(x')$. Then, there is $z \in W^{p-1}$ and $x \in V^p$ such that:
    \begin{align*}
        dx &= x', \\
        y - dz &= f(x).
    \end{align*}
\end{itemize}
\end{lemma}
\begin{proof}
Let us assume that $f$ is a quasi-isomorphism, and let $y \in W^p$ and $x' \in V^{p+1}$ such that $dy=f(x')$. Since $f(x')$ is a coboundary and $f$ is injective in cohomology, we find $x_0 \in V^p$ such that $dx_0 = x'$. Now, we see that
\[
d(y-f(x_0)) = dy - f(dx_0) = f(x')-f(x')=0,
\]
so $y-f(x_0)$ is a $p$-cocycle in $W^\bullet$. Since $f$ is surjective in cohomology, we find $x'_0 \in V^p$ with $dx'_0 = 0$ and $z \in W^{p-1}$ such that
\[
y - f(x_0) = f(x'_0) + dz.
\]
Taking $x=x_0 +x'_0$ we conclude.

Conversely, assume that the above condition holds. We want to prove that, for $p \in \mathbb Z$, $H^p(f)$ is an isomorphism. First, let $x' \in V^p$ such that $[f(x')] = [0]$, namely
\[
dy = f(x')
\]
for some $y \in W^{p-1}$. Then, from the hypothesis we find $x \in V^{p-1}$ such that $x'=dx$, which means that $[x']=[0]$ and $H^p(f)$ is injective. To prove surjectivity, let $[y] \in H^p(W^\bullet)$. $y$ is a cocycle, so we have
\[
dy=0=f(0).
\]
Applying the hypothesis with $x'=0$, we find $x \in V^p$ such that $dx=0$ and $z \in W^{p-1}$ such that
\[
y-f(x) = dz,
\]
namely $[y]=[f(x)]$.
\end{proof}

\begin{lemma} \label{lemma:tw_quasifullyfaithful}
Let $u \colon \varcat A \to \varcat B$ be a quasi-fully faithful dg-functor. Then, $\Tw(u) \colon \Tw(\varcat A) \to \Tw(\varcat B)$ is quasi-fully faithful.
\end{lemma}
\begin{proof}
Without loss of generality, we may assume that both $\varcat A$ and $\varcat B$ are strictly concentrated in nonpositive degrees. Applying Lemma \ref{lemma:quasiiso_characterization} and using shifts suitably, we reduce to prove the following claim:
\begin{itemize}
    \item Let $A^\bullet = (A^i, a_i^j)$ and $B^\bullet = (B^i, b_i^j)$ be objects in $\Tw(\varcat A)$. Let $g \colon u(A^\bullet) \to u(B^\bullet)$ be a degree $0$ morphism, and let $h \colon A^\bullet \to B^\bullet$ be a degree $1$ morphism such that $dg=u(h)$. Then, there exist a degree $0$ morphism $f \colon A^\bullet \to B^\bullet$ and a degree $-1$ morphism $\alpha \colon u(A^\bullet) \to u(B^\bullet)$ such that:
    \begin{align*}
        df &= h, \\
        g - d\alpha &=u(f).
    \end{align*}
\end{itemize}

We shall define $f$ and $\alpha$ inductively, as follows. First, we set:
\begin{align*}
    f_i^k &=0, \\
    \alpha_i^k &=0,
\end{align*}
for all $i \in \mathbb Z$, for $k < i$. Next, we suppose we have $f_i^{i+p}$ and $\alpha_i^{i+p}$ for all $p < n$, for $n=0,1,\ldots$, for all $i \in \mathbb Z$, such that:
\begin{align}
    (-1)^{i+p} df_i^{i+p} + b_k^{i+p} f_i^k - f_k^{i+p} a_i^k &= h_i^{i+p}, \tag{$\ast1$} \\
    g_i^{i+p} - ((-1)^{i+p} d \alpha_i^{i+p} + u(b_k^{i+p}) \alpha_i^k + \alpha_k^{i+p} u(a_i^k)) &= u(f_i^{i+p}), \tag{$\ast2$} \label{eq:gij_inductivehypothesis}
\end{align}
where we suppressed the summation symbols, adopting Einstein summation convention. Notice that the components $f_i^j$ and $\alpha_i^j$ appearing in the above expressions are the ones already known by inductive hypothesis.

Next, we try to define $f_i^{i+n}$ and $\alpha_i^{i+n}$ satisfying the suitable relations. First, we compute:
\begin{align*}
    u(h)_i^{i+n} &= (dg)_i^{i+n} \\
    &= (-1)^{i+n}dg_i^{i+n} + u(b_k^{i+n}) g_i^k - g_k^{i+n} u(a_i^k).
\end{align*}
We may substitute $g_i^k$ using \eqref{eq:gij_inductivehypothesis}. We find:
\begin{equation} \label{eq:u(h)i_formula}
    u(h)_i^{i+n} =  (-1)_i^{i+n} dg_i^{i+n} + u(b_k^{i+n})u(f_i^k) - u(f_k^{i+n})u(a_i^k) + X_i^{i+n}, \tag{$\ast3$}
\end{equation}
where
\begin{align*}
    X_i^{i+n} = u(b_k^{i+n})&((-1)^k d\alpha_i^k + u(b_s^k)\alpha_i^s + \alpha_s^k u(a_i^s)) - ((-1)^{i+n} d\alpha_k^{i+n} + u(b_s^{i+n}) \alpha_k^s + \alpha_s^{i+n} u(a_k^s))u(a_i^k) \\
     = (-1)^k & u(b_k^{i+n})d\alpha_i^k  + u(b_k^{i+n})u(b_s^k)\alpha_i^s + u(b_k^{i+n})\alpha_s^k u(a_i^s) \\
    &  -  (-1)^{i+n}d\alpha_k^{i+n} u(a_i^k)  - u(b_s^{i+n})\alpha_k^s u(a_i^k) - \alpha_s^{i+n} u(a_k^s)u(a_i^k).
\end{align*}
The term $u(b_k^{i+n})\alpha_s^k u(a_i^s)$ cancels out, and we may also use the formulas:
\begin{align*}
    (-1)^{i+n}du(b_k^{i+n}) + u(b^{i+n}_s)u(b^s_k)  &=0, \\
    (-1)^{s}du(a^s_i) + u(a^s_k)u(a^k_i) &= 0,
\end{align*}
and we find that:
\begin{align*}
    X_i^{i+n} = (-1)^k & u(b_k^{i+n})d\alpha_i^k - (-1)^{i+n} du(b_s^{i+n})\alpha_i^s - (-1)^{i+n}d\alpha_k^{i+n} u(a_i^k) + (-1)^s \alpha_s^{i+n} du(a_i^s)  \\
    = (-1)&^{i+n-1} (du(b_s^{i+n})\alpha_i^s + (-1)^{k-i-n+1} u(b_k^{i+n})d\alpha_i^k) \\ 
    &+ (-1)^{i+n-1}(d\alpha_k^{i+n}u(a_i^k) + (-1)^{k-i-n+1} \alpha_k^{i+n} du(a_i^k)).
\end{align*}
Applying the Leibniz rule, we finally find out that
\[
X_i^{i+n} = - (-1)^{i+n}d(u(b_k^{i+n})\alpha_i^k + \alpha_k^{i+n}u(a_i^k)).
\]
We substitute this in \eqref{eq:u(h)i_formula} and we find:
\begin{equation} \label{eq:u(h)i_formula2}
    u(h)_i^{i+n} = u(b_k^{i+n} f_i^k - f_k^{i+n} a_i^k) + (-1)^{i+n} d(g_i^{i+n} - u(b_k^{i+n})\alpha_i^k - \alpha_k^{i+n}u(a_i^k)) \tag{$\ast4$}
\end{equation}
Using that $u$ is quasi-fully faithful, we deduce that
\begin{equation} \label{eq:hi_formula3}
h_i^{i+n} = (-1)^{i+n}d\phi_i^{i+n} + b_k^{i+n}f_i^k - f_k^{i+n}a_i^k, \tag{$\ast5$}
\end{equation}
for some $\phi_i^{i+n}$. We apply $u$ to the above equation and compare the result with \eqref{eq:u(h)i_formula2}. We find:
\[
du(\phi_i^{i+n}) = d(g_i^{i+n} - (u(b_k^{i+n})\alpha_i^k + \alpha_k^{i+n} u(a_i^k))).
\]
Using again that $u$ is quasi-fully faithful, we find $x_i^{i+n}$ such that $dx_i^{i+n}=0$ and
\[
u(\phi_i^{i+n} - x_i^{i+n}) = g_i^{i+n} - ((-1)^{i+n}d\alpha_i^{i+n} + u(b_k^{i+n})\alpha_i^k + \alpha_k^{i+n}u(a_i^k)),
\]
for a suitable $\alpha_i^{i+n}$. Finally, by setting $f_i^{i+n} = \phi_i^{i+n}-x_i^{i+n}$ and observing that $d\phi_i^{i+n} = d f_i^{i+n}$, we finally deduce from \eqref{eq:hi_formula3} and the above equation that:
\begin{align*}
    h_i^{i+n} &= (-1)^{i+n}df_i^{i+n} + b_k^{i+n}f_i^k - f_k^{i+n}a_i^k, \\
    u(f_i^{i+n}) & = g_i^{i+n} - ((-1)^{i+n}d\alpha_i^{i+n} + u(b_k^{i+n})\alpha_i^k + \alpha_k^{i+n}u(a_i^k)).
\end{align*}
The induction is now complete.
\end{proof}

\begin{proposition} \label{prop:Tw_preserve_quasiequivalences}
Let $u \colon \varcat A \to \varcat B$ be a quasi-equivalence. Then, $\Tw(u)$ is also a quasi-equivalence.
\end{proposition}
\begin{proof}
By Lemma \ref{lemma:tw_quasifullyfaithful}, we already know that $\Tw(u)$ is quasi-fully faithful. We need to prove that
\[
H^0(\Tw(u)) \colon H^0(\Tw(\varcat A)) \to H^0(\Tw(\varcat B))
\]
is essentially surjective. We will sometimes abuse notation and write $u$ instead of $\Tw(u)$. Let $B^\bullet = (B^i,b_i^j) \in \Tw(\varcat B)$. Using that $H^0(u)$ is essentially surjective, we may choose an isomorphism
\[
f_{0,0} \colon B^0 = B^\bullet_{[0,0]} \to u(A^\bullet_{0,0})
\]
in $H^0(\Tw(\varcat A))$, where $A^\bullet_{0,0}=A^0$ is a twisted complex concentrated in degree $0$. Consider the following diagram:
\begin{equation} \label{equation:esssurj_proof_1} \tag{$\ast$}
\begin{tikzcd}[ampersand replacement=\&]
	{B^1[-1]} \& {B^\bullet_{[0,1]}} \& {B^\bullet_{[0,0]}} \& {B^1} \\
	{u(A^1)[-1]} \& {u(A^\bullet_{0,1})} \& {u(A^\bullet_{0,0})} \& {u(A^1).}
	\arrow[from=1-1, to=1-2]
	\arrow[from=1-2, to=1-3]
	\arrow[from=1-3, to=1-4]
	\arrow[from=1-1, to=2-1]
	\arrow[from=2-1, to=2-2]
	\arrow[from=2-2, to=2-3]
	\arrow[dashed, from=2-3, to=2-4]
	\arrow[dashed, from=1-4, to=2-4]
	\arrow["{f_{0,0}}", from=1-3, to=2-3]
	\arrow["{f_{0,1}}", dotted, from=1-2, to=2-2]
\end{tikzcd}
\end{equation}
The idea is: we find an isomorphism $B^1 \to u(A^1)$ in $H^0(\varcat B)$, for some object $A^1 \in \varcat A$, which we view as a twisted complex concentrated in degree $0$. Then, we may find a morphism $u(A^\bullet_{0,0}) \to u(A^1)$ which makes the rightmost square commutative in $H^0(\Tw(\varcat B))$. We already know that $\Tw(u)$ is quasi-fully faithful, hence we can find a closed degree $0$ morphism
\[
A^\bullet_{0,0} \to A^1
\]
whose cohomology class maps to $u(A^\bullet_{0,0}) \to u(A^1)$. The twisted complex $A^\bullet_{0,1}$ is defined as
\[
A^0 \to A^1,
\]
and it sits in the ``rotated'' pretriangle
\[
A^1[-1] \to A^\bullet_{0,1} \to A^\bullet_{0,0} \to A^1,
\]
recall in particular Construction \ref{construction:twistedcompl_construction}. In the above diagram \eqref{equation:esssurj_proof_1} both rows are ``rotated'' pretriangles; the bottom row is obtained by applying $u=\Tw(u)$. We may choose a degree $-1$ homotopy $B^\bullet_{[0,0]} \to u(A^1)$ detecting the commutativity of the rightmost square in $H^0(\Tw(\varcat B))$, and use this to define a morphism
\[
f_{0,1} \colon B^\bullet_{[0,1]} \to u(A^\bullet_{0,1}),
\]
which will be an isomorphism in $H^0(\Tw(\varcat B))$. In particular, its components $(f_{0,1})_i^i$ are isomorphisms in $H^0(\varcat B)$ (cf. Proposition \ref{proposition:iso_twisted_components}). the middle and left squares of \eqref{equation:esssurj_proof_1} are strictly commutative. We can iterate this construction and construct a twisted complex $A^\bullet_{0,n}$ concentrated in degrees from $0$ to $n$ for all $n \geq 0$, fitting in a diagram:
\[
\begin{tikzcd}[ampersand replacement=\&]
	{B^n[-n]} \& {B^\bullet_{[0,n]}} \& {B^\bullet_{[0,n-1]}} \& {B^n[-n+1]} \\
	{u(A^n)[-n]} \& {u(A^\bullet_{0,n})} \& {u(A^\bullet_{0,n-1})} \& {u(A^n)[-n+1].}
	\arrow[from=1-1, to=1-2]
	\arrow[from=1-2, to=1-3]
	\arrow[from=1-3, to=1-4]
	\arrow[from=1-1, to=2-1]
	\arrow[from=2-1, to=2-2]
	\arrow[from=2-2, to=2-3]
	\arrow[dashed, from=2-3, to=2-4]
	\arrow[dashed, from=1-4, to=2-4]
	\arrow["{f_{0,n-1}}", from=1-3, to=2-3]
	\arrow["{f_{0,n}}", dotted, from=1-2, to=2-2]
\end{tikzcd}
\]
The morphism $f_{0,n}$ is an isomorphism in $H^0(\Tw(\varcat B))$, and in particular its components $(f_{0,n})_i^i$ are isomorphisms in $H^0(\varcat B)$; the middle and left squares of the above diagram are strictly commutative; the morphism $u(A^\bullet_{0,n-1}) \to u(A^n)[-n+1]$ comes from a closed degree $0$ morphism $A^\bullet_{0,n-1} \to A^n[-n+1]$.

Now, again recalling Construction  \ref{construction:twistedcompl_construction}, we can define a twisted complex $A^\bullet_0$, concentrated in nonnegative degrees, such that $\sigma_{\leq n} A^\bullet_0 = A^\bullet_{0,n}$. In a similar fashion, the morphisms $f_{0,n}$ can be directly used to define a closed degree $0$ morphism
\[
f_0 \colon B^\bullet_{\geq 0} \to u(A^\bullet_0),
\]
namely $f_0 = \varprojlim_n f_{0,n}$. By construction, the components $(f_0)_i^i$ are isomorphisms in $H^0(\varcat B)$, hence it is an isomorphism in $H^0(\Tw(\varcat B))$ (Proposition \ref{proposition:iso_twisted_components}).

Next, we use a similar iterative argument in order to construct a family of twisted complexes $A^\bullet_{-n}$, concentrated in degrees $\geq -n$ and fitting in the following diagram:
\[\begin{tikzcd}[ampersand replacement=\&]
	{B^{-n}[n-1]} \& {B^\bullet_{\geq -n+1}} \& {B^\bullet_{\geq -n}} \& {B^{-n}[n]} \\
	{u(A^{-n})[n-1]} \& {u(A^\bullet_{-n+1})} \& {u(A^\bullet_{-n})} \& {u(A^{-n})[n].}
	\arrow[from=1-1, to=1-2]
	\arrow[from=1-2, to=1-3]
	\arrow[from=1-3, to=1-4]
	\arrow[dashed, from=1-1, to=2-1]
	\arrow[dashed, from=2-1, to=2-2]
	\arrow[from=2-2, to=2-3]
	\arrow[from=2-3, to=2-4]
	\arrow["{f_{-n+1}}", from=1-2, to=2-2]
	\arrow["{f_{-n}}", dotted, from=1-3, to=2-3]
	\arrow[from=1-4, to=2-4]
\end{tikzcd}\]
The middle and right squares of the above diagram are strictly commutative, the leftmost square is commutative in $H^0(\Tw(\varcat B))$. The vertical arrows are isomorphisms in $H^0(\Tw(\varcat B))$. we may then define a twisted complex $A^\bullet \in \Tw(\varcat A)$ such that $\sigma_{\geq -n} A^\bullet = A^\bullet_{-n}$ (again, see Construction \ref{construction:twistedcompl_construction}). Hence, we may define a morphism
\[
f \colon B^\bullet \to u(A^\bullet),
\]
such that $f= \varinjlim_n f_{-n}$, and in particular the components $f_i^i$ of $f$ are isomorphisms in $H^0(\Tw(\varcat B))$. By Proposition \ref{proposition:iso_twisted_components} we conclude that $f$ is an isomorphism, and this finished the proof.
\end{proof}
\begin{remark}
It is clear that Lemma \ref{lemma:tw_quasifullyfaithful} and Proposition \ref{prop:Tw_preserve_quasiequivalences} are still true if we replace $\Tw(-)$ with $\Tw^+(-)$ or $\Tw^-(-)$.
\end{remark}

\subsection{Products, coproducts and pretriangulated structure on twisted complexes} \label{subsec:tw_addsumprod}

If $\varcat A$ has direct sums, we can also form mapping cones of closed degree $0$ morphisms of twisted complexes:
\begin{proposition} \label{proposition:tw_stronglypretr}
Let $\varcat A$ be a dg-category which has strict finite direct sums (including zero objects). Then, the dg-category $\Tw(\varcat A)$ is strongly pretriangulated.
\end{proposition}
\begin{proof}
The dg-category $\Tw(\varcat A)$ is always closed under shifts. Cones and pretriangles are described in Lemma \ref{lemma:mappingcone_twistedcomplexes}, and they always exist thanks to the fact that $\varcat A$ has strict direct sums.
\end{proof}
The dg-category $\Tw(\varcat A)$ is most interesting when $\varcat A$ has finite direct sums (including zero objects) and is also concentrated in nonpositive degrees. We can also require these properties to hold cohomologically, namely:
\begin{itemize}
    \item $\varcat A$ is such that $H^0(\varcat A)$ is additive.
    \item $\varcat A$ is cohomologically concentrated in nonpositive degrees.
\end{itemize}
\begin{lemma} \label{lemma:dgcat_replace_strictadditive}
Let $\varcat A$ be a dg-category such that $H^0(\varcat A)$ is additive and $\varcat A$ is cohomologically concentrated in nonpositive degrees. Then, there is a dg-category $\varcat A'$ strictly concentrated in nonpositive degrees and having strict finite direct sums and zero objects, and a chain of quasi-equivalences:
\[
\varcat A' \xleftarrow{\approx} \varcat A'' \xrightarrow{\approx} \varcat A,
\]
for a suitable dg-category $\varcat A''$. In particular, applying Proposition \ref{prop:Tw_preserve_quasiequivalences}, we obtain a chain of quasi-equivalences:
\begin{equation}
\Tw(\varcat A') \xleftarrow{\approx} \Tw(\varcat A'') \xrightarrow{\approx} \Tw(\varcat A),
\end{equation}
\end{lemma}
\begin{proof}
We take $\varcat A'' = \tau_{\leq 0} \varcat A$, which is now strictly concentrated in nonpositive degrees. The natural dg-functor
\[
\varcat A'' \to \varcat A
\]
is a quasi-equivalence. Then, we take $\varcat A' = (\varcat A'')^\oplus$, namely, the closure of $\varcat A'$ under strict finite direct sums and strict zero objects. $\varcat A''$ is still strictly concentrated in nonpositive degrees (Lemma \ref{lemma:trunc_directsumclosure}). There is an inclusion dg-functor
\[
\varcat A' \to \varcat A''
\]
which is a quasi-equivalence, since $H^0(\varcat A'')$ is additive (Lemma \ref{lemma:dgcat_strictadditive_closure}).
\end{proof}
\begin{corollary} \label{corollary:TwApretriangulated}
Let $\varcat A$ be a dg-category such that $H^0(\varcat A)$ is additive and $\varcat A$ is cohomologically concentrated in nonpositive degrees. Then, $\Tw(\varcat A)$ is a pretriangulated dg-category.
\end{corollary}
\begin{proof}
Directly applying the above Lemma \ref{lemma:dgcat_replace_strictadditive}, we find a quasi-equivalence $\Tw(\varcat A) \cong \Tw(\varcat A')$, where $\varcat A'$ is strictly concentrated in nonpositive degrees and has strict finite direct sums and zero objects. We know from Proposition \ref{proposition:tw_stronglypretr} that $\Tw(\varcat A')$ is strongly pretriangulated. Hence, $\Tw(\varcat A)$ is pretriangulated.
\end{proof}
Thanks to the above Lemma \ref{lemma:dgcat_replace_strictadditive}, we can work, without loss of generality, with twisted complexes $\Tw(\varcat A)$ on a dg-category $\varcat A$ which is strictly concentrated in nonpositive degrees and has strict direct sums and zero objects.

If $\varcat A$ has strict direct sums or products, then it is immediate to show that the same holds for $\Tw(\varcat A)$:
\begin{lemma} \label{lemma:tw_strict_directsumsprods}
Let $\varcat A$ be a dg-category strictly concentrated in nonpositive degrees and let $\kappa$ be a regular cardinal. Assume that $\varcat A$ has strict direct sums (respectively, strict direct products) indexed by sets of cardinality $\leq \kappa$. Then, the same is true for $\Tw(\varcat A)$.
\end{lemma}
\begin{proof}
Let us deal first with direct sums. Let $\{A^\bullet_s : s \in I\}$ be a family of twisted complexes, where we write $A^\bullet_s = (A^i_s, (a_s)_i^j)$. We claim that $\bigoplus_{s \in I} A^\bullet_s = A^\bullet = (A^i,a_i^j)$ is described termwise:
\begin{align*}
    A^i & = \bigoplus_{s \in I}A^i_s, \\
    a_i^j & = \bigoplus_{s \in I} (a_s)_i^j.
\end{align*}
To check that $A^\bullet$ is a well-defined twisted complex and indeed the direct sum of the $A^\bullet_s$ is straightforward.

The case of direct products is completely analogous.
\end{proof}

Having dealt with strict direct sums, products and cones, and using that taking twisted complexes $\Tw(-)$ preserves quasi-equivalences, we may finally prove \emph{cohomological} closure of twisted complexes under such constructions.
\begin{proposition} \label{proposition:twisted_closed_prodcoprod}
Let $\varcat A$ be a dg-category, and let $\kappa$ be a regular cardinal. We assume that:
\begin{itemize}
    \item $\varcat A$ is cohomologically concentrated in nonpositive degrees.
    \item $H^0(\varcat A)$ is additive.
    \item The graded cohomology $H^*(\varcat A)$ has direct sums (and/or direct products) indexed by sets of cardinality $\leq \kappa$.
\end{itemize}
Then:
\begin{itemize}
    \item $\Tw(\varcat A)$ is a pretriangulated dg-category.
    \item $H^*(\Tw(\varcat A))$ has direct sums (and/or direct products) indexed by sets of cardinality $\leq \kappa$.
\end{itemize}
\end{proposition}
\begin{proof}
The fact that $\Tw(\varcat A)$ is pretriangulated has been already proven in Corollary \ref{corollary:TwApretriangulated}.

Next, assume that $H^*(\varcat A)$ has direct sums and/or direct product as in the hypothesis. In the case of direct sums, we may use Lemma \ref{lemma:homotopyproducts_closure} and replace $\varcat A$ with the quasi-equivalent $\tau_{\leq 0} (\varcat A^\amalg)$, so that $\Tw(\varcat A)$ will be quasi-equivalent to $\Tw(\varcat A^\amalg) = \Tw(\tau_{\leq 0}(\varcat A^\amalg))$. From Lemma \ref{lemma:trunc_prodclosure} we know that $\tau_{\leq 0}(\varcat A^\amalg)$ has strict direct sums, hence we may apply the above Lemma \ref{lemma:tw_strict_directsumsprods} and see that $\Tw(\varcat A^\amalg)$ has strict direct sums. We conclude that $H^*(\Tw(\varcat A)) \cong H^*(\Tw(\varcat A^\amalg))$ has direct sums, as we claimed.

The case of direct products is dealt with similarly, by replacing $\varcat A$ with the quasi-equivalent $\varcat A^\Pi$.
\end{proof}
\begin{remark}
It is clear that every result in this subsection \S \ref{subsec:tw_addsumprod} can be directly adapted to twisted complexes bounded from above or below, namely, $\Tw^-(-)$ and $\Tw^+(-)$.
\end{remark}

\section{t-structures on twisted complexes} \label{section:tstructures}

We will now deal with t-structure and co-t-structures on the category of twisted complexes. We shall fix once and for all a dg-category $\varcat A$ which is cohomologically concentrated in nonpositive degrees and such that $H^0(\varcat A)$ is additive. We know from Corollary \ref{corollary:TwApretriangulated} that $\Tw(\varcat A)$ is a pretriangulated dg-category. Thanks to Lemma \ref{lemma:dgcat_replace_strictadditive}, we will be able to assume that $\varcat A$ is \emph{strictly} concentrated in nonpositive degrees and has strict direct sums and zero objects; in particular, $\Tw(\varcat A)$ will be a strongly pretriangulated dg-category (Proposition \ref{proposition:tw_stronglypretr}).

\subsection{The co-t-structure on $\Tw(\varcat A)$} \label{subsection:cotstruct_twisted}

Unbounded twisted complexes always come with a ``canonical'' co-t-structure given by the brutal truncations of twisted complexes.
\begin{proposition} \label{proposition:twistedcomplexes_cotstruct}
For $n \in \mathbb Z$, we define $\sigma_{\leq n} \Tw(\varcat A)$ and $\sigma_{\geq n} \Tw(\varcat A)$ as the full dg-subcategories of $\Tw(\varcat A)$ respectively spanned by twisted complexes $X^\bullet$ such that $X^i=0$ for $i> n$:
\[
\cdots \to X^{n-1} \to X^n \to 0 \to \cdots,
\]
and by twisted complexes $X^\bullet$ such that $X^i=0$ for $i< n$:
\[
\cdots \to 0 \to X^n \to X^{n+1} \to \cdots.
\]

Next, we define $\Tw(\varcat A)^w_{\leq n}$ and $\Tw(\varcat A)^w_{\geq n}$ respectively as the closures of $\sigma_{\leq n} \Tw(\varcat A)$ and $\sigma_{\geq n} \Tw(\varcat A)$ under isomorphisms in $H^0(\Tw(\varcat A)$. Then, the pair $(\Tw(\varcat A)^w_{\geq 0}, \Tw(\varcat A)^w_{\leq 0})$ is a co-t-structure on $\Tw(\varcat A)$, which we will call the \emph{canonical co-t-structure}. The intersection $\Tw(\varcat A)^w_{\geq 0} \cap \Tw(\varcat A)^w_{\leq 0}$ is quasi-equivalent to $\varcat A$.
\end{proposition}
\begin{proof}
We need to do the following:
\begin{enumerate}
    \item First, we prove that $\Tw(\varcat A)^w_{\leq n}$ is closed under positive shifts and that $\Tw(\varcat A)^w_{\geq n}$ is closed under negative shifts. Then, we check that both $\Tw(\varcat A)^w_{\leq n}$ and $\Tw(\varcat A)^w_{\geq n}$ are closed under extensions, direct sums and direct summands in $H^0(\Tw(\varcat A))$. \label{enumerate:twistedcomplexes_cotstruct_1}
    \item If $X^\bullet \in \Tw(\varcat A)^w_{\geq 0}$ and $Y^\bullet \in \Tw(\varcat A)^w_{\leq -1}$, we prove that
    \[
    H^0(\Tw(\varcat A))(X^\bullet, Y^\bullet) \cong 0.
    \] \label{enumerate:twistedcomplexes_cotstruct_2}
    \item If $X^\bullet \in \Tw(\varcat A)$, we prove the existence of a distinguished triangle
    \[
    \sigma_{\geq 0} X^\bullet \to X^\bullet \to \sigma_{\leq -1} X^\bullet 
    \]
    in $H^0(\Tw(\varcat A))$, where $\sigma_{\geq 0} X^\bullet \in \Tw(\varcat A)^w_{\geq 0}$ and $\sigma_{\leq -1} X^\bullet \in \Tw(\varcat A)^w_{\leq -1}$. \label{enumerate:twistedcomplexes_cotstruct_3}
\end{enumerate}
We may assume that $\varcat A$ is strictly concentrated in nonpositive degrees and has strict direct sums and zero objects, so that $\Tw(\varcat A)$ is strongly pretriangulated (Lemma \ref{lemma:dgcat_replace_strictadditive} and Proposition \ref{proposition:tw_stronglypretr}).

\emph{Step \ref{enumerate:twistedcomplexes_cotstruct_1}.} Direct sums in $\Tw(\varcat A)$ are defined termwise, and mapping cones are described in Proposition \ref{proposition:tw_stronglypretr}. Hence, the only slightly nontrivial claim to prove is closure of $\Tw(\varcat A)^w_{\leq n}$ and $\Tw(\varcat A)^w_{\geq n}$ under direct summands. We assume $n=0$ for simplicity, and we first deal with $\Tw(\varcat A)^w_{\leq 0}$. Let $X^\bullet \in \Tw(\varcat A)^w_{\leq 0}$ such that $X^\bullet \cong X_1^\bullet \oplus X_2^\bullet$ in $H^0(\Tw(\varcat A))$. Up to isomorphism in that homotopy category, we may assume that $X^\bullet$ is strictly concentrated in nonnegative degrees, and that $X_1^\bullet \oplus X_2^\bullet$ is a strict direct sum in $\Tw(\varcat A)$. We immediately see that $X_1^i \cong 0$ and $X_2^i \cong 0$ in $H^0(\varcat A)$ for all $i>0$. Then, consider the natural projections
\[
X^\bullet_k \to \sigma_{\leq 0} X^\bullet_k, \qquad k=1,2,
\]
described in \eqref{equation:twistedcomplex_projections}. Thanks to the characterization of isomorphisms of twisted complexes (Proposition \ref{proposition:iso_twisted_components}), we immediately see that those projections are isomorphisms in $H^0(\Tw(\varcat A))$. We conclude that $X^\bullet_1$ and $X^\bullet_2$ lie in $\Tw(\varcat A)^w_{\leq 0}$, as we claimed. The analogous result for $\Tw(\varcat A)^w_{\geq 0}$ is dealt with analogously, using the inclusions $\sigma_{\geq 0} X^\bullet_k \to X^\bullet_k$ described in \eqref{equation:twistedcomplex_inclusions} when needed.

\emph{Step \ref{enumerate:twistedcomplexes_cotstruct_2}.} Clearly, we may assume that $X^\bullet \in \sigma_{\geq 0} \Tw(\varcat A)$ and $Y^\bullet \in \sigma_{\leq -1} \Tw(\varcat A)$. Then, from the very definition of the morphisms in $\Tw(\varcat A)$, it is clear that there are no nonzero degree $0$ morphisms $X^\bullet \to Y^\bullet$. In particular, $H^0(\Tw(\varcat A))(X^\bullet, Y^\bullet) \cong 0$.

\emph{Step \ref{enumerate:twistedcomplexes_cotstruct_3}.} The existence of the distinguished triangle
\[
\sigma_{\geq 0} X^\bullet \to X^\bullet \to \sigma_{\leq -1} X^\bullet
\]
with the desired properties follows directly from the pretriangle \eqref{eq:pretriangle_cotstruct} (cf. Proposition \ref{prop:pretriangle_cotstruct}).

Finally, the intersection $\Tw(\varcat A)^w_{\geq 0} \cap \Tw(\varcat A)^w_{\leq 0}$ consists, up to isomorphism in $H^0(\Tw(\varcat A))$, of the twisted complexes concentrated in degree $0$. Recalling Remark \ref{remark:twcomplex_coefficients_inclusion}, this is quasi-equivalent to $\varcat A$.
\end{proof}
\subsection{t-structures on $\Tw^+(\varcat A)$ or $\Tw^-(\varcat A)$} \label{subsection:tstruct_bounded_twisted}
If the dg-category $\varcat A$ has suitable properties, we know from \cite{genovese-lowen-vdb-dginj} that twisted complexes which are bounded from above or below (cf. Definition \ref{definition:boundedtwistedcomplexes}) are endowed with t-structures.
\begin{definition}[{see also \cite[\S 5.1]{genovese-lowen-vdb-dginj}}] \label{definition:dgcat_derproj} 
Let $\varcat A$ be a dg-category. We say that $\varcat A$ is a \emph{dg-category of derived projectives} if:
\begin{itemize}
    \item $\varcat A$ is cohomologically concentrated in nonpositive degrees, and $H^0(\varcat A)$ is additive.
    \item $H^0(\varcat A)$ is right coherent, namely, the category $\modfp(H^0(\varcat A))$ of finitely presented right $H^0(\varcat A)$-modules is an abelian subcategory of $\Mod(H^0(\varcat A))$.
    \item For any $A \in \varcat A$ and for any $k \in \mathbb Z$, the right $H^0(\varcat A)$-module $H^k(\varcat A(-,A))$ is finitely presented (namely, it lies in $\modfp(H^0(\varcat A))$).
    \item $H^0(\varcat A)$ is idempotent complete.
\end{itemize}

Dually, we say that $\varcat A$ is a \emph{dg-category of derived injectives} if $\opp{\varcat A}$ is a dg-category of derived projectives.
\end{definition}
The above terminology comes from the fact that, if $\varcat A$ is a dg-category of derived projectives, then $\Tw^-(\varcat A)$ has a t-structure with enough derived projectives, and its derived projectives are quasi-equivalent to $\varcat A$ (cf. \cite[Theorem 7.1]{genovese-lowen-vdb-dginj}. Dually, if $\varcat A$ is a dg-category of derived injectives, then $\Tw^+(\varcat A)$ has a t-structure with enough derived injectives, and its derived injectives are quasi-equivalent to $\varcat A$. The reader can find more informations on the definition and properties of derived projectives or injectives in Appendix \ref{appendix:derproj_cotstruct}; our main results will not directly involve such notions.

The facts that we mentioned in the above discussion can be made a bit more precise:
\begin{proposition} \label{proposition:tstructures_boundedtw}
Let $\varcat A$ be a dg-category of derived projectives. Then, the pretriangulated dg-category $\Tw^-(\varcat A)$ has a non-degenerate t-structure $(\Tw^-(\varcat A)_{\leq 0}, \Tw^-(\varcat A)_{\geq 0})$ such that
\[
\Tw^-(\varcat A)_{\leq 0} = \Tw(\varcat A)^w_{\leq 0},
\]
where $\Tw(\varcat A)^w_{\leq 0}$ is the co-aisle of the canonical co-t-structure on $\Tw(\varcat A)$ discussed in Proposition \ref{proposition:twistedcomplexes_cotstruct}. The heart of such t-structure is equivalent to $\modfp(H^0(\varcat A))$.

Dually, let $\varcat A$ be a dg-category of derived injectives. Then, the pretriangulated dg-category $\Tw^+(\varcat A)$ has a non-degenerate t-structure $(\Tw^+(\varcat A)_{\leq 0}, \Tw^+(\varcat A)_{\geq 0})$ such that
\[
\Tw^+(\varcat A)_{\geq 0} = \Tw(\varcat A)^w_{\geq 0},
\]
where $\Tw(\varcat A)^w_{\geq 0}$ is the co-aisle of the canonical co-t-structure on $\Tw(\varcat A)$ discussed in Proposition \ref{proposition:twistedcomplexes_cotstruct}. The heart of such t-structure is equivalent to $\opp{\modfp(H^0(\opp{\varcat A}))}$.
\end{proposition}
\begin{proof}
We need to check only the first assertion, the other one being dual. The existence of a t-structure on $\Tw^-(\varcat A)$ follows from \cite[Theorem 5.9]{genovese-lowen-vdb-dginj}. The equality $\Tw^-(\varcat A)_{\leq 0} = \Tw(\varcat A)^w_{\leq 0}$ is actually proven in \cite[Lemma 5.12]{genovese-lowen-vdb-dginj} and \cite[Proposition 5.17]{genovese-lowen-vdb-dginj}.
\end{proof}

\subsection{Extending t-structures to unbounded twisted complexes} 
Our main result deals with extending the t-structure on $\Tw^-(\varcat A)$ (or dually $\Tw^+(\varcat A)$) described in Proposition \ref{proposition:tstructures_boundedtw} to a t-structure on the dg-category of unbounded twisted complexes $\Tw(\varcat A)$. The canonical co-t-structure on $\Tw(\varcat A)$ will be left adjacent (or dually right adjacent) to this ``extended'' t-structure on $\Tw(\varcat A)$.
\begin{remark} \label{remark:Twbounded_subcat_strictlyfull_abuse}
When viewing $\Tw^-(\varcat A)$ and $\Tw^+(\varcat A)$ as full dg-subcategories of $\Tw(\varcat A)$, we will abuse notation and identify them with their closures under isomorphisms in $H^0(\Tw(\varcat A))$.
\end{remark}

\begin{theorem} \label{theorem:tstruct_unbounded}
Let $\varcat A$ be a dg-category of derived projectives, and assume that $\tau_{\leq 0} \Tw(\varcat A)^w_{\leq 0}$ is \emph{closed under sequential homotopy limits as a dg-category concentrated in nonpositive degrees}, in the sense that, for any sequence $(X^\bullet_{n+1} \to X^\bullet_n)_{n \geq 0}$ of closed degree $0$ morphisms in $\tau_{\leq 0} \Tw(\varcat A)^w_{\leq 0}$, there is an object $\holim_n X^\bullet_n$ in $\tau_{\leq 0}\Tw(\varcat A)^w_{\leq 0}$ together with an isomorphism
\begin{equation} \label{equation:holim_nonpositive}
\tau_{\leq 0} \Tw(\varcat A)^w_{\leq 0} (-, \holim_n X^\bullet_n) \xrightarrow{\sim} \tau_{\leq 0} \holim_n \tau_{\leq 0} \Tw(\varcat A)^w_{\leq 0} (-, X^\bullet_n)
\end{equation}
in $\dercomp(\tau_{\leq 0} \Tw(\varcat A)^w_{\leq 0})$.  Then, there is a unique t-structure $(\Tw(\varcat A)^{\operatorname{proj}}_{\leq 0}, \Tw(\varcat A)^{\operatorname{proj}}_{\geq 0})$ on $\Tw(\varcat A)$, called the \emph{projective t-stucture}, such that
\[ 
\Tw(\varcat A)^{\operatorname{proj}}_{\leq 0} = \Tw^-(\varcat A)_{\leq 0} = \Tw(\varcat A)^w_{\leq 0}
\]
and the inclusion $\Tw^-(\varcat A) \hookrightarrow \Tw(\varcat A)$ is t-exact. The heart of such t-structure is equivalent to $\modfp(H^0(\varcat A))$.

Dually, let $\varcat A$ be a dg-category of derived injectives, and assume that $\tau_{\leq 0}\Tw(\varcat A)^w_{\geq 0}$ is \emph{closed under sequential homotopy colimits as a dg-category concentrated in nonpositive degrees}, in the sense that, for any sequence $(X^\bullet_n \to X^\bullet_{n+1})_{n \geq 0}$ of closed degree $0$ morphisms in $\tau_{\leq 0} \Tw(\varcat A)^w_{\geq 0}$, there is an object $\hocolim_n X^\bullet_n$ in $\tau_{\leq 0}\Tw(\varcat A)^w_{\geq 0}$ together with an isomorphism
\begin{equation} \label{equation:hocolim_nonpositive}
\tau_{\leq 0} \Tw(\varcat A)^w_{\geq 0} (\hocolim_n X^\bullet_n,-) \xrightarrow{\sim} \tau_{\leq 0} \holim_n \tau_{\leq 0} \Tw(\varcat A)^w_{\geq 0} (X^\bullet_n,-)
\end{equation}
in $\dercomp(\tau_{\leq 0} \opp{(\Tw(\varcat A)^w_{\geq 0})})$. Then, there is a unique t-structure $(\Tw(\varcat A)^{\operatorname{inj}}_{\leq 0}, \Tw(\varcat A)^{\operatorname{inj}}_{\geq 0})$ on $\Tw(\varcat A)$, called the \emph{injective t-stucture}, such that
\[
\Tw(\varcat A)^{\operatorname{inj}}_{\geq 0} = \Tw^+(\varcat A)_{\geq 0} = \Tw(\varcat A)^w_{\geq 0}
\]
and the inclusion $\Tw^+(\varcat A) \hookrightarrow \Tw(\varcat A)$ is t-exact. The heart of such t-structure is equivalent to $\opp{\modfp(H^0(\opp{\varcat A}))}$.
\end{theorem}
\begin{proof}
We show the first claim, the other one being dual. Using Proposition \ref{proposition:tstruct_from_aisles}, we will obtain the desired t-structure once we show that the inclusion 
\[
H^0(\Tw^-(\varcat A)_{\leq 0}) \hookrightarrow H^0(\Tw(\varcat A))
\]
has a right adjoint $\tau_{\leq 0}$. More practically, for a given $X^\bullet \in \Tw(\varcat A)$, we want to find an object $\tau_{\leq 0} X^\bullet \in \Tw^-(\varcat A)_{\leq 0}$ and an isomorphism
\[
\tau_{\leq 0}\Tw(\varcat A)(-,\tau_{\leq 0} X^\bullet) \cong \tau_{\leq 0} \Tw(\varcat A)(-,X^\bullet)
\]
in the derived category $\dercomp(\tau_{\leq 0} \Tw^-(\varcat A)_{\leq 0})$.

The idea is to approximate $X^\bullet$ with a sequence of twisted complexes in $\Tw^-(\varcat A)$, then use the left truncation in $\Tw^-(\varcat A)$. More precisely, we may apply Corollary \ref{corollary:brutaltrunc_limitcolimit} and write
\[
X^\bullet \cong \holim_k (\sigma_{\leq k} X)^\bullet.
\]
Then, $(\sigma_{\leq k} X)^\bullet \in \Tw^-(\varcat A)$ for all $k \geq 0$ and we may apply the left truncation $\tau_{\leq 0}$ of the t-structure on $\Tw^-(\varcat A)$ described in Proposition \ref{proposition:tstructures_boundedtw}. Using the existence of homotopy colimits in $\tau_{\leq 0} \Tw^-(\varcat A)_{\leq 0}$, we may define:
\[
\tau_{\leq 0} X^\bullet = \holim_k \tau_{\leq 0} (\sigma_{\leq k} X)^\bullet.
\]

Then, we have isomorphisms in $\dercomp(\tau_{\leq 0} \Tw^-(\varcat A)_{\leq 0})$:
\begin{align*}
    \tau_{\leq 0} \Tw(\varcat A)(-, \tau_{\leq 0} X^\bullet) & \cong \tau_{\leq 0} \holim_k \tau_{\leq 0} \Tw^-(\varcat A)(-,\tau_{\leq 0} \sigma_{\leq k} X^\bullet) \qquad \text{(cf. \eqref{equation:holim_nonpositive})} \\
    & \cong \tau_{\leq 0} \holim_k \tau_{\leq 0} \Tw^-(\varcat A)(-, \sigma_{\leq k} X^\bullet)  \qquad \text{(cf. Lemma \ref{lemma:truncations_iso_complexes})} \\
    & \cong \tau_{\leq 0} \holim_k \Tw^-(\varcat A)(-, \sigma_{\leq k} X^\bullet)  \qquad \text{(cf. Lemma \ref{lemma:technical_trunc_holim} below)} \\
    & \cong \tau_{\leq 0} \Tw(\varcat A)(-, X^\bullet). \qquad \text{(cf. Corollary \ref{corollary:brutaltrunc_limitcolimit})}
\end{align*}
The second isomorphism above actually involves a comparison between homotopy limits. To explain this in detail, we first notice that we have commutative diagrams in $H^0(\tau_{\leq 0} \Tw^-(\varcat A))$ for $k \geq 0$:
\[\begin{tikzcd}[ampersand replacement=\&]
	{\tau_{\leq 0} \sigma_{\leq k+1}X^\bullet} \& {\sigma_{\leq k+1}X^\bullet} \\
	{\tau_{\leq 0} \sigma_{\leq k}X^\bullet} \& {\sigma_{\leq k}X^\bullet,}
	\arrow[from=1-1, to=1-2]
	\arrow[from=1-1, to=2-1]
	\arrow[from=2-1, to=2-2]
	\arrow[from=1-2, to=2-2]
\end{tikzcd}\]
which, thanks to the (suitably restricted) derived Yoneda embedding, induce commutative diagrams in $\dercomp(\tau_{\leq 0} \Tw^-(\varcat A)_{\leq 0})$:
\[\begin{tikzcd}[ampersand replacement=\&]
	{\tau_{\leq 0}\Tw^-(\varcat A)(-,\tau_{\leq 0} \sigma_{\leq k+1}X^\bullet}) \& {\tau_{\leq 0}\Tw^-(\varcat A)(-,\sigma_{\leq k+1}X^\bullet)} \\
	{\tau_{\leq 0}\Tw^-(\varcat A)(-,\tau_{\leq 0} \sigma_{\leq k}X^\bullet)} \& {\tau_{\leq 0}\Tw^-(\varcat A)(-,\sigma_{\leq k}X^\bullet),}
	\arrow["\sim", from=1-1, to=1-2]
	\arrow[from=1-1, to=2-1]
	\arrow["\sim", from=2-1, to=2-2]
	\arrow[from=1-2, to=2-2]
\end{tikzcd}\]
where the horizontal arrows are isomorphisms. Hence, we obtain a (non unique) isomorphism in $\dercomp(\tau_{\leq 0} \Tw^-(\varcat A)_{\leq 0})$ between the homotopy limits:
\[
\holim_k \tau_{\leq 0} \Tw^-(\varcat A)(-,\tau_{\leq 0} \sigma_{\leq k} X^\bullet) \xrightarrow{\sim} \holim_k \tau_{\leq 0} \Tw^-(\varcat A)(-,\sigma_{\leq k} X^\bullet).
\]

This projective t-structure on $\Tw(\varcat A)$ is uniquely determined by the left aisle $\Tw(\varcat A)_{\leq 0}^{\operatorname{proj}}$: the right aisle $\Tw(\varcat A)^{\operatorname{proj}}_{\geq 0}$ is obtained as the suitable orthogonal.

Next, we show that the inclusion $\Tw^-(\varcat A) \hookrightarrow \Tw(\varcat A)$ is t-exact.. Since $\Tw^-(\varcat A)_{\leq 0} = \Tw(\varcat A)^{\operatorname{proj}}_{\leq 0}$ by construction, we only need to show the inclusion $\Tw^-(\varcat A)_{\geq 0} \subseteq \Tw(\varcat A)^{\operatorname{proj}}_{\geq 0}$. This is immediate: if $X^\bullet \in \Tw^-(\varcat A)_{\geq 0}$, we have
\[
H^0(\Tw(\varcat A))(Y^\bullet, X^\bullet) = 0
\]
for any $Y \in \Tw^-(\varcat A)_{\leq -1} = \Tw(\varcat A)^{\operatorname{proj}}_{\leq -1}$. In turn, this implies that $X^\bullet \in \Tw(\varcat A)^{\operatorname{proj}}_{\geq 0}$.

To conclude, we show that $\Tw^-(\varcat A) \hookrightarrow \Tw(\varcat A)$ restricts to the identity on the hearts. Indeed, we have an equality
\[
\Tw^-(\varcat A)_{\leq 0} \cap \Tw^-(\varcat A)_{\geq 0} = \Tw(\varcat A)^{\operatorname{proj}}_{\leq 0} \cap \Tw(\varcat A)^{\operatorname{proj}}_{\geq 0}
\]
of full dg-subcategories of $\Tw(\varcat A)$. The inclusion $\subseteq$ is clear. On the other hand, if $X^\bullet \in \Tw(\varcat A)^{\operatorname{proj}}_{\leq 0} \cap \Tw(\varcat A)^{\operatorname{proj}}_{\geq 0}$, we have that $X^\bullet \in \Tw^-(\varcat A)_{\leq 0}$ and that for any $Y^\bullet \in \Tw(\varcat A)^{\operatorname{proj}}_{\leq -1} = \Tw^-(\varcat A)_{\leq -1}$:
\[
H^0(\Tw(\varcat A))(Y^\bullet, X^\bullet) = 0.
\]
In particular, $X^\bullet \in \Tw^-(\varcat A)_{\geq 0}$.
\end{proof}
Here is the technical lemma we used in the above proof:
\begin{lemma} \label{lemma:technical_trunc_holim}
Let $\cat A = \tau_{\leq 0} \cat A$ be a dg-category strictly concentrated in nonpositive degrees, and let $(M_{n+1} \to M_n)_{n \geq 0}$ be a sequence of closed degree $0$ morphisms of of right $\cat A$-dg-modules. Taking smart truncations, we have an induced sequence $(\tau_{\leq 0} M_{n+1} \to \tau_{\leq 0} M_n)_{n \geq 0}$ of closed degree $0$ morphisms of right $\cat A$-dg-modules. The natural morphisms $\tau_{\leq 0} M_n \to M_n$ induce an isomorphism
\[
\tau_{\leq 0} \holim_n \tau_{\leq 0} M_n \to \tau_{\leq 0} \holim_n M_n
\]
in $\dercomp(\cat A)$.
\end{lemma}
\begin{proof}
We have a morphism $\holim_n \tau_{\leq 0} M_n \to \holim_n M_n$ in $\dercomp(\cat A)$ which fits in the following morphism of distinguished triangles:
\[\begin{tikzcd}[ampersand replacement=\&]
	{\holim_n \tau_{\leq 0} M_n} \& {\prod_{n \geq 0} \tau_{\leq 0} M_n} \& {\prod_{n \geq 0} \tau_{\leq 0} M_n} \\
	{\holim_n M_n} \& {\prod_{n \geq 0} M_n} \& {\prod_{n \geq 0} M_n.}
	\arrow[from=1-1, to=1-2]
	\arrow["{1-\nu}", from=1-2, to=1-3]
	\arrow[from=1-3, to=2-3]
	\arrow[from=1-2, to=2-2]
	\arrow[dotted, from=1-1, to=2-1]
	\arrow[from=2-1, to=2-2]
	\arrow["{1-\nu}", from=2-2, to=2-3]
\end{tikzcd}\]
Now, applying the five lemma to the induced diagram in cohomology we conclude that
\[
H^{-k} (\holim_n \tau_{\leq 0} M_n) \to H^{-k} (\holim_n M_n)
\]
is an isomorphism for $k \geq 0$. This implies our claim.
\end{proof}

We now discuss some examples where the above Theorem \ref{theorem:tstruct_unbounded} can be applied. The main technical hurdle will usually be the closure under sequential homotopy limits or colimits.
\begin{example} \label{example:dgcat_closed_countable_limits_applythm}
Let $\varcat A$ be a dg-category of derived projectives such that $H^*(\varcat A)$ has countable direct products. Hence, $H^0(\Tw(\varcat A))$ has countable direct products (Proposition \ref{proposition:twisted_closed_prodcoprod}) and we can prove that $\tau_{\leq 0} \Tw(\varcat A)^w_{\leq 0}$ is closed under sequential homotopy colimits as a dg-category concentrated in nonpositive degrees as in \eqref{equation:holim_nonpositive}, concluding that Theorem \ref{theorem:tstruct_unbounded} is applicable and $\Tw(\varcat A)$ can be endowed with the projective t-structure.

To check this, we argue as follows. First, using Lemma \ref{lemma:homotopyproducts_closure}, Lemma \ref{lemma:trunc_prodclosure} and the fact that $\Tw(-)$ preserves quasi-equivalences (Proposition \ref{prop:Tw_preserve_quasiequivalences}), we may assume that $\varcat A$ is strictly concentrated in nonpositive degrees and has strict countable products. Now, let $(X^\bullet_{n+1} \to X^\bullet_n)_{n \geq 0}$ be a sequence of closed degree $0$ morphisms in $\Tw^-(\varcat A)_{\leq 0} = \Tw(\varcat A)^w_{\leq 0}$. Without loss of generality, we may assume that the $X^\bullet_n$ are all strictly concentrated in nonpositive degrees. The homotopy limit of this sequence fits in the following (rotated) pretriangle:
\[
\holim_n X^\bullet_n \to \prod_{n \geq 0} X^\bullet_n \xrightarrow{1-\nu} \prod_{n \geq 0} X^\bullet_n.
\]
The direct product $\prod_{n \geq 0} X^\bullet_n$ is described termwise and it lies in $\Tw(\varcat A)^w_{\leq 0} = \Tw^-(\varcat A)_{\leq 0}$, hence we see that $\holim_n X^\bullet_n$ lies in $\Tw^-(\varcat A)$. We may use the t-structure on $\Tw^-(\varcat A)$ and take the truncation $\tau_{\leq 0} \holim_n X^\bullet_n$. We claim that this objects comes with an isomorphism
\[
\tau_{\leq 0} \Tw(\varcat A)^w_{\leq 0} (-, \tau_{\leq 0} \holim_n X^\bullet_n) \cong \tau_{\leq 0} \holim_n \tau_{\leq 0} \Tw(\varcat A)^w_{\leq 0} (-, X^\bullet_n)
\]
in the derived category, satisfying \eqref{equation:holim_nonpositive}, as desired. Indeed, we have isomorphisms in the derived category:
\begin{align*}
    \tau_{\leq 0} \Tw^-(\varcat A)_{\leq 0} (-, \tau_{\leq 0} \holim_n X^\bullet_n) & \cong \tau_{\leq 0} \Tw^-(\varcat A)_{\leq 0} (-, \holim_n X^\bullet_n) \qquad \text{(cf. Lemma \ref{lemma:truncations_iso_complexes})} \\
    & \cong \tau_{\leq 0} \holim_n \Tw^-(\varcat A)_{\leq 0} (-, X^\bullet_n) \\
    & \cong \tau_{\leq 0} \holim_n \tau_{\leq 0} \Tw^-(\varcat A)_{\leq 0} (-, X^\bullet_n). \qquad \text{(cf. Lemma \ref{lemma:technical_trunc_holim})}
\end{align*}

Clearly, we can dualize the above discussion and prove that, if $\varcat A$ is a derived category of derived injectives such that $H^*(\varcat A)$ has countable under sums, Theorem \ref{theorem:tstruct_unbounded} is applicable and $\Tw(\varcat A)$ is endowed with the injective t-structure.
\end{example}
\begin{example} \label{example:grothendiecktstruct}
Let $\cat A$ be a pretriangulated dg-category endowed with a ``Grothendieck-like'' t-structure as in \cite[Setup 3.1.1]{genovese-ramos-gabrielpopescu}. In particular:
\begin{itemize}
    \item The t-structure $(\cat A_{\leq 0}, \cat A_{\geq 0})$ is non-degenerate.
    \item $H^0(\cat A)$ is well-generated, has arbitrary direct sums and products (cf. \cite[Proposition 8.4.6]{neeman-triangulated}) and the cohomology $H^0_t \colon H^0(\cat A) \to H^0(\cat A)^\heartsuit$ preserves direct sums.
    \item The heart $H^0(\cat A)^\heartsuit$ is a Grothendieck abelian category and $\cat A$ has enough derived injectives.
\end{itemize}
Thanks to the ``reconstruction theorem'' \cite[Theorem 1.3] {genovese-lowen-vdb-dginj}, we have a t-exact quasi-equivalence
\[
\cat A^+ \cong \Tw^+(\varcat J),
\]
where $\varcat J$ is the full dg-subcategory of $\cat A$ spanned by the derived injective objects. $\varcat J$ is a dg-category of derived injectives (cf. \cite[Lemma 6.10]{genovese-lowen-vdb-dginj}), and $\Tw^+(\varcat J)$ is endowed with the t-structure described in Proposition \ref{proposition:tstructures_boundedtw}.

We now prove that $\tau_{\leq 0}\Tw(\varcat J)^w_{\geq 0}$ is closed under sequential homotopy colimits as a dg-category concentrated in nonpositive degrees, so that Theorem \ref{theorem:tstruct_unbounded} is applicable and $\Tw(\varcat J)$ is endowed with the injective t-structure.

To do so, let $(X^\bullet_n \to X^\bullet_{n+1})_{n \geq 0}$ be a sequence of closed degree $0$ morphisms in $\tau_{\leq 0}\Tw(\varcat J)^w_{\geq 0} = \Tw^+(\varcat J)_{\geq 0}$. By our assumptions (in particular: cocompleteness of $H^0(\cat A)$, non-degeneracy and the fact that $H^0_t(-)$ preserves direct sums), we know that the direct sum $\oplus_{n \geq 0} X^\bullet_n$ exists in $\Tw^+(\varcat J)$ and lies in $\Tw^+(\varcat J)_{\geq 0}$. We consider the distinguished triangle
\[
\bigoplus_{n \geq 0} X^\bullet_n \xrightarrow{1-\mu} \bigoplus_{n \geq 0} X^\bullet_n \to \hocolim_n X^\bullet_n.
\]
in $\Tw^+(\varcat J)$. Taking the right truncation $\tau_{\geq 0} \hocolim_n X^\bullet_n$ and arguing as in the above Example \ref{example:dgcat_closed_countable_limits_applythm}, we can find an isomorphism
\[
\tau_{\leq 0} \Tw(\varcat A)^w_{\geq 0} (\tau_{\leq 0} \hocolim_n X^\bullet_n,-) \xrightarrow{\sim} \tau_{\leq 0} \holim_n \tau_{\leq 0} \Tw(\varcat A)^w_{\geq 0} (X^\bullet_n,-)
\]
in the derived category, hence satisfying \eqref{equation:hocolim_nonpositive} as we wanted.
\end{example}
\begin{remark}
In the above Example \ref{example:grothendiecktstruct}, we can take $\cat A = \dercompdg(\mathfrak G)$ to be the derived dg-category of a Grothendieck abelian category $\mathfrak G$, endowed with the natural t-structure. Then, it is not difficult to see that the dg-category of derived injectives of $\dercompdg(\mathfrak G)$ coincides with the (linear) category $\operatorname{Inj}(\mathfrak G)$ of injective objects in $\mathfrak G$.

A direct inspection shows that $\Tw(\operatorname{Inj}(\mathfrak G))$ is just the dg-category $\operatorname{Ch_{dg}}(\operatorname{Inj}(\mathfrak G))$ of complexes of injective objects, and $H^0(\operatorname{Ch_{dg}}(\operatorname{Inj}(\mathfrak G)))$ can be identified with the \emph{homotopy category of injectives} $\mathsf{K}(\operatorname{Inj}(\mathfrak G))$. This is also called the \emph{unseparated derived category $\check{\mathsf{D}}(\mathfrak G)$} of $\mathfrak G$, cf. \cite[C.5.8]{lurie-SAG}. The injective t-structure is such that
\[
\operatorname{Ch_{dg}}(\operatorname{Inj}(\mathfrak G))^{\operatorname{inj}}_{\geq 0} \cong \dercompdg(\mathfrak G)_{\geq 0}.
\]
\end{remark}

\appendix

\section{Derived projectives/injectives and adjacent co-t-structures} \label{appendix:derproj_cotstruct}

In this appendix, we discuss the notions of derived projectives and injectives in triangulated categories endowed with a t-structure, and we prove (Theorem \ref{theorem:derproj_cotstruct}) that they are strictly related to adjacent co-t-structures (see Definition \ref{definition:adjacent_cotstruct}). This result is likely known to experts, albeit perhaps in a different language than ours (cf. \cite{angeleri-marks-vitoria-torsion} or \cite[\S 4]{nicolas-silting}). We still think it could be an interesting addition to the paper and we include it here.

\emph{Derived projectives} and \emph{derived injectives} generalize ordinary projective and injective objects in abelian categories to the framework of t-structures. We shall define them essentially following \cite[\S 5.1]{rizzardo-vdb-nonFM}. Analogous notions have appeared in literature, for example \emph{injective objects in stable $\infty$-categories} (cf. \cite[\S C.5.7]{lurie-SAG} or \emph{Ext-projectives} (cf. \cite{assem-extproj}).
\begin{definition}
Let $\cat T$ be a triangulated category endowed with a t-structure $(\cat T_{\leq 0}, \cat T_{\geq 0})$, and let $P \in \cat \Proj(T^\heartsuit)$ be a projective object in the heart $\cat T^\heartsuit = \cat T_{\leq 0} \cap \cat T_{\geq 0}$. The \emph{derived projective associated to $P$} is an object $S(P)$ which represents the cohomological functor $\cat T^\heartsuit(P,H^0(-)) \colon \cat T \to \Mod(\basering k)$, namely:
\[
\cat T^\heartsuit(P,H^0(-)) \cong \cat T(S(P),-). 
\]
Clearly, if $S(P)$ exists, it is uniquely determined by $P$ up to isomorphism.

An object $Q \in \cat T$ will be called \emph{derived projective} if there is a projective $P \in \Proj(\cat T^\heartsuit)$ such that $Q \cong S(P)$. We will denote by $\DGProj(\cat T)$ the full subcategory of $\cat T$ spanned by the derived projectives.

If for any $P \in \operatorname{Proj}(\cat T^\heartsuit)$ an object $S(P)$ as above exists, we say that \emph{$\cat T$ has derived injectives}. Moreover, if $\cat T$ has derived projectives and the heart $\cat T^\heartsuit$ has enough projectives, we say that \emph{$\cat T$ has enough derived projectives}. 
\end{definition}
\begin{remark}
\emph{Derived injectives} are defined as derived projectives in the opposite triangulated category $\opp{\cat T}$ endowed with the opposite t-structure. For simplicity, we will concentrate on derived projectives, but everything can be dualized to derived injectives in the straightforward way.
\end{remark}
We list some basic properties of derived projectives:
\begin{proposition}[{cf. \cite[Proposition 2.3.3]{genovese-ramos-gabrielpopescu}}] \label{prop:dgproj_basicproperties}
Let $\cat T$ be a triangulated category with a t-structure $(\cat T_{\leq 0}, \cat T_{\geq 0})$, and let $P \in \Proj(\cat T^\heartsuit)$ be a projective object in the heart. Assume that the derived injective $S(P)$ associated to $P$ exists. Then:
	     \begin{enumerate}
	         \item \label{item:dginj_basicproperties_1} $S(P) \in \cat T_{\leq 0}$. 
	         \item \label{item:dginj_basicproperties_2} $H^0(S(P)) \cong P$.
	         \item \label{item:dginj_basicproperties_3} The functor $H^0 \colon \cat T \to \cat T^\heartsuit$ induces an isomorphism
	         \[
	         H^0 \colon \cat T(S(P),A) \xrightarrow{\sim} \cat T^\heartsuit(P,H^0(A)),
	         \]
	         for all $A \in \cat T$.
	     \end{enumerate}
	\end{proposition}

We already mentioned above that derived projectives are essentially the same concept as Ext-projectives (cf. \cite[\S 1]{assem-extproj} for the definition). This is made clearer by the following result:
\begin{proposition}[{cf. \cite[Proposition 2.3.5]{genovese-ramos-gabrielpopescu}}] \label{prop:dgproj_equivalent_definitions}
	    Let $\cat T$ be a triangulated category with a t-structure, and let $Q \in \cat T$ be an object. The following are equivalent:
	    \begin{enumerate}
	        \item $Q$ is a derived projective.
	        \item $Q \in \cat T_{\leq 0}$ and for any $Z \in \cat T_{\leq 0}$ we have
	        \begin{equation}
	            \cat T(Q, Z[1]) \cong 0. \label{eq:dgproj_vanishing}
	        \end{equation}
	    \end{enumerate}
	\end{proposition}

We will be interested in triangulated categories $\cat T$ endowed with a t-structure $(\cat T_{\leq 0}, \cat T_{\geq 0})$ and a left adjacent co-t-structure $(\cat T^w_{\geq 0}, \cat T^w_{\leq 0}=\cat T_{\leq 0})$ (cf. Definition \ref{definition:adjacent_cotstruct}). We recall that the \emph{co-heart} of the co-t-structure $(\cat T^w_{\geq 0}, \cat T^w_{\leq 0})$ is the intersection $\cat T^w_{\geq 0} \cap \cat T^w_{\leq 0}$. As an immediate corollary of the above Proposition \ref{prop:dgproj_equivalent_definitions}, we can prove:
\begin{corollary} \label{corollary:coheart_implies_dgproj}
Let $\cat T$ be a triangulated category with a t-structure $(\cat T_{\leq 0}, \cat T_{\geq 0})$ and a left adjacent co-t-structure $(\cat T^w_{\geq 0}, \cat T^w_{\leq 0}=\cat T_{\leq 0})$. Let $Q \in \cat T^w_{\geq 0} \cap \cat T^w_{\leq 0}$ be an object in the co-heart. Then, $Q$ is a derived projective object. More precisely, $H^0(Q)$ is a projective object in $\cat T^\heartsuit$ and $Q \cong S(H^0(Q))$.
\end{corollary}

Under reasonable assumptions on $\cat T$, the presence of a left adjacent co-t-structure actually implies that the t-structure has enough derived projectives, and the derived projectives coincide with the co-heart.
\begin{theorem} \label{theorem:derproj_cotstruct}
Let $\cat T$ be a triangulated category with a non-degenerate t-structure $(\cat T_{\leq 0}, \cat T_{\geq 0})$ which admits a left adjacent co-t-structure $(\cat T^w_{\geq 0}, \cat T^w_{\leq 0})$, such that the co-heart $\cat T^w_{\geq 0} \cap \cat T^w_{\leq 0}$ is idempotent complete. Then, $\cat T$ has enough derived projectives, and the derived projectives coincide with the co-heart:
\[
\DGProj(\cat T) = \cat T^w_{\geq 0} \cap \cat T^w_{\leq 0}.
\]
\end{theorem}
The proof of Theorem \ref{theorem:derproj_cotstruct} is based on the following two lemmas:
\begin{lemma} \label{lemma:cotstruct_enoughproj}
Let $\cat T$ be a triangulated category with a t-structure $(\cat T_{\leq 0}, \cat T_{\geq 0})$ and a left adjacent co-t-structure $(\cat T^w_{\geq 0}, \cat T^w_{\leq 0})$. Let $A \in \cat T^\heartsuit$, and consider a non-functorial distinguished triangle
\[
\sigma_{\geq 0} A \to A \to \sigma_{\leq -1} A
\]
obtained from the co-t-structure on $\cat T$, where $\sigma_{\geq 0} A \in \cat T^w_{\geq 0}$ and $\sigma_{\leq -1} A \in \cat T^w_{\leq -1}$. Then, $\sigma_{\geq 0} A$ lies in the co-heart $\cat T^w_{\geq 0} \cap \cat T^w_{\leq 0}$, and the morphism $\sigma_{\geq 0} A \to A$ induces an epimorphism
\[
H^0(\sigma_{\geq 0} A) \to A.
\]
In particular, $\sigma_{\geq 0} A$ is a derived projective object and $\cat T^\heartsuit$ has enough projectives.
\begin{proof}
To prove that $\sigma_{\geq 0} A \in \cat T^w_{\leq 0} = \cat T_{\leq 0}$, we fix $Z \in \cat T_{\geq 1}$ and we check that $\cat T(\sigma_{\geq 0} A, Z) = 0$. We have an exact sequence:
\[
\cat T(A,Z) \to \cat T(\sigma_{\geq 0} A, Z) \to \cat T((\sigma_{\leq -1} A)[-1], Z).
\]
Since $A \in \cat T^\heartsuit$, we know that $\cat T(A,Z) =0$. Moreover, $(\sigma_{\leq -1}A)[-1] \in \cat T^w_{\leq 0} = \cat T_{\leq 0}$, hence $\cat T((\sigma_{\leq -1} A)[-1], Z) = 0$. By exactness we conclude that $\cat T(\sigma_{\geq 0} A, Z) = 0$, as claimed.
\end{proof}
\end{lemma}
\begin{lemma} \label{lemma:S(P)}
Let $\cat T$ be a triangulated category with a non degenerate t-structure $(\cat T_{\leq 0}, \cat T_{\geq 0})$ and a left adjacent co-t-structure $(\cat T^w_{\geq 0}, \cat T^w_{\leq 0})$. Let $P \in \Proj(\cat T^\heartsuit)$. Assume that the co-heart $\cat T^w_{\geq 0} \cap \cat T^w_{\leq 0}$ is idempotent complete. Then, there exists $S(P) \in \cat T^w_{\geq 0} \cap \cat T^w_{\leq 0}$ such that $H^0(S(P)) \cong P$.
\end{lemma}
\begin{proof}
Let $P \in \Proj(\cat T^\heartsuit)$. We consider a (non-functorial) distinguished triangle given by the co-t-structure $(\cat T^w_{\geq 0}, \cat T^w_{\leq 0})$:
\[
\sigma_{\geq 0} P \to P \to \sigma_{\leq -1} P.
\]
From Lemma \ref{lemma:cotstruct_enoughproj} we know that $\sigma_{\geq 0} P \in \cat T^w_{\geq 0} \cap \cat T^w_{\leq 0}$ and $H^0(\sigma_{\geq 0} P) \to P$ is an epimorphism. Since $P$ is projective, it has a section $P \to H^0(\sigma_{\geq 0} P)$. From this, we obtain an idempotent $e \colon H^0(\sigma_{\geq 0} P) \to H^0(\sigma_{\geq 0} P)$ such that $P = \ker(e)$. Since $\sigma_{\geq 0} P = S(H^0(\sigma_{\geq 0} P))$ is derived projective (Corollary \ref{corollary:coheart_implies_dgproj}), we may uniquely lift $e$ to an idempotent
\[
\tilde{e} \colon \sigma_{\geq 0} P \to \sigma_{\geq 0} P
\]
such that $H^0(\tilde{e}) = e$. Since the co-heart $\cat T^w_{\geq 0} \cap \cat T^w_{\leq 0}$ is idempotent complete, $\tilde{e}$ has a kernel which is a direct summand of $\sigma_{\geq 0} P$. 

We now set
\[
S(P) = \ker(\tilde{e}).
\]
Being a direct summand of $\sigma_{\geq 0} P$, which lies in the co-heart (which is closed under finite direct sums and summands in $\cat T$), we have that $S(P) \in \cat T^w_{\geq 0} \cap \cat T^w_{\leq 0}$. The idempotent $\tilde{e}$ can be viewed as a projection map 
\[
\begin{psmallmatrix} 0 & 0 \\ 0 & 1 \end{psmallmatrix} \colon S(P) \oplus Q \to S(P) \oplus Q,
\]
where $Q$ is the complement of $S(P)$: $S(P) \oplus Q \cong \sigma_{\geq 0} P$. From this, we see that the (additive) functor $H^0$ preserves the kernel of $\tilde{e}$, and finally we obtain $H^0(S(P)) \cong P$.
\end{proof}

\begin{proof}[Proof of Theorem \ref{theorem:derproj_cotstruct}]
From Lemma \ref{lemma:cotstruct_enoughproj} we know that $\cat T^\heartsuit$ has enough projectives, and from Lemma \ref{lemma:S(P)} we easily see that $\cat T$ has derived projectives. Indeed, from Corollary \ref{corollary:coheart_implies_dgproj} we know that $S(P)$ is derived projective, and for any $X \in \cat T$, we have isomorphisms:
\[
\cat T(S(P),X) \xrightarrow{\sim} \cat T^\heartsuit(H^0(S(P)),H^0(X)) \xrightarrow{\sim} \cat T^\heartsuit(P,H^0(X)),
\]
natural in $X \in \cat T$. Hence, $\cat T$ has enough derived projectives.

From Corollary \ref{corollary:coheart_implies_dgproj}, we already know that $\cat T^w_{\geq 0} \cap \cat T^w_{\leq 0} \subseteq \DGProj(\cat T)$. To see the other inclusion, let $Q \in \DGProj(\cat T)$ be a derived projective. Thanks to Lemma \ref{lemma:S(P)} and the previous part of the proof, we can find an object $S(H^0(Q))$ in the co-heart $\cat T^w_{\geq 0} \cap \cat T^w_{\leq 0}$ which is the derived projective associated to the projective object $H^0(Q)$. By uniqueness, we conclude that $Q \cong S(H^0(Q))$ indeed lies in the co-heart $\cat T^w_{\geq 0} \cap \cat T^w_{\leq 0}$, for it is closed under isomorphisms in $\cat T$.
\end{proof}
\begin{remark}
The ``recostruction theorem'' \cite[Theorem 7.2]{genovese-lowen-vdb-dginj}, combined with Proposition \ref{proposition:tstructures_boundedtw}, can be viewed as a (partial) converse to the above Theorem \ref{theorem:derproj_cotstruct}.
\end{remark}

\bibliographystyle{amsplain}

\begin{thebibliography}{10}

\bibitem{angeleri-marks-vitoria-torsion}
Lidia Angeleri~H\"{u}gel, Frederik Marks, and Jorge Vit\'{o}ria, \emph{Torsion
  pairs in silting theory}, Pacific J. Math. \textbf{291} (2017), no.~2,
  257--278.

\bibitem{assem-extproj}
Ibrahim Assem, Mar\'{\i}a~Jos\'{e} Souto~Salorio, and Sonia Trepode,
  \emph{Ext-projectives in suspended subcategories}, J. Pure Appl. Algebra
  \textbf{212} (2008), no.~2, 423--434.

\bibitem{beilinson-bernstein-deligne-perverse}
Alexander~A. Be\u{\i}linson, Joseph Bernstein, and Pierre Deligne,
  \emph{Faisceaux pervers}, Analysis and topology on singular spaces, {I}
  ({L}uminy, 1981), Ast\'{e}risque, vol. 100, Soc. Math. France, Paris, 1982,
  pp.~5--171.

\bibitem{bondal-kapranov-enhanced}
Alexey~I. Bondal and Mikhail~M. Kapranov, \emph{Enhanced triangulated
  categories}, Math. USSR Sbornik \textbf{70} (1991), no.~1, 93--107.

\bibitem{bondal-larsen-lunts-grothendieck}
Alexey~I. Bondal, Michael Larsen, and Valery~A. Lunts, \emph{Grothendieck ring
  of pretriangulated categories}, Int. Math. Res. Not. (2004), no.~29,
  1461--1495.

\bibitem{bondarko-weight}
Mikhail~V. Bondarko, \emph{Weight structures vs. {$t$}-structures; weight
  filtrations, spectral sequences, and complexes (for motives and in general)},
  J. K-Theory \textbf{6} (2010), no.~3, 387--504.

\bibitem{canonaco-stellari-internalhoms}
Alberto Canonaco and Paolo Stellari, \emph{Internal {H}oms via extensions of dg
  functors}, Adv. Math. \textbf{277} (2015), 100--123.

\bibitem{genovese-lowen-vdb-dginj}
Francesco Genovese, Wendy Lowen, and Michel Van~den Bergh, \emph{t-structures
  and twisted complexes on derived injectives}, Adv. Math. \textbf{387} (2021),
  Paper No. 107826, 70.

\bibitem{genovese-ramos-gabrielpopescu}
Francesco Genovese and Julia Ramos~González, \emph{{A Derived
  Gabriel–Popescu Theorem for t-Structures via Derived Injectives}},
  International Mathematics Research Notices (2022), rnab367.

\bibitem{iyengar-krause-acyclicity}
Srikanth Iyengar and Henning Krause, \emph{Acyclicity versus total acyclicity
  for complexes over {N}oetherian rings}, Doc. Math. \textbf{11} (2006),
  207--240.

\bibitem{jorgensen-projective}
Peter J{\o}rgensen, \emph{The homotopy category of complexes of projective
  modules}, Adv. Math. \textbf{193} (2005), no.~1, 223--232.

\bibitem{jorgensen-cotstruct}
\bysame, \emph{Co-t-structures: the first decade}, Surveys in representation
  theory of algebras, Contemp. Math., vol. 716, Amer. Math. Soc., [Providence],
  RI, 2018, pp.~25--36.

\bibitem{keller-dgcat}
Bernhard Keller, \emph{On differential graded categories}, International
  {C}ongress of {M}athematicians. {V}ol. {II}, Eur. Math. Soc., Z\"{u}rich,
  2006, pp.~151--190.

\bibitem{keller-vossieck-aisles}
Bernhard Keller and Dieter Vossieck, \emph{Aisles in derived categories}, Bull.
  Soc. Math. Belg. S\'{e}r. A \textbf{40} (1988), no.~2, 239--253.

\bibitem{krause-stable}
Henning Krause, \emph{The stable derived category of a {N}oetherian scheme},
  Compos. Math. \textbf{141} (2005), no.~5, 1128--1162.

\bibitem{lunts-schnurer-smoothness-equivariant}
Valery~A. Lunts and Olaf~M. Schn\"{u}rer, \emph{Smoothness of equivariant
  derived categories}, Proc. Lond. Math. Soc. (3) \textbf{108} (2014), no.~5,
  1226--1276.

\bibitem{lurie-SAG}
Jacob Lurie, \emph{Spectral algebraic geometry}, February 2018 version,
  downloaded from https://www.math.ias.edu/\textasciitilde lurie/.

\bibitem{murfet-mockproj}
Daniel Murfet, \emph{The mock homotopy category of projectives and
  {G}rothendieck duality}, Ph.D. thesis, Australian National University, 2007,
  available at http://www.therisingsea.org/thesis.pdf.

\bibitem{neeman-triangulated}
Amnon Neeman, \emph{Triangulated categories}, Annals of Mathematics Studies,
  vol. 148, Princeton University Press, Princeton, NJ, 2001.

\bibitem{neeman-homotopyflat}
\bysame, \emph{The homotopy category of flat modules, and {G}rothendieck
  duality}, Invent. Math. \textbf{174} (2008), no.~2, 255--308.

\bibitem{neeman-homotopyinjectives}
\bysame, \emph{The homotopy category of injectives}, Algebra Number Theory
  \textbf{8} (2014), no.~2, 429--456.

\bibitem{nicolas-silting}
Pedro Nicol\'{a}s, Manuel Saor\'{\i}n, and Alexandra Zvonareva, \emph{Silting
  theory in triangulated categories with coproducts}, J. Pure Appl. Algebra
  \textbf{223} (2019), no.~6, 2273--2319.

\bibitem{paukszello-cotstruct}
David Pauksztello, \emph{Compact corigid objects in triangulated categories and
  co-{$t$}-structures}, Cent. Eur. J. Math. \textbf{6} (2008), no.~1, 25--42.

\bibitem{rizzardo-vdb-nonFM}
Alice Rizzardo, Michel Van~den Bergh, and Amnon Neeman, \emph{An example of a
  non-{F}ourier-{M}ukai functor between derived categories of coherent
  sheaves}, Invent. Math. \textbf{216} (2019), no.~3, 927--1004.

\bibitem{toen-dgcat-invmath}
Bertrand To\"{e}n, \emph{The homotopy theory of {$dg$}-categories and derived
  {M}orita theory}, Invent. Math. \textbf{167} (2007), no.~3, 615--667.

\bibitem{toen-lectures}
\bysame, \emph{Lectures on dg-categories}, Topics in algebraic and topological
  {$K$}-theory, Lecture Notes in Math., vol. 2008, Springer, Berlin, 2011,
  pp.~243--302.

\bibitem{stovicek-purity}
Jan \v{S}\v{t}ov\'{\i}\v{c}ek, \emph{On purity and applications to coderived
  and singularity categories}, 2014, preprint arXiv:1412.1615.

\bibitem{weibel-homological}
Charles~A. Weibel, \emph{An introduction to homological algebra}, Cambridge
  Studies in Advanced Mathematics, vol.~38, Cambridge University Press,
  Cambridge, 1994.

\end{thebibliography}

\end{document}